\documentclass{article}
\usepackage[utf8]{inputenc}
\usepackage[T1]{fontenc}

\usepackage{amsmath,amsbsy,amsthm,amssymb,mathrsfs,mathtools}
\usepackage{multirow,multicol,tabularx,booktabs}
\usepackage{graphicx,placeins,url,subfig,xcolor}
\usepackage[shortlabels]{enumitem}

\usepackage[top=2.5cm,bottom=2.5cm,right=2.5cm,left=2.5cm]{geometry}
\usepackage[colorlinks=true, citecolor=blue, linkcolor=blue, urlcolor=blue]{hyperref}

\newcommand{\mathitem}[2]{{\item #1 \leavevmode\vspace*{-\dimexpr\baselineskip+\abovedisplayskip+#2\relax}}}

\newtheorem{theorem}{Theorem}
\newtheorem{corollary}[theorem]{Corollary}
\newtheorem{proposition}[theorem]{Proposition}

\newtheorem{lemma}[theorem]{Lemma}
\newtheorem{remark}[theorem]{Remark}
\newtheorem{assumption}{Assumption}

\usepackage{dsfont}
\usepackage{tikz}

\renewcommand{\theassumption}{A\the\value{assumption}}

\newcommand{\E}{\mathbb{E}}
\newcommand{\Var}{\mathbb{V}\mathrm{ar}}

\newcommand{\Corr}{\mathrm{Corr}}
\newcommand{\mb}{\mathbf}
\newcommand{\bs}{\boldsymbol}
\newcommand{\sign}{\operatorname{sign}}

\newcommand{\liminftyg}{{
\lim_{b_1, b_2 \to +\infty}}}
\newcommand{\liminftyd}{{
\lim_{N \to +\infty}}}

\newcommand{\liminftygd}{{
\lim_{b_1, b_2, N \to +\infty}}}

\newcommand{\equalinlaw}{{\overset{\textnormal{law}}{=}}}

\def\Pb{{\mathbb P}}
\def\Rb{{\mathbb R}}

\def\Gc{{\mathcal G}}
\def\Kc{{\mathcal K}}

\def\T{{\mb{T}}}

\def\W{{\mb{W}}}
\def\X{{\mb{X}}}
\def\Z{{\mb{Z}}}

\def\w{{\mb{w}}}
\def\x{{\mb{x}}}
\def\z{{\mb{z}}}

\def\one{{\bf 1} }
\def\1{{\bf 1} }
\def\Matone{{ \mathbf{1} } }

\title{Fast estimation of Kendall's Tau and conditional Kendall's Tau matrices under structural assumptions}

\author{Rutger van der Spek\thanks{Department of Applied Mathematics, Delft University of Technology, Mekelweg 4, 2628 CD, Delft, Netherlands.}
\, and Alexis Derumigny\footnotemark[1]\hspace{0.3em}\footnote{Corresponding author, email: a.f.f.derumigny@tudelft.nl.}}

\date{\today}

\begin{document}

\maketitle


    
    
    
    
    
    \begin{abstract}
        Kendall's tau and conditional Kendall's tau matrices are multivariate (conditional) dependence measures between the components of a random vector. For large dimensions, available estimators are computationally expensive and can be improved by averaging.
        Under structural assumptions on the underlying Kendall's tau and conditional Kendall's tau matrices, we introduce new estimators that have a significantly reduced computational cost while keeping a similar error level.
        In the unconditional setting we assume that, up to reordering, the underlying Kendall's tau matrix is block-structured with constant values in each of the off-diagonal blocks. Consequences on the underlying correlation matrix are then discussed.
        The estimators take advantage of this block structure by averaging over (part of) the pairwise estimates in each of the off-diagonal blocks.
        Derived explicit variance expressions show their improved efficiency.
        In the conditional setting, the conditional Kendall's tau matrix is assumed to have a block structure, for some value of the conditioning variable.
        Conditional Kendall's tau matrix estimators are constructed similarly as in the unconditional case by averaging over (part of) the pairwise conditional Kendall's tau estimators.
        We establish their joint asymptotic normality, and show that the asymptotic variance is reduced compared to the naive estimators.
        Then, we perform a simulation study which displays the improved performance of both the unconditional and conditional estimators. 
        Finally, the estimators are used for estimating the value at risk of a large stock portfolio; backtesting illustrates the obtained improvements compared to the previous estimators.
    
    
\end{abstract}


\section{Introduction}

In dependence modeling, the main object of interest is the copula, which is a cumulative distribution function on $[0,1]^p$ with uniform margins, describing the links between elements of a $p$-dimensional random vector $\X$. However, the copula belongs to an infinite-dimensional space, and it is not easy to represent it as soon as $p$ is larger than $3$ or $4$. In such cases, finite-dimensional statistics becomes more useful to understand the dependence, the most well-known of them being Kendall's tau matrix.

\medskip

Kendall's tau between two random variables $X_i$ and $X_j$, denoted by $\tau_{i,j} = \tau(\Pb_{i,j})$, is defined as the probability of concordance between two independent replications from the distribution $\Pb_{i,j}$ of $(X_i, X_j)$ minus the probability of discordance.
The equality $\tau(\Pb_{i,j}) = 4\int C_{i,j} dC_{i,j} - 1$ relates Kendall's tau with the copula $C_{i,j}$ of $X_i$ and $X_j$; we refer to \cite{nelsen2007introduction} for an extensive introduction to Kendall's tau and copulas.

\medskip

When a covariate $\Z \in \Rb^d$ is available, we can extend the definition of Kendall's tau to the conditional setting.
Conditional Kendall's tau is then defined as
$\tau_{i,j|\Z=\z} := \tau(\Pb_{(i,j)|\Z=\z})$, where $\Pb_{(i,j)|\Z=\z}$ denotes the conditional law of $(X_i, X_j)$ given $\Z=\z$, for some $\z \in \Rb^d$.
In \cite{derumigny2019kernel,CKTVeraverbeke1,CKTVeraverbeke2}, smoothing-based estimators of conditional Kendall's tau are studied.
In \cite{derumigny2018classification}, it is shown that the estimation of conditional Kendall's tau can be written as a classification task; they proposed to use classification algorithms to estimate conditional Kendall's tau. In \cite{kendallregr}, a regression-type model is used to estimate conditional Kendall's tau in a parametric conditional framework. \cite{Ascorbebeitia2022testing} uses conditional Kendall's tau for hypothesis testing.

\medskip

For a random vector $\X$, we define Kendall's tau matrix by $\T := [\tau_{i,j}]_{1 \leq i,j \leq p}$, which contains all pairwise Kendall's taus; respectively, in the conditional framework, a natural counterpart is conditional Kendall's tau matrix, denoted by $\T_{|\Z=\z} := [\tau_{i,j|\Z=\z}]_{1 \leq i,j \leq p}$. Kendall's tau matrix is especially useful for elliptical graphical models and their generalizations, see \cite{barber2018rocket,liu2012transelliptical}. In \cite{lu2018post}, a time-varying graphical model is studied using an estimate of conditional Kendall's tau matrix. Kendall's tau matrix plays an important role since it allows robust estimation of the dependence, and can be used to fit an appropriate copula~\cite{VarKT}. In an elliptical distribution framework, it can also be used to estimate the Value at Risk of a portfolio, see~\cite{pimenova2012semi,VaRell}.

\medskip

Estimation of the $p \times p$ Kendall's tau matrix $\T$ becomes particularly challenging in the high-dimensional setting when $p$ is large. Simple use of the naive Kendall's tau matrix estimator of all pairwise sample Kendall's taus will result in noisy estimates with estimation errors piling up due to the estimates' individual imprecision \cite{factormodel}. Over the past two decades, various regularization strategies have been proposed to reduce the aggregation of estimation errors. Ultimately, these methods all make certain assumptions on the underlying dependence structure, hereby reducing the number of free parameters to estimate.

\medskip

In many instances, sparsity of the target matrix is assumed. For such settings, various (combinations of) thresholding and shrinkage methods have been proposed, see for example~\cite{thresholding,gray2018shrinkage,generalisedthresholding}.
However, such assumptions are certainly not appropriate for the modeling of most financial data, e.g. market risk is reflected in all share prices and therefore their returns are certainly correlated.  To this end, factor models are usually imposed, where the correlations depend on a number of common factors, which may or may not be latent, see \cite{factormodel,latentfactormodel}.

\medskip

In \cite{perreault2020structures,perreault}, an alternative approach to estimating large Kendall's tau matrices was introduced. 
They studied a model in which it is assumed that the set of variables could be partitioned into smaller clusters with exchangeable dependence. As such, after reordering of the variables by cluster, the corresponding Kendall's tau matrix is block-structured with constant values within each block. Following naturally is an improved estimation by averaging all pairwise sample Kendall's taus within each of the blocks.
Additionally, they have proposed a robust algorithm identifying such structures (see also~\cite{perreault2020hypothesis} for testing for the presence of such a structure).

\medskip

In this article, we study a similar framework as in \cite{perreault}, where we relax the partial exchangeability assumption: we only assume that off-diagonal blocks of Kendall's tau matrix are constant. One of the drawbacks of the estimator studied in \cite{perreault} is its computational cost, which is close to the one of the naive Kendall's tau matrix estimator: the number of pairwise sample Kendall's taus that are to be computed scales quadratically with the dimension~$p$.

\medskip

Naturally, the idea of averaging among several Kendall's taus can be applied to part of the blocks, which allows for faster computations.
As such, we propose several estimators that average among part of the Kendall's tau per off-diagonal block and study their efficiencies and computational costs.
For every off-diagonal block, we will consider averaging over elements in the same row, averaging over elements on the diagonal and averaging over a number of randomly selected elements.
We will be referring to these estimators as the \textit{row}, \textit{diagonal} and \textit{random} estimators; the estimator that averages over all elements is referred to as the \textit{block} estimator.

\medskip

We then extend this model to the conditional setup: conditional Kendall's taus are depending on $\z$ and are assumed to be clustered such that for all $\z \in \Rb^d$, the Kendall's tau conditionally on $\Z = \z$ between variables of different groups is only depending on group numbers and on the value of $\z$.
In view of applications to finance, the conditional version of our structural assumption could, for example, be seen as assuming that the correlations between European stocks of two different groups are equal and react equally to changes of some other American stock or portfolio.
Furthermore, in \cite{DependentCorr3,DependentCorr1,DependentCorr2}, it was shown that stock returns actually exhibit higher correlations during market declines than during market upturns, and moreover that the same applies to exchange rates in \cite{intrCondCopulas1}. In such a model, it is also important to limit computation times and study improved estimators that can take advantage of the block structure of the Kendall's tau matrix.

\medskip

In this framework, we adopt nonparametric estimates of the conditional Kendall's tau based on kernel smoothing. Based on these nonparametric estimates we introduce conditional versions of the averaging estimators and study their asymptotic behavior as the sample size $n$ tends to infinity.
It is worth noting that conditional estimates of Kendall's tau using kernel smoothing carry significantly more computational cost than their unconditional counterparts, especially when the covariate's dimension $d$ is large. Therefore, faster computations of conditional Kendall's tau matrices will be of particular use in the conditional, nonparametric setup. 

\medskip

The rest of this article is structured as follows. 
In Section~\ref{sec:KT}, we present the unconditional framework, and detail a few consequences on the correlation matrix. Then we construct the different estimators in this framework, and derive variance expressions.
Similarly, Section~\ref{sec:CKT} is devoted to the improved estimation of the conditional Kendall's tau matrix, where we propose averaged conditional estimators and we derive the estimators' joint asymptotic normality.
In Section~\ref{sec:SS} we perform a simulation study in order to support the theoretical findings.
Finally, in Section~\ref{sec:RD}, we examine a possible application to study the behavior of the estimators in real data conditions.
The estimators are used for the robust inference of the covariance matrix to estimate the value at risk of a large stock portfolio.
Proofs are postponed to the Appendix.

\medskip

\noindent
{\it Notations.} We denote by $\Matone$ be the vector and the matrix with all entries equal to $1$, where the dimensions can be inferred from the context. For a matrix $M$ of size $p \times p$, and a set of indices $J \subset \{1, \dots, p\}^2$, we denote by $[M]_J$ the submatrix $(M_j)_{j \in J}$.

\section{Fast estimation of Kendall's tau matrix}
\label{sec:KT}

\subsection{The Structural Assumption}
\label{sec:SA}

Let $n \geq 2$ and assume that we observe $n$ i.i.d. replications $\X_i = (X_{i,1}, \ldots, X_{i,p})$ of a random vector $\X=(X_1, \dots, X_p) \in \Rb^p$, for $i=1, \dots, n$. Moreover, assume that the Kendall's tau matrix $\T := [\tau_{i,j}]_{1 \leq i,j \leq p}$ of $\X$ satisfies the following structural assumption.

\begin{assumption}[Structural Assumption]
    \label{as:structural} 
    There exists $K>0$, a partition $\Gc = \{\Gc_1, \dots, \Gc_K\}$ of $\{ 1, \dots, p\}$, a set $J \subset \{1, \dots, K\}^2$ and some constants $(\tau_{k_1,k_2})_{(k_1, k_2) \in J} \in [-1, 1]^J$ such that for all $(k_1, k_2) \in J$,
    \begin{align*}
        \left[\T\right]_{\Gc_{k_1} \times \Gc_{k_2}}
        = \tau_{k_1, k_2} \cdot \one.
    \end{align*}
\end{assumption}

\begin{figure}[bt]
  \centering
  \subfloat[unclustered]{\includegraphics[width=0.49\textwidth]{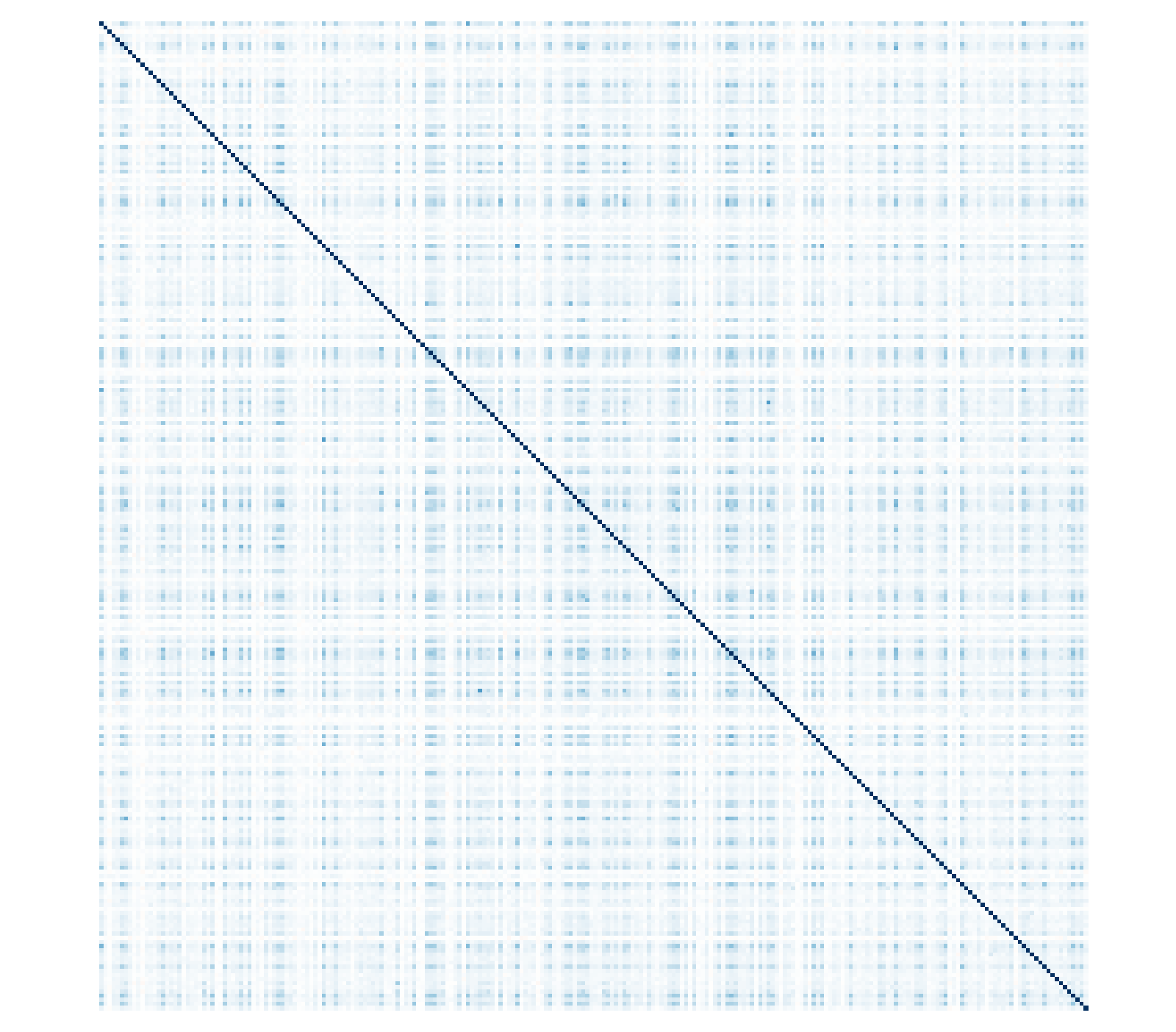}\label{fig:f1}}
  \hfill
  \subfloat[clustered]{\includegraphics[width=0.49\textwidth]{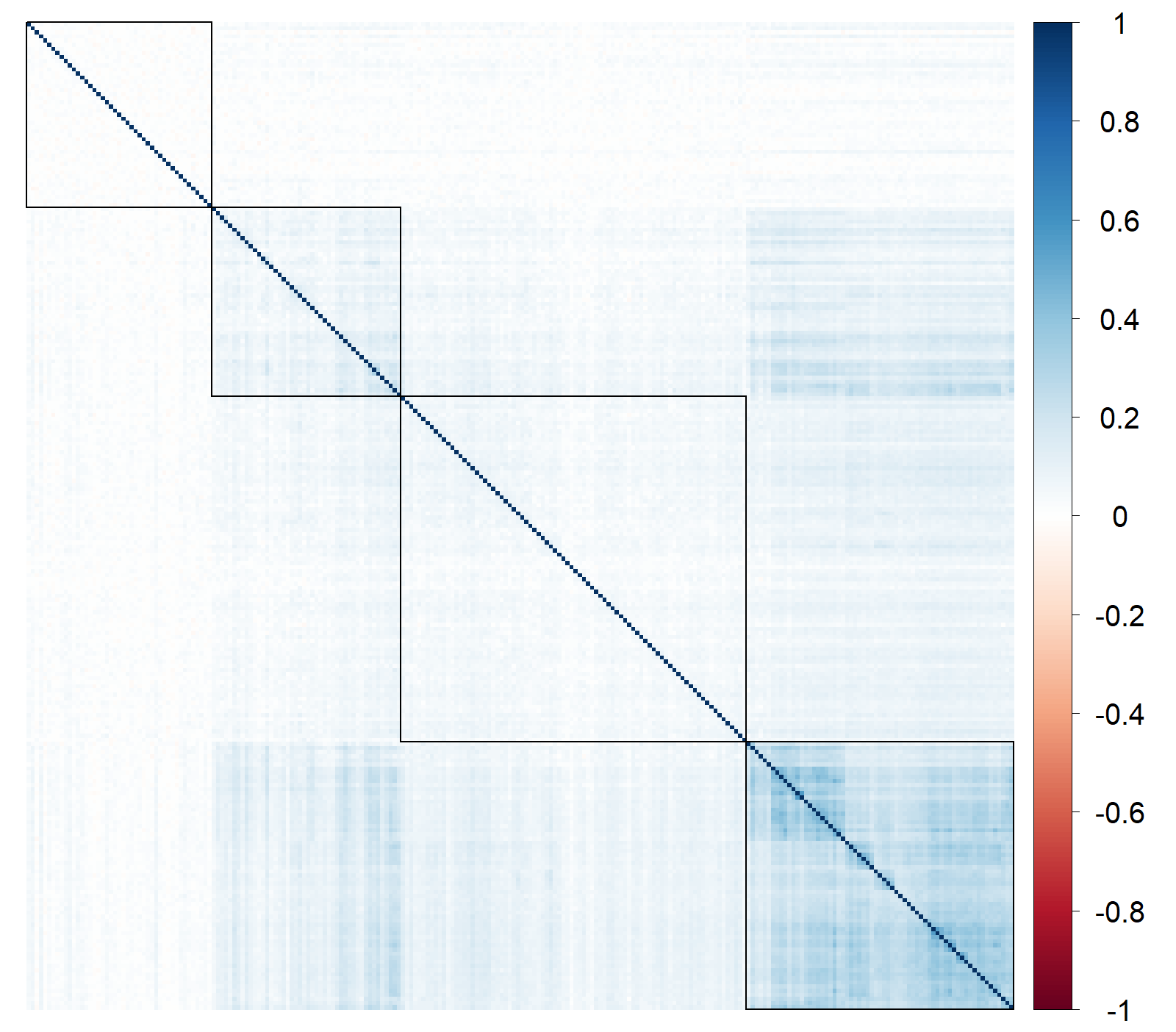}\label{fig:f2}}
  \caption{Heatmap plots of the sample Kendall's tau matrix computed on the daily log returns from 01 January 2007 until 14 January 2022 of all 240 portfolio stocks (whose list is available in~\ref{sec:stocksLists}).}   \label{fig:clustering}
\end{figure}

Note that after reordering of the variables by group, the corresponding Kendall's tau matrix is block-structured with constant values in some of the off-diagonal blocks.
The interest in investigating this structural assumption originates from applications in stock return modeling.
In this context, the clustering of the variables could for instance be considered as grouping companies by sector or economy.
It then seems at least intuitive to assume that companies from different groups have correlations that depend only on the groups they are in, without making any assumptions on the correlations between companies from the same group.
We will therefore call $\tau_{k_1, k_2}$
\textbf{the intergroup Kendall's tau} between groups $\Gc_{k_1}$ and $\Gc_{k_2}$.

\medskip

This can be clearly seen on Figure~\ref{fig:clustering}: in each of the off-diagonal blocks the Kendall's tau is mostly homogeneous, but significant differences can be seen in the fourth diagonal block. Indeed, it gathers companies whose link to other groups is constant, but with different relationships inside the group. This may be explained by the presence of subgroups inside this fourth group, even if the relationship with variables from other groups do not seem to be related to these subgroup structures. 

\medskip

Obviously, the structural assumption is satisfied for any set of variables by using only groups of length $1$. Therefore, assuming larger groups will make the assumption more constraining.
Indeed, in this framework, Kendall's tau matrix depends on
$$\frac{1}{2} K (K-1) + \frac{1}{2} \sum_{k=1}^K \lvert \Gc_k \rvert (\lvert \Gc_k\rvert -1)$$
free parameters. For a dimension of $100$, assuming we can split into $K=10$ groups of equal size, this translates to a reduction by factor of 10 of the number of free parameters to estimate (from $4950$ to $495$).
Such a reduction suggests that the use of appropriate estimators can lead to significant estimation improvements.
We will define estimators of the Kendall's tau matrix $\T$ under Assumption \ref{as:structural} for some known partitions $\Gc$ of $ \{1,\ldots, p\}$. Note that such partition can also be inferred from the data, see \cite{perreault,perreault2020hypothesis}
and the thesis~\cite{perreault2020structures}.
Even if Assumption \ref{as:structural} is not satisfied, the estimators that we will propose can still be of interest, for example to do linear shrinkage.

\medskip

Note that although our results are only interesting when $k_1 \neq k_2$, they also hold if a pair of form $(k_1, k_1)$ belongs to $J$. Indeed in this case we must have $\tau_{k_1, k_1} = 1$ and then all pairs of observations in the block will be concordant, so all our estimators will be equal to $1$.

\medskip

In this article, we are interested in the estimation of the intergroup Kendall's tau $\tau_{k_1, k_2}$ for some pair $(k_1, k_2)$ belonging to $J$, i.e. such that the block $[\T]_{\Gc_{k_1} \times \Gc_{k_2}}$ is a constant block.
This means that the random vector $(\X_{\Gc_{k_1}}^\top, \X_{\Gc_{k_2}}^\top)^\top$ satisfy the following assumption.
\begin{assumption}[Simplified Structural Assumption]
    \label{as:simp_structural} 
    Let $\X$ be a $p$-dimensional random vector of interest.
    Kendall's tau matrix of the random vector $\X$ can be written in the block form
    \begin{align*}
        \T = \begin{pmatrix}
            \cdot & \tau \1 \\
            \tau \1 & \cdot
        \end{pmatrix}
    \end{align*}
    where $\tau \1$ represent a block filled with the value $\tau \in [-1, 1]$ and the symbol $\cdot$ represent any matrices of respective sizes $b_1 \times b_1$ and $b_2 \times b_2$ for some $b_1 \in \{1, \dots, p\}$ and $b_2 := p - b_1$.
\end{assumption}
Under Assumption~\ref{as:simp_structural},
we call $\tau$ \textbf{the intergroup Kendall's tau}, as a particular case of the previous framework.
As discussed above, Assumption~\ref{as:simp_structural} may seem more constraining than the previous Assumption~\ref{as:structural}, but both assumptions are actually equivalent insofar as we are only interested in the estimation of each interblock Kendall's tau $\tau_{k_1, k_2}$.
Therefore and in order to simplify the notations, we will choose to assume Assumption~\ref{as:simp_structural}.
Similarly as the more general Assumption~\ref{as:structural}, Assumption~\ref{as:simp_structural} can be tested using the framework developed in~\cite{perreault2020hypothesis}.

\medskip

Note that \cite{perreault} proposed a similar model with a more restrictive version of Assumption~\ref{as:structural}, the Partial Exchangeability Assumption, by assuming that the variables could be partitioned into $K$ clusters with exchangeable dependence.

\begin{assumption}[Partial Exchangeability Assumption]
\label{as:PEA}
    For $j \in\{1, \ldots, p\}$, let $U_{j}=F_{j}\left(X_{j}\right)$, where $F_j$ is the cumulative distribution function of $X_j$, and $C$ be the copula of $\X$.
    For any partition $\Gc=\left\{\Gc_{1}, \ldots, G_{K}\right\}$ of $\{1, \ldots, p\}$, let $\pi(G)$ be the set of permutations $\pi$ of $\{1, \ldots, p\}$ such that for all $j \in\{1, \ldots, p\}$ and all $k \in\{1, \ldots, K\}$,  $\pi(j) \in \Gc_{k}$ if and only if $j \in \Gc_{k}$.
    
    A partition $\Gc=\left\{\Gc_{1}, \ldots, G_{K}\right\}$ satisfies the Partial Exchangeability Assumption if for any $u_{1}, \ldots, u_{p} \in[0,1]$ and any permutation $\pi \in \pi(G)$, one has
    $$
    C\left(u_{1}, \ldots, u_{p}\right)
    = C\left(u_{\pi(1)}, \ldots, u_{\pi(p)}\right),
    $$
    or, equivalently, $\left(U_{1}, \ldots, U_{p}\right) \equalinlaw \left(U_{\pi(1)}, \ldots, U_{\pi(p)}\right)$.
\end{assumption}
Note that the Partial Exchangeability Assumption imposes restrictions on the underlying copula, whereas Assumption \ref{as:structural} only does so on the underlying Kendall's tau matrix, making it a lot less restrictive.
Further, under Assumption \ref{as:PEA}, Kendall's tau matrix is fully block-structured including constant diagonal blocks as well, after reordering of the variables.
In contrast to \cite{perreault}, we are more interested in a model where we do not consider partial exchangeability nor constant interdependence of marginal variables within the same cluster. 
Particularly in view of the aforementioned application of stock returns, the Partial Exchangeability Assumption seems quite restrictive and a model without partial exchangeability in which companies from the same cluster have different mutual dependence is more plausible (see Figure~\ref{fig:clustering}).
For these reasons, we opt for a more flexible variant of this model.

\subsection{Consequences of the block structure on the correlation matrix}
\label{sec:consequences_corMatrix}

As explained in \cite[Chapter 3]{kurowicka2006uncertainty}, if $\X$ follows a multivariate Gaussian distribution with exchangeable dependence and correlation $\rho \in (-1,1)$, then $\Corr(\X)$ is positive definite if and only if $\rho > -1/(p-1)$. In terms of Kendall's tau, this constraint translates as $\tau > - (2/\pi) \mathrm{Arcsin}(1/(p-1))$.
Assumption~\ref{as:structural} is weaker than exchangeable dependence, and even in the Gaussian setting with two groups of variables, it allows for \textit{arbitrary} negative correlation in the off-diagonal block.
The results presented in this section are related to the results of~\cite{cadima2010eigenstructure} who study the eigenstructure of such block structured correlation matrices, but without giving precise constraint on the allowed values of the correlation.
\cite{mcneil2022attainability} discuss attainability of Kendall's tau matrices in a general framework, i.e. without discussing the block structure specifically.

\begin{proposition}
    Let $b_1, b_2 \geq 2$ be integers. Let $\rho_1, \rho_2, \rho_3 \in (-1, 1)$, and define the block matrix $M \in \Rb^{(b_1 + b_2)^2}$ with blocks of size $b_1$ and $b_2$ by
    \begin{align*}
        M := \left( \begin{array}{cc}
            I + \rho_1 \widetilde{I} & \rho_3 \1 \\
            \rho_3 \1 & I + \rho_2\widetilde{I}
        \end{array} \right),
    \end{align*}
    where $I$ is the identity matrix and $\widetilde{I} := \1 - I$ is the matrix with $1$ at each off-diagonal entry and $0$ on the diagonal.
    Then $M$ is positive definite if and only if
    \begin{align}
        \big(b_2 b_1 - b_1 - b_2 + 1\big) \rho_1 \rho_2
        - b_1 b_2 \rho_3^{2}
        + (b_2 - 1) \rho_2
        + (b_1 - 1) \rho_1
        + 1 > 0.
        \label{eq:constraint_rho12}
    \end{align}
    Furthermore, this inequality is satisfied as soon as
    \begin{align*}
        \left\{ \begin{array}{l}
            \rho_{1}
            > \dfrac{b_{1} b_{2} \rho_{3}^{2} - (b_{2} - 1) \rho_{2} - 1}{
            \left(b_{1} b_{2}-b_{1}-b_{2}+1\right) \rho_{2} + b_{1} - 1}, \\
            \rho_{2}
            > - \dfrac{b_{1}-1}{b_{1} b_{2} - b_{1} - b_{2} + 1}.
        \end{array} \right.
    \end{align*}
    \label{prop:correlation_matrix}
\end{proposition}

This result is proved in Appendix~\ref{proof:prop:correlation_matrix}. 
We can remark that the constraint \eqref{eq:constraint_rho12} is always satisfied as soon as
\begin{equation}
    \rho_1 \rho_2 \geq \rho_3^2,
    \label{eq:asympt_constr_rho}
\end{equation}
i.e. the absolute value of $\rho_3$ has to be smaller than the geometric mean of $\rho_1$ and $\rho_2$.
Furthermore, in the high-dimensional setting where $b_1, b_2 \to + \infty$ and $\rho_1, \rho_2$ are fixed and positive, \eqref{eq:constraint_rho12} actually becomes equivalent to the simplified constraint \eqref{eq:asympt_constr_rho}.
Note that \eqref{eq:constraint_rho12} allows for situations where $\rho_3$ is arbitrarily close to $-1$, for any choice of block sizes. This is possible, for example, by setting $\rho_1 = \rho_2 = |\rho_3|$. Concretely, if all variables of one group have a high correlation with all variables of the second group, then in each group the intragroup correlation should be high.
Such a result translates directly for Kendall's tau matrix by using the relationship $\tau = (2/\pi) \mathrm{Arcsin}(\rho)$, allowing for Kendall's tau matrix with arbitrary entries in the off-diagonal blocks.

\medskip

Rather surprisingly, as soon as the group sizes $b_1$ and $b_2$ are large enough, they don't appear anymore in the constraint~\eqref{eq:asympt_constr_rho}. This phenomenon is in fact typical of block-structured matrices and will appear again in the performance of our estimators in the next sections.
We now give a lower bound for the intergroup Kendall's tau in the setting where $K$ groups are present, with equal intergroup Kendall's tau. It is proved in Appendix~\ref{proof:prop:correlation_matrix}.

\begin{proposition}
    Let $K \geq 2$ and let $b_1, \dots, b_K$ be positive integers. Let $\rho \in (-1,1)$, $\Sigma_1, \dots, \Sigma_K$ be $K$ correlation matrices of size $b_1, \dots, b_K$ respectively and let $M \in \Rb^{(b_1 + b_2 + \dots + b_K)^2}$ be the block matrix defined by
    \begin{align*}
        M := \left( 
        \begin{array}{cccc}
            \Sigma_1 & \rho \1 & \cdots & \rho \1 \\
            \rho \1 & \Sigma_2 & \ddots & \vdots \\
            \vdots  & \ddots   & \ddots & \rho \1  \\
            \rho \1 & \cdots   & \rho \1 & \Sigma_K \\
        \end{array}
        \right)
    \end{align*}
    Then
    \begin{align*}
        \inf \big\{ \rho \in (-1, 1): \exists \Sigma_1, \dots, \Sigma_K
        \text{ such that } M \text{ is a correlation matrix } \big\}
        = \frac{-1}{K - 1}.
    \end{align*}
    \label{proof:lowerbound_rho_K}
\end{proposition}

Interestingly, this bound does not depend on the sizes of the blocks $b_1, \dots, b_K$. This shows that, for such a matrix, the structure of each block does not have a strong influence on the possible choice of intergroup Kendall's tau. The constraint $\rho \geq - 1 /(p - 1)$ for exchangeable correlation matrices becomes in this framework $\rho \geq - 1 / (K - 1)$ for exchangeable intergroup dependence, suggesting that the number of blocks becomes the relevant dimension for this problem (instead of the number of variables).
Nevertheless, the knowledge of the intragroup correlations matrices $\Sigma_1, \dots, \Sigma_K$ will still constrain the range of possible values of $\rho$, as this was the case in Proposition~\ref{prop:correlation_matrix} for the particular case $K=2$ and exchangeable dependence in each block.

\subsection{Construction of estimators}
\label{sec:est}

First, note that we can naturally rely on the usual estimator of Kendall's tau between $X_{j_1}$ and $X_{j_2}$ defined by
\begin{equation}
    \widehat{\tau}_{j_1, j_2}
    := \frac{2}{n(n-1)} \sum_{i_1 <i_2} \sign
    \big( ( X_{i_1, j_1} - X_{i_2, j_1} )
    ( X_{i_1, j_2} - X_{i_2, j_2} ) \big), \label{eq:sKT1s}
\end{equation}
for any $1 \leq j_1, j_2 \leq p$.
We denote the corresponding Kendall's tau matrix estimator by $\widehat{\T} = \left[ \widehat{\tau}_{j_1,j_2} \right]_{1 \leq j_1, j_2 \leq p}$, which serves as a first step estimator for obtaining a better estimator of the Kendall's tau matrix. 
This estimated Kendall's tau matrix $\widehat{\T}$ does not make any use of the underlying structure and will therefore be a rather naive tool in practice.
As discussed in the previous section,
we will introduce several estimators in the setting of Assumption~\ref{as:simp_structural}, with straightforward generalizations under Assumption~\ref{as:structural}.
More precisely, for every estimator $\widehat{\tau}^{\ref{as:simp_structural}}$ of $\tau$ under Assumption~\ref{as:simp_structural}, we define a corresponding matrix estimator $\widehat{\T}^{\ref{as:structural}}$ under Assumption~\ref{as:structural} by 
\begin{align}
    \widehat{\T}^{\ref{as:structural}}
    := \left[ \widehat{T}^{\ref{as:structural}}_{j_1,j_2} 
    \right]_{1 \leq j_1, j_2 \leq p} = 
    \begin{cases}
      \widehat{\tau}^{\ref{as:simp_structural}}
      (\Gc_{k_1}, \Gc_{k_2}),
      & \text{if}\ j_1 \in \Gc_{k_1}, \ j_2 \in \Gc_{k_2}
      \text{ and } (k_1, k_2) \in J, \\
      \widehat{\tau}_{j_1, j_2}, & \text{else.}
    \end{cases}
    \label{eq:def:T_hat_A1}
\end{align}
Remember that $J$ is the set of pairs $(k_1, k_2)$ of block indices such that Kendall's tau $\tau_{j_1, j_2}$ does not depend on $j_1 \in \Gc_{k_1}, \ j_2 \in \Gc_{k_2}$.
Since we assume that the Kendall's taus in (some of) the off-diagonal blocks are equal, the idea of averaging the pairwise sample Kendall's tau follows naturally. 
Let us introduce the block estimator $\widehat{\tau}^{B}$ that averages all sample Kendall's taus within each of the off-diagonal blocks. Formally, we have
\begin{equation*}
    \widehat{\tau}^{B}
    := \frac{1}{b_1 b_2}
    \sum_{j_1 = 1}^{b_1} \sum_{j_2 = b_1 + 1}^p
    \widehat{\tau}_{j_1, j_2}.
\end{equation*}
We define $\widehat{\T}^B$ the corresponding estimator under Assumption~\ref{as:structural} by Equation~\eqref{eq:def:T_hat_A1}.

\medskip

Under the Partial Exchangeability Assumption, \cite{perreault} showed that the estimator $\widehat{\T}^B$ is asymptotically normal and optimal under the Mahalanobis distance.
However, in terms of computational efficiency, the block estimator $\widehat{\T}^B$ does not show any improvement over the usual estimator $\widehat{\T}$, as both estimators require the computation of the usual Kendall's tau between all pairs of variables anyway. 

\medskip

To reduce the computation time, we propose not averaging over all Kendall's taus in the block but only over some of them. This would lead to computationally cheaper estimates. Naturally, the question arises over which elements then to average over. For this purpose, we introduce several estimators that average over different subsets of elements within each of the off-diagonal blocks.

\medskip

We introduce two estimators that each average $N \in \{1, \dots, b_1 \vee b_2 \}$ pairs in the off-diagonal block under Assumption~\ref{as:simp_structural}, so that we can compare estimators that either average pairs in the same row/column, or pairs on the diagonal. Without loss of generality, since we can switch both blocks of variables, we assume that $b_1 \leq b_2$ and we will average over the row.
For averaging pairs on the diagonal, it is moreover required that $N \leq b_1 \wedge b_2$.
The number of Kendall's tau estimates is then reduced to scaling linearly with group size, which is a significant improvement over the previous quadratic growth.
We set
\begin{align*}
    \widehat{\tau}^{R}
    := \frac{1}{N} \sum_{j=1}^N \widehat{\tau}_{1, b_1 + j},
    \text{ and }
    \widehat{\tau}^{D}
    := \frac{1}{N} \sum_{j=1}^N \widehat{\tau}_{j, b_1 + j},
\end{align*}
Then, the ``row-based'' Kendall's tau matrix estimator $\widehat{\T}^{R}$ and the ``diagonal-based'' Kendall's tau matrix estimator $\widehat{\T}^{D}$ are respectively defined by Equation~\eqref{eq:def:T_hat_A1} for the choices
$\widehat{\tau}^{\ref{as:simp_structural}} = \widehat{\tau}^{R}$
and $\widehat{\tau}^{\ref{as:simp_structural}} = \widehat{\tau}^{D}$.
As such, for each of the off-diagonal blocks $\widehat{\T}^{R}$ averages only the pairs on the first line along the largest side, whereas $\widehat{\T}^{D}$ averages only the pairs along the first diagonal.

\medskip

Lastly, we introduce the estimator that randomly selects pairs to average over per block. We denote the (deterministic) number of averaged pairs per block by $N  \in \{1, \dots, b_1 \times b_2 \}$ and the corresponding estimator by $\T^U$.
The pairs are selected with uniform probability and without replacement. We define
\begin{align*}
    \widehat{\tau}^{U}
    := \frac{1}{N} \sum_{j_1 = 1}^{b_1} \sum_{j_2 = 1}^{b_2}
    W_{j_1,j_2} \widehat{\tau}_{j_1,j_2}, 
\end{align*}
where $\W$ is a $b_1 \times b_2$ matrix of random weights that selects $N$ pairs per off-diagonal block with uniform probability and without replacement. $W_{j_1,j_2} = 1$ corresponds to selecting pair $(X_{j_1},X_{j_2})$ and $W_{j_1,j_2} = 0$ corresponds to passing over it. The corresponding matrix estimator is then denoted by $\widehat{\T}^{U}$ with $W$ defined as selecting $N$ pairs in each of the averaged blocks.

\subsection{Comparison of their variances}
\label{sec:varKT}

Before we proceed with the main theoretical results on the estimators' variances, let us introduce some auxiliary notations. 
For every $j_1, j_2 \in \{1, \ldots, p\}$ we set 
$P_{j_1, j_2} := \mathbb{P} \big(
\left(X_{1, j_1} - X_{2, j_1}\right)
\left(X_{1, j_2} - X_{2, j_2}\right) > 0 \big).$
The quantity $P_{j_1, j_2}$ is equal to the probability of concordance of the variables $X_{j_1}$ and $X_{j_2}$ and thus $\tau_{j_1, j_2} = 2P_{j_1, j_2} - 1$.
As such, the structural assumption \ref{as:simp_structural} ensures that $P_{j_1, j_2}$ is independent of the choice of pair whenever $j_1$ and $j_2$ are not in the same block.
Alternatively we can write $P_{j_1, j_2}$ in terms of the copula $C_{j_1, j_2}$ of $(X_{j_1}, X_{j_2})$ by
$P_{j_1, j_2}
= 2 \int_{[0,1]^2} C_{j_1,j_2}(u_1,u_2) \mathrm{d} C_{j_1,j_2}(u_1,u_2).$
By extension, we define
\begin{align*}
    P_{j_1, j_2, j_3, j_4} &:=
    \mathbb{P}\left(
    \left(X_{1, j_1} - X_{2, j_1}\right)
    \left(X_{1, j_2} - X_{2, j_2}\right) > 0,
    \left(X_{1, j_3} - X_{2, j_3}\right)
    \left(X_{1, j_4} - X_{2, j_4}\right) > 0 \right), \\
    Q_{j_1, j_2, j_3, j_4} &:= \mathbb{P}\left(
    \left(X_{1, j_1} - X_{2, j_1}\right)
    \left(X_{1, j_2} - X_{2, j_2}\right) > 0, 
    \left(X_{1, j_3} - X_{3, j_3}\right)
    \left(X_{1, j_4} - X_{3, j_4}\right) > 0 \right),  \\
    S_{j_1, j_2, j_3, j_4} &:= \mathbb{P}\left(
    \left(X_{1, j_1} - X_{2, j_1}\right)
    \left(X_{1, j_2} - X_{2, j_2}\right) > 0, 
    \left(X_{4, j_3} - X_{3, j_3}\right)
    \left(X_{4, j_4} - X_{3, j_4}\right) > 0 \right)
    = P_{j_1, j_2} P_{j_3, j_4},
\end{align*}
for every $j_1, j_2, j_3, j_4 \in \{1, \ldots, p\}$.
Note that both $P$ and $Q$ quantities can be understood as some kind of ``cross-concordance measures'', but there is an important difference between them:
for the $Q$ measure of cross-concordance there is a third copy $\X_{3, 1:p}$ needed.
When $j_1 = j_3$ and $j_2 = j_4$, we obtain
$P_{j_1, j_2, j_3, j_4} = P_{j_1, j_2}$.
$Q$-type measures of cross-concordance naturally appears in the asymptotic variance of the usual estimator of Kendall's tau through the particular case 
\begin{align}
    Q_{j_1, j_2}
    := Q_{j_1, j_2, j_1, j_2}
    &= \mathbb{P}\left(
    \left(X_{1, j_1} - X_{2, j_1}\right)
    \left(X_{1, j_2} - X_{2, j_2}\right) > 0,
    \left(X_{1, j_1} - X_{3, j_1}\right)
    \left(X_{1, j_2} - X_{3, j_2}\right) > 0 \right) 
    \nonumber \\
    &= \int_{[0,1]^2}
    (C_{j_1,j_2}(u_1,u_2) + \bar{C}_{j_1,j_2}(u_1,u_2))^2 
    \mathrm{d}C_{j_1,j_2}(u_1,u_2),
    \label{eq:Q_copulaExpression}
\end{align}
where $\bar{C}$ denotes the survival function of a copula $C$.
For $S$, a fourth independent copy is needed, but by independence it reduces to the product $P_{j_1, j_2} P_{j_3, j_4}$.
For completeness, this expression is derived in Appendix~\ref{proof:eq:Q_copulaExpression}.
Note that this equality was already given in \cite[Equation (8)]{VarKT}.
Note that both $P$ and $Q$ can be written as 8-dimensional integrals involving the copula of $(X_{j_1}, X_{j_2}, X_{j_3}, X_{j_4})$.
As a consequence, they are functions of the joint law $\Pb_{(j_1,j_2,j_3,j_4)}$ of the random vector $\X_{(j_1,j_2,j_3,j_4)}$. 
Note further that even under Assumption \ref{as:structural}, they are both depending on the choice of pairs within the considered off-diagonal block. Obviously, this is not the case if we assume the stronger Assumption~\ref{as:PEA}.

\medskip


To give explicit expression for our estimators, we need to average $P$, $Q$, and $S$ quantities. We need to separate these averages depending on the number of common variables in the cross-concordance terms. We define
\begin{align*}
    P_{B,2} &:= \frac{1}{b_1 b_2}
    \sum_{j_1 = 1}^{b_1} \sum_{j_2 = b_1 + 1}^{p} P_{j_1, j_2},
    \\
    P_{B,1} &:= \frac{1}{b_1 (b_1-1) b_2 + b_2 (b_2 - 1) b_1 }
    \bigg(
    \sum_{j_1 = 1}^{b_1} \sum_{j_2 = b_1 + 1}^{p}
    \sum_{j_3 = 1, \, j_3 \neq j_1}^{b_1} P_{j_1, j_2, j_3, j_2}
    + \sum_{j_1 = 1}^{b_1} \sum_{j_2 = b_1 + 1}^{p}
    \sum_{j_4 = b_1 + 1, j_4 \neq j_2}^{p}
    P_{j_1, j_2, j_1, j_4} \bigg),
    \\
    P_{B,0} &:= \frac{1}{b_1 (b_1-1) b_2 (b_2 - 1)}
    \sum_{j_1 = 1}^{b_1} \sum_{j_2 = b_1 + 1}^{p}
    \sum_{j_3 = 1, \, j_3 \neq j_1}^{b_1}
    \sum_{j_4 = b_1 + 1, j_4 \neq j_2}^{p}
    P_{j_1, j_2, j_3, j_4},
\end{align*}
and similarly, we define $P_{\alpha,\beta}$, $Q_{\alpha,\beta}$, $S_{\alpha,\beta}$ for $\alpha \in \{B, R, D\}$ and $\beta \in \{0, 1, 2\}$.
In the latter expression, $\alpha$ denotes the type of averaging (respectively, over the Block, over the Row and over the Diagonal) and $\beta$ denotes the number of common variables,
i.e. the size of $\{j_1, j_2\} \cap \{j_3, j_4\}$.
This means that for $\beta = 2$, $P_{B,2}$ is the average of the $P_{j_1, j_2, j_3, j_4}$ over the set of $(j_1, j_2, j_3, j_4)$ such that $j_1 = j_3$ and $j_2 = j_4$ (2 common variables, since both variables are the same).
$P_{B,1}$ is the average of the $P_{j_1, j_2, j_3, j_4}$ over the set of $(j_1, j_2, j_3, j_4)$ such that $j_1 = j_3$ or $j_2 = j_4$ (1 common variable, since one is the same and one is different).
Finally, $P_{B,0}$ is the average of the $P_{j_1, j_2, j_3, j_4}$ over the set of $(j_1, j_2, j_3, j_4)$ such that $j_1 \neq j_3$ and $j_2 \neq j_4$ (0 common variables, all variables are different).
To deal with the random estimator, we need to define the corresponding weighted quantities for a $p \times p$ matrix $\w$:
\begin{align*}
    P_{\w,2} &:= \frac{1}{b_1 b_2}
    \sum_{j_1 = 1}^{b_1} \sum_{j_2 = b_1 + 1}^{p}
    w_{j_1, j_2} P_{j_1, j_2},
    \displaybreak[0] \\
    P_{\w,1} &:= \frac{1}{b_1 (b_1-1) b_2 + b_2 (b_2 - 1) b_1 }
    \bigg(
    \sum_{j_1 = 1}^{b_1} \sum_{j_2 = b_1 + 1}^{p}
    \sum_{j_3 = 1, \, j_3 \neq j_1}^{b_1} w_{j_1, j_2} 
    w_{j_3, j_2}  P_{j_1, j_2, j_3, j_2} \\
    &\hspace{5cm} + \sum_{j_1 = 1}^{b_1} \sum_{j_2 = b_1 + 1}^{p}
    \sum_{j_4 = b_1 + 1, j_4 \neq j_2}^{p} w_{j_1, j_2} 
    w_{j_1, j_4}
    P_{j_1, j_2, j_1, j_4} \bigg),
    \displaybreak[0] \\
    P_{\w,0} &:= \frac{1}{b_1 (b_1-1) b_2 (b_2 - 1)}
    \sum_{j_1 = 1}^{b_1} \sum_{j_2 = b_1 + 1}^{p}
    \sum_{j_3 = 1, \, j_3 \neq j_1}^{b_1}
    \sum_{j_4 = b_1 + 1, j_4 \neq j_2}^{p}
    w_{j_1, j_2} w_{j_3, j_4}
    P_{j_1, j_2, j_3, j_4},
\end{align*}
and similarly for $Q_{\w, \beta}$ and $S_{\w, \beta}$ for $\beta \in \{0, 1, 2\}$.

\medskip

Now that we have all auxiliary notations in place, let us start by showing that each of the estimators is in fact a U-statistic.

\begin{lemma}
    \label{lem:Ustat}
    Under Assumption \ref{as:simp_structural}, 
    the estimators $\widehat{\tau}_{j_1,j_2}$, 
    $\widehat{\tau}^B$, $\widehat{\tau}^R$,
    $\widehat{\tau}^D$ and $\widehat{\tau}^{U}$
    are all second order U-statistics, in the sense that they can be written as 
    $(n(n-1))^{-1} \sum_{1 \leq i_1 \neq i_2 \leq n} g(\X_{i_1}, \X_{i_2})$,
    for some real-valued function $g$. We denote these functions respectively by $g^*$, $g^B$, $g^R$, $g^D$ and $g^U$.
\end{lemma}
This lemma is proved in Appendix~\ref{proof:lem:Ustat}
where the expression of the kernels
$g^*$, $g^B$, $g^R$, $g^D$ and $g^U$ are given.
From Lemma \ref{lem:Ustat} and the fact that $\E\left[  g\left( \X_{1} , \X_{2} \right) \right]
= \tau$ for each kernel $g$,
it follows that $\widehat{\T}$, $\widehat{\T}^B$, $\widehat{\T}^R$, $\widehat{\T}^D$ and $\widehat{\T}^U$ are all unbiased estimators of the Kendall's tau matrix under Assumption~\ref{as:structural}. The finite sample variances are given in the following theorem, proved in Appendix~\ref{proof:th:var}.

\begin{theorem}
    \label{th:var}
    Let $1 \leq j_1, j_2 \leq p$, $1 \leq b_1 \leq p$ and let $b_2 := p - b_1$.
    The variances of $\widehat{\tau}_{j_1,j_2}$, 
    $\widehat{\tau}^B$, $\widehat{\tau}^{R}$, 
    $\widehat{\tau}^{D}$, and $\widehat{\tau}^{U}$, and the conditional variance of $\widehat{\tau}^{U}$ conditionally to $\W$
    are given by
    \vspace{0.1em}
    \begin{itemize}
    \mathitem{[(i)]}{-0.19em}
    \begin{align*}
        \hspace{-9.5em}
        \Var \left[ \widehat{\tau}_{j_1,j_2} \right]
        &= \frac{8}{n(n-1)}\left(2(n-2) \left(
        Q_{j_1, j_2} - P_{j_1, j_2}^2 \right)
        + P_{j_1, j_2} - P_{j_1, j_2}^2\right),
    \end{align*}
    
    \vspace{0.5em}
    \mathitem{[(ii)]}{-0.15em}
    \begin{align*}
        \Var \left[ \widehat{\tau}^{B} \right]
        = \frac{8}{b_1 b_2 n(n-1)} \Big(
        &P_{B,2} - S_{B,2}
        + (b_1 + b_2 - 2) (P_{B,1} - S_{B,1})
        + (b_1 - 1)(b_2 - 1) (P_{B,0} - S_{B,0}) \\
        + \, 2(n-2) \big( &Q_{B,2} - S_{B,2}
        + (b_1 + b_2 - 2) (Q_{B,1} - S_{B,1})
        + (b_1 - 1) (b_2 - 1) (Q_{B,0} - S_{B,0}) \big)
         \Big).
    \end{align*}
    
    \vspace{0.15em}
    \mathitem{[(iii)]}{-0.3em} 
    \begin{align*}
        \Var \left[ \widehat{\tau}^{R} \right]
        = \frac{8}{N n(n-1)} \Big(
        P_{R,2} - S_{R,2}
        + (N - 1) (P_{R,1} - S_{R,1})
        + 2(n-2) \big( Q_{R,2} - S_{R,2}
        + (N - 1) (Q_{R,1} - S_{R,1}) \big)  \Big).
    \end{align*}
    
    \vspace{0.2em}
    \mathitem{[(iv)]}{-0.2em} 
    \begin{align*}
        \Var \left[ \widehat{\tau}^{D} \right]
        = \frac{8}{N n(n-1)} \Big(
        P_{D,2} - S_{D,2}
        + (N - 1) (P_{D,0} - S_{D,0})
        + 2(n-2) \big( Q_{D,2} - S_{D,2}
        + (N - 1) (Q_{D,0} - S_{D,0}) \big) \Big).
    \end{align*}
    
    \vspace{0.4em}
    \mathitem{[(v)]}{-0.4em}
    \begin{align*}
        \Var \left[ \widehat{\tau}^{U} | \W \right]
        = \frac{8}{b_1 b_2 n(n-1)} \Big(
        &P_{\W,2} - S_{\W,2}
        + (b_1 + b_2 - 2) (P_{\W,1} - S_{\W,1})
        + (b_1 - 1)(b_2 - 1) (P_{\W,0} - S_{\W,0}) \\
        + 2(n-2) \big( &Q_{\W,2} - S_{\W,2}
        + (b_1 + b_2 - 2) (Q_{\W,1} - S_{\W,1})
        + (b_1 - 1) (b_2 - 1) (Q_{\W,0} - S_{\W,0}) \big) \Big).
    \end{align*}
    
    \vspace{0em}
    \mathitem{[(vi)]}{-0.3em}
    \begin{align*}
        & \Var \left[\widehat{\tau}^{U} \right]
        = \Var \left[ \tau_{J_1, J_2} \right]
        + \frac{8}{b_1 b_2 n(n-1)} \Big(
        2(n-2) \\
        &\times \big( Q_{B,2} - S_{B,2}
        + \frac{N - 1}{b_1 b_2 - 1}
        (b_1 + b_2 - 2) (Q_{B,1} - S_{B,1})
        + \frac{N - 1}{b_1 b_2 - 1}
        (b_1 - 1) (b_2 - 1) (Q_{B,0} - S_{B,0}) \big) \\
        & + \big( P_{B,2} - S_{B,2}
        + \frac{N - 1}{b_1 b_2 - 1}
        (b_1 + b_2 - 2) (P_{B,1} - S_{B,1})
        + \frac{N - 1}{b_1 b_2 - 1}
        (b_1 - 1)(b_2 - 1) (P_{B,0} - S_{B,0}) \big) \Big),
        \hspace{5cm}
    \end{align*}
    where $J_1$ and $J_2$ are independent random variables, uniformly distributed on
    $\{1, \dots, b_1\}$ and $\{b_1+1, \dots, p\}$,
    respectively.
    \end{itemize}
\end{theorem}

Note that the variance of the usual Kendall's tau estimator $\widehat{\tau}_{j_1,j_2}$ is already known, see e.g. \cite{VarKT}. We can also remark that this theorem always holds, even if Assumption~\ref{as:simp_structural} is not satisfied. However in such a case, the estimators will have different expectations: this would be the average of the considered Kendall's tau over the corresponding pairs.
Note that the conditional variance $\Var \left[ \widehat{\tau}^{U} | \W \right]$ is the variance of the estimator $\widehat{\tau}^{U}$ for a fixed choice of pairs to average; it only measure the variability due to the randomness of the sample.
On the contrary, the unconditional variance $\Var \left[\widehat{\tau}^{U} \right]$ takes into account both the randomness of the sample and the randomness of the choice of the pairs.

\begin{remark}
    If Assumption~\ref{as:simp_structural} holds, the first term in the variance of $\widehat{\tau}^{U}$ vanishes, and $P_{B,2} = P_{R,2} = P_{D,2} = P_{j_1, j_2}$ and all $S$ terms become equal to $P_{j_1, j_2}^2$.
    If the stronger Assumption \ref{as:PEA} holds, all quantities $P_{j_1,j_2,j_3,j_4}$ and $Q_{j_1,j_2,j_3,j_4}$ becomes independent of the choice of indices $(j_1,j_2,j_3,j_4)$. Therefore, all $P$-type averages becomes equal, as well as all $Q$-type averages.
\end{remark}

\begin{corollary}
    \label{corollary:asympt_norm_uncond}
    Under Assumption \ref{as:simp_structural},
    $P^2 := P_{j_1, j_2}^2$ does not depend on the choice of $j_1, j_2$ and it holds that as $n\to \infty$
    \begin{align*}
        n^{1 / 2}
        \left( \widehat{\tau} - \tau \right) 
        \stackrel{\textnormal{law}}{\longrightarrow} 
        \mathcal{N}\left(0, V\right),
    \end{align*}
    here $\widehat{\tau}$ denotes any of the estimators $\widehat{\tau}_{j_1,j_2}$,
    $\widehat{\tau}^{B}$,
    $\widehat{\tau}^{R}$,
    $\widehat{\tau}^{D}$,
    $\widehat{\tau}^{U}$ given $\W$, 
    $\widehat{\tau}^{U}$,
    and the corresponding asymptotic variances $V$ are respectively given by
    \begingroup \allowdisplaybreaks
    \begin{align}
        \label{eq:def_asymptV1}
        &V_{j_1,j_2} \hspace{0.14em} = V_{j_1,j_2}(\Pb_\X) \hspace{0.14em} 
        := 16 (Q_{j_1, j_2} - P^2), \\[1em]
        &V^B = V^B(\Pb_\X)
        := \frac{16}{b_1 b_2}
        \Big( Q_{B,2} - P^2
        + \big( b_1 + b_2 - 2 \big) \big( Q_{B,1} - P^2 \big)
        + \big(b_1 - 1 \big) \big(b_2 - 1\big) 
        (Q_{B,0} - P^2) \Big) , \\
        &V^R = V^R(\Pb_\X)
        := \frac{16}{N} \Big( Q_{R,2} - P^2
        + (N - 1) (Q_{R,1} - P^2) \Big),\\
        &V^D = V^D(\Pb_\X)
        := \frac{16}{N} \Big( Q_{D,2} - P^2
        + (N - 1) (Q_{D,0} - P^2) \Big),\\
        &V^{U|\W} = V^{U|\W}(\Pb_\X)
        := \frac{16}{N} \Big( Q_{\W,2} - P^2
        + \big( b_1 + b_2 - 2 \big)
        \big( Q_{\W,1} - P^2 \big)
        + \big(b_1 - 1 \big) \big(b_2 - 1\big) 
        (Q_{\W,0} - P^2) \Big) , \\
        &V^U = V^U(\Pb_\X)
        := \frac{16}{N} \bigg( Q_{B,2} - P^2
        + \frac{N - 1}{b_1 b_2 - 1 }
        \Big( \big( b_1 + b_2 - 2 \big)
        \big(Q_{B,1} - P^2 \big)
        + \big( b_1 - 1 \big) \big( b_2 - 1 \big) 
        \big( Q_{B,0} - P^2 \big) \Big) \bigg),
        \label{eq:def_asymptVU}
    \end{align}
    \endgroup
    where $\Pb_\X$ denotes the law of the random vector $\X$.
    Furthermore, these distributions are degenerate whenever the corresponding $Q$-type averages are equal to $P^2$. A sufficient condition for this to happen is when all Kendall's tau are equal to $1$.
\end{corollary}

This corollary can straightforwardly be derived by combining Theorem A of~\cite[Section 5.5.1]{serfling2009approximation} and the computations of the corresponding $\zeta_1$'s in the proof of Theorem~\ref{th:var}.
The asymptotic normality of $\widehat{\tau}^{B}$ was already known to hold under the stronger assumption of partial exchangeability, see \cite[Theorem 1.2]{perreault2020structures}.
From the lengthy expressions in Theorem \ref{th:var} we can derive the asymptotic variances in the setting where
$n, b_1, b_2, N \to +\infty$.
\begin{corollary}
\label{corollary:uncond}
Under Assumption \ref{as:simp_structural}, as
$n, b_1, b_2, N \to +\infty$, we have the following equivalents
\begin{itemize}
    \item $\displaystyle \Var \left[ \widehat{\tau}_{j_1,j_2} \right] 
    \sim \frac{16}{n}
    \left(Q_{j_1,j_2} - P^2 \right)
    = \frac{1}{n} \times V_{j_1,j_2} (\Pb_\X),$
    \item $\displaystyle \Var \left[ \widehat{\tau}^{B} \right] 
    \sim \frac{16}{n}
    \left(Q_{B,0} - P^2 \right)
    = \frac{1}{n} \times \liminftyg V^B(\Pb_\X),$
    \item $\displaystyle \Var \left[ \widehat{\tau}^{R} \right] 
    \sim \frac{16}{n} \left(Q_{R,1} - P^2 \right)
    = \frac{1}{n} \times \liminftyd V^R(\Pb_\X),$
    \item $\displaystyle \Var \left[ \widehat{\tau}^{D} \right] 
    \sim \frac{16}{n} \left(Q_{D,0} - P^2 \right)
    = \frac{1}{n} \times \liminftyd V^D(\Pb_\X),$
    \item $\displaystyle \Var \left[ \widehat{\tau}^{U} | \W \right] 
    \sim \frac{16}{n} \left(Q_{\W,0} - P^2 \right)
    = \frac{1}{n} \times \liminftygd V^{U|\W}(\Pb_\X),$
    \item $\displaystyle \Var \left[ \widehat{\tau}^{U} \right] 
    \sim \frac{16}{n} \left(Q_{B,0} - P^2 \right)
    = \frac{1}{n} \times \liminftygd V^U(\Pb_\X),$
\end{itemize}
assuming respectively that $Q_{j_1,j_2}, Q_{B,0}, Q_{R,1}, Q_{D,0}, Q_{\W,0}, Q_{B,0}$ are strictly larger than $P^2$. If this is not the case, the corresponding variances converges to $0$ at a rate faster than $O(1/n)$.
\end{corollary}

Surprisingly, note that the variances \textbf{do not depend on the block dimensions} as soon as they are large enough. This is also true if the block dimensions tends to the infinity at different rates. In the limit, the quality of the estimator will therefore not improve by averaging over additional elements in general. Note that this is coherent, since we do not assume the dependence to converge to $0$, which would correspond to some mixing assumption. Therefore, averaging must have only a limited effect, as in the simpler statistical model $Y_i = \theta + \epsilon_i$ where $\epsilon = (\epsilon_1, \dots, \epsilon_n)$ follows a centered exchangeable Normal distribution with correlation $\rho > 0$. In this case, the average $\overline Y_n$ is inconsistent for estimating $\theta$ since $\E[(\overline Y_n - \theta)^2] = n^{-2} \sum_{i,j=1}^n \E[\epsilon_i \epsilon_j] = 1/n + (n-1) \rho / n \to \rho$ as $n \to \infty$.

\medskip

For large sample sizes, only the $Q$-type averages determine the levels of variance.
For the diagonal, random and block estimators, the number of terms corresponding to non-overlapping combinations grow faster than the number of overlapping combinations. 
However, for the row estimator only pairs within the same row are averaged, and thus the limiting variance contains the quantity $Q_{R,1}$ (instead of a hypothetical $Q_{R,0}$).

\medskip

\noindent\textbf{Open problem 1.} As all asymptotic variances are equal up to a constant, it seems logical to ask which is the best estimator. For this, we need to compare these constants. However, these constants are defined using $8$-dimensional integrals, making explicit computations difficult.

\medskip

Interestingly, the block estimator and the random estimator perform equally well in the limit.
Hence, we can greatly reduce computation time by using the random estimator instead of the block estimator, while still maintaining a low asymptotic variance.

\medskip

This is coherent with Theorem 1 of \cite{perreault} which shows that under Assumption \ref{as:PEA} the block averaging estimator is optimal with respect to the Mahalanobis distance.
Furthermore, since the diagonal estimator averages solely over non-overlapping combinations, note that it should converge faster than that of the random estimator.
Therefore, if computation costs are to be reduced, the diagonal estimator is preferable to both the random and the row estimator.

\medskip

Lastly, we note that if only part of the row or diagonal is averaged, the asymptotic variances of the resulting estimators do not change. By doing so we can further lower computation times, but at the cost of attaining the limiting variances at slower rates. It therefore makes sense to choose $N$ large enough to attain the asymptotic regime, but not too large to keep a low computation time.

\section{Fast estimation of conditional Kendall's tau matrix}
\label{sec:CKT}

We extend the aforementioned setting to the conditional setup, when a $d$-dimensional covariate $\Z$ is available taking values in $\mathcal{Z} \subset \Rb^d$. Formally, this means that we observe a sample $(\X_i, \Z_i)_{i=1, \dots, n}$ of $n$ independent and identically distributed replications of a random vector $(\X, \Z) \in \Rb^{p+d}$.
The objective is now to estimate the $p \times p$ conditional Kendall's tau matrix $\T_{\mid \Z=\z} = \left[ \tau_{j_1,j_2\mid \Z=\z} \right]_{1 \leq j_1, j_2 \leq p}$
for a given point $\z \in \mathcal{Z}$.

\subsection{Estimation of conditional Kendall's tau}
\label{sec:EstCKT}

For construction of nonparametric estimates of the conditional Kendall's tau, let us start by recalling the expression of the conditional Kendall's tau, following~\cite{derumigny2019kernel}:
\begin{align*}
    \tau_{1,2\mid \Z=\z}
    &= \mathbb{P}\left(\left(X_{1,1}-X_{1,2}\right)\left(X_{2,1}-X_{2,2}\right)>0\mid \Z_1 = \Z_2 = \z\right)
    -\mathbb{P}\left(\left(X_{1,1}-X_{1,2}\right)\left(X_{2,1}-X_{2,2}\right)<0\mid \Z_1 = \Z_2 = \z\right)
\end{align*}
Following the approach of \cite{derumigny2019kernel}, we introduce a kernel-based estimator of $\tau_{1,2\mid \Z=\z}$ as
\begin{align}
    \label{eq:def:est_ckt}
    &\widehat{\tau}_{1,2\mid \Z=\z}
    := \frac{1}{1-s_n}
    \sum_{i_1=1}^n \sum_{i_2=1}^n
    w_{i_1,n}(\z) w_{i_2,n}(\z)
    g( \X_{i_1,(1,2)} , \X_{i_2,(1,2)} ), \\
    &\text{ where } g(  \X_{i_1,(1,2)} , \X_{i_2,(1,2)} )
    := \operatorname{sign}
    \left(\left(X_{i_1, 1} - X_{i_2, 1}\right)
    \left(X_{i_1, 2}-X_{i_2, 2}\right)\right)
    %
    \nonumber
\end{align}
with Nadaraya-Watson weights $w_{i,n}$ given by 
\begin{equation}
    w_{i,n}(\z) := \frac{\Kc_h(\Z_i - \z)}{\sum_{k=1}^n \Kc_h (\Z_{k}-\z)},
    \label{eq:def:weights_w_in}
\end{equation}
and $s_n := \sum_{i=1}^n w_{i,n}^2(\z),$ for some kernel $\Kc$ on $\Rb^d$ and a bandwidth sequence $h=h(n)$ converging to zero as $n\to \infty$.
In this sense, $\widehat{\tau}_{1,2\mid \Z=\z}$ is a smoothed estimator of $\tau_{1,2\mid \Z=\z}
=  \E \left[ g(\X_1, \X_2 ) \mid \Z_1 = \Z_2 = \z \right]$.
The factor $1/(1-s_n)$ tends to $1$ as $n \to \infty$,
and ensures that the estimated conditional Kendall's tau takes values in the whole interval $[-1,1]$.

\medskip

We adapt the Simplified Structural Assumption~\ref{as:simp_structural} by assuming that the underlying structural pattern applies to the conditional Kendall's tau matrix given $\Z = \z$. 
%
\begin{assumption}[Simplified Structural Assumption conditionally to $\Z = \z \in \mathcal{Z}$]
    \label{as:structuralconditional} 
    Kendall's tau matrix of the random vector $\X$ can be written in the block form
    \begin{align*}
        \T_{\mid \Z=\z} = \begin{pmatrix}
        \cdot & \tau_{\mid \Z=\z} \1 \\
        \tau_{\mid \Z=\z} \1 & \cdot
        \end{pmatrix}
    \end{align*}
    for some value $\tau_{\mid \Z=\z} \in [-1, 1]$,
    where $\tau_{\mid \Z=\z} \1$ represent a block filled with the value $\tau_{\mid \Z=\z} \in [-1, 1]$ and the $\cdot$ represent any matrices of respective sizes $b_1 \times b_1$ and $b_2 \times b_2$ for some $b_1 \in \{1, \dots, p\}$ and $b_2 := p - b_1$.
\end{assumption}

In terms of stock return modeling, Assumption \ref{as:structuralconditional} has the following interpretation: conditionally on a given market state or portfolio movement, the stocks of companies from different sectors/countries have equal rank correlations with every other pair from the respective groups.
This could for instance be used for the computation of conditional risk measures.

\medskip

Note that Assumption~\ref{as:structuralconditional} only concern a fixed value of $\z$, and is quite general in the sense that it allows the existence of different block structures depending on the value of the conditional variable $\z$.

\medskip

Let us denote the naive (unaveraged) conditional Kendall's tau matrix estimator as $\widehat{\T}_{\mid \Z=\z} = \left[ \widehat{\tau}_{j_1,j_2 \mid \Z=\z} \right]_{1 \leq j_1, j_2 \leq p}$ with $\widehat{\tau}_{j_1,j_2 \mid \Z=\z}$ as in \eqref{eq:def:est_ckt}. 
As in the unconditional framework that was studied previously, we define averaged versions of the conditional estimators $\widehat{\tau}^{B}_{\mid \Z=\z}$, $\widehat{\tau}^{R}_{\mid \Z=\z}$,
$\widehat{\tau}^{D}_{\mid \Z=\z}$,
$\widehat{\tau}^{U}_{\mid \Z=\z}$,
that respectively average over the whole block, the first row, the first diagonal, and uniformly chosen entries of the off-diagonal block.

\subsection{Comparison of their asymptotic variances}
\label{sec:AsympRes}

Before proceeding with the asymptotic results, we need to formalize some regularity assumptions on the kernel $\Kc$, the covariate $\Z$ and the bandwidth sequence $h(n)$. 
Since we will give similar results as in \cite[Proposition 9]{derumigny2019kernel} where the case of the (bivariate) conditional Kendall's tau was treated, we give an adapted version of their assumptions.  

\begin{assumption}
\label{as:kernel}
\begin{itemize}
    \item[(a)] The kernel $\Kc$ is bounded, compactly supported, symmetrical in the sense that $\Kc(\mb{u})=\Kc(-\mb{u})$ for every $\mb{u} \in \Rb^{d}$ and satisfies $\int \Kc=1, \int|\Kc|<\infty, \int \Kc^{2}<\infty$. 
    \item[(b)] The kernel is of order $\alpha$ for some integer $\alpha>1$, i.e. for all $k=1, \ldots, \alpha-1$ and every indices $j_{1}, \ldots, j_{k}$ in $\{1, \ldots, d\}$, $$\int \Kc(\mb{u}) u_{j_{1}} \ldots u_{j_{k}} d \mb{u}=0.$$
    \item[(c)] In addition, $\mathbb{E}\left[\Kc_h(\Z-\z)\right]>0$ for every $h > 0$.
\end{itemize}
\end{assumption}

These assumptions are classical in nonparametric statistics to obtain convergence rates of kernel-based estimators. The assumption of compactness of the support means that the estimators taken at different points in $\mathcal{Z}$ are independent if the sample size is large enough, since the bandwidth $h_n$ tends to $0$. The assumptions that the kernel is bounded and that $\int|\Kc|$ and $\int \Kc^{2}$ are finite rule out too much irregular kernels that have fat tails or irregular behavior. 
Higher-order kernels allows to obtain estimators that converge faster, under the assumption below that the joint density $f_{\X, \Z}$ is smooth enough, see for example Section 1.2.1 of \cite{tsybakov2003introduction}.
Assumption~\ref{as:kernel}(c) ensures that the weights that appear in Equation~\eqref{eq:def:weights_w_in} are well-defined asymptotically since the denominator will converge to a strictly positive value by the law of large numbers.

\begin{assumption}
    \label{as:conditionalvariate}
    For every $\x \in \Rb^{p}, \z \mapsto f_{\X, \Z}(\x, \z)$ is continuous and almost everywhere differentiable on a neighborhood of $\z$ up to the order $\alpha$. For every $0 \leq k \leq \alpha$ and every $1 \leq j_{1}, \ldots, j_{\alpha} \leq d$, let
    $$
    \mathcal{H}_{k, \vec{\jmath}}\left(\mb{u}, \mb{v}, \x_{1}, \x_{2}, \z \right)
    := \sup_{t \in[0,1]} \left|
    \frac{\partial^{k} f_{\X, \Z}}{\partial z_{j_{1}} \ldots \partial z_{j_{k}}}\left(\x_{1}, \z+t h \mb{u}\right)
    \frac{\partial^{\alpha-k} f_{\X, \Z}}{\partial z_{j_{k+1}} \ldots \partial z_{j_{\alpha}}}\left(\x_{2}, \z+t h \mb{v}\right)
    \right|
    $$
    denoting $\vec{\jmath}=\left(j_{1}, \ldots, j_{\alpha}\right) .$ Assume that $\mathcal{H}_{k, \vec{\jmath}}\left(\mb{u}, \mb{v}, \x_{1}, \x_{2}, \z\right)$ is integrable and there exists a finite constant $C_{\mb{XZ}, \alpha}>$ 0 such that, for every $h<1$,
    $$
    \int|\Kc|(\mb{u})|\Kc|(\mb{v}) \sum_{k=0}^{\alpha}\binom{\alpha}{k} \sum_{j_{1}, \ldots, j_{\alpha}=1}^{d}
    \mathcal{H}_{k, \vec{\jmath}}\left(\mb{u}, \mb{v}, \x_{1}, \x_{2}, \z \right)
    \left|u_{j_{1}} \ldots u_{j_{k}} v_{j_{k+1}} \ldots v_{j_{\alpha}}\right|
    d \mb{u} d \mb{v} d \x_{1} d \x_{2}
    $$
    is less than $C_{\mb{XZ}, \alpha}$.
\end{assumption}

The regularity condition on $f_{\X, \Z}$ can be interpreted in the classical way: smoother functions are easier to estimate than very irregular ones. Therefore, densities $f_{\X, \Z}$ that are $\alpha$-times differentiable allows to use a larger range of bandwidths for large $\alpha$; this can be seen in the following Assumption~\ref{as:bandwidth}.

Note that $C_{\mb{XZ}, \alpha}
\leq C_{\alpha, p} \| K \|_\infty^2 \sup_{k \in \{0, \dots, \alpha \}}
\sup_{\vec{\jmath} \, \in \, \{1,\dots, d\}^k}
\| \partial^k f_{\X, \Z} / \partial z_{j_{1}} \ldots \partial z_{j_{k}} \|_\infty^2$, where $C_{\alpha, p}$ is a constant that only depends on $\alpha$ and $p$.
In this sense, the second part of Assumption~\ref{as:conditionalvariate} can be seen as a relaxed version of a uniform control on the higher-order derivatives of the density. Therefore, Assumption~\ref{as:conditionalvariate} is implied by Assumption~\ref{as:kernel} and by the (stronger) assumption that all partial derivatives of $f_{\X, \Z}$ with respect to components of $\z$ exist up to the order $\alpha$ and are bounded.

\begin{assumption}
\label{as:bandwidth}
$ n h_{n}^{d} \rightarrow \infty \text { and } n h_{n}^{d+2 \alpha} \rightarrow 0 \text{ as } n \to \infty.$
\end{assumption}

This assumption controls the rate at which the sequence $(h_n)$ tends to $0$. It should tends to $0$ fast enough but not too fast: the second condition controls the bias (see Equation~\eqref{eq:control_bias} in the proof) while the first condition ensures that the rate $n h_{n}^{d}$ is meaningful and allows us to verify Lyapunov's condition (second part of \cite[Section A.10]{derumigny2019kernel}).
We now present our main theoretical result on the joint asymptotic normality at different points of the conditioning variable $\Z$, at the nonparametric rate $(n h_{n}^{d})^{1/2}$.

\begin{theorem}
\label{th:varconditional}[Joint asymptotic normality at different points].
Let $n' > 0$, and let $\z_{1}^{\prime}, \ldots, \z_{n^{\prime}}^{\prime}$ be a collection of $n'$ points of $\mathcal{Z} \subset \mb{R}^{d}$ such that Assumptions \ref{as:structuralconditional}-\ref{as:bandwidth} are satisfied for any choice $\z \in \{ \z_{1}^{\prime}, \ldots, \z_{n^{\prime}}^{\prime} \}$.
Then, as $n \rightarrow \infty$,
\begin{align*}
    \left(n h_{n}^{d}\right)^{1 / 2}
    \left(\widehat{\tau}_{ \mid \Z=\z'_{j'}}-\tau_{k_1,k_2 \mid \Z=\z'_{j'}}\right)_{j'=1, \ldots, n^{\prime}}
    \stackrel{\textnormal{law}}{\longrightarrow} \mathcal{N}\left(0, \mb{H}_{k_1,k_2}\right),
\end{align*}
where the diagonal matrix $\mb{H}$ is given by
\begin{align*}
    \mb{H}
    = \left[ \frac{\int \Kc^{2} \mathds{1}_{\{j_1'=j_2'\}}}{
    f_{\mb{Z}}\left(\mb{z}_{j_1'}^{\prime}\right)} \times
    V \Big(\Pb_{\X | \mb{Z} = \mb{z}'_{j_1'}} \Big),
    \right]_{1 \leq j_1', j_2' \leq n'}
\end{align*}
and the asymptotic variance functions $V$ are respectively defined in Equations~\eqref{eq:def_asymptV1}-\eqref{eq:def_asymptVU}. 
\end{theorem}

The proof of this result is given in Appendix~\ref{proof:th:varconditional}.

\begin{remark}
Under the assumptions of Theorem~\ref{th:varconditional}, we have for a given $\z \in \mathcal{Z}$,
\begin{align*}
    \lim_{n \to +\infty} \left(n h_{n}^{d}\right)^{1/2}
    \big(\widehat{\tau}_{\mid \Z=\z}
    - \tau_{\mid \Z=\z} \big)
    \, \equalinlaw \,
    \frac{\int \Kc^{2}}{f_{\Z}(\z)} \times 
    \lim_{n \to +\infty} n^{1/2}
    \big(\widehat{\tau}(\Pb_{\X | \Z = \z})
    - \tau_{\mid \Z=\z} \big),
\end{align*}
where on the left-hand side $\widehat{\tau}_{\mid \Z=\z}$ denotes any of the estimators
$\widehat{\tau}_{j_1,j_2 \mid \Z=\z}$, 
$\widehat{\tau}^{B}_{\mid \Z=\z}$,
$\widehat{\tau}^{R}_{\mid \Z=\z}$,
$\widehat{\tau}^{D}_{\mid \Z=\z}$, 
$\widehat{\tau}^{U}_{k_1,k_2 \mid \Z=\z}$,
and on the right-hand side $\widehat{\tau}(\Pb_{\X | \Z = \z})$ denotes the similarly-averaged estimated Kendall's tau if we had observed a sample of size $n$ from the distribution $\Pb_{\X | \Z = \z}$.
\end{remark}


\begin{corollary}
Under the same assumptions as in Theorem \ref{th:varconditional} and by letting the sample size and dimensions tend to infinity, the following holds for $\Pb_\Z$-almost all $\z \in \mathcal{Z}$,
\begin{itemize}
    \item $\Var \left[ \widehat{\tau}_{j_1,j_2\mid \Z=\z } \right]
    \sim \dfrac{16 \int \Kc^{2}}{nh^d f_{\Z}(\z)}
    \left(Q_{j_1,j_2 \mid \Z=\z} - P_{\mid \Z=\z}^2\right),$
    
    \item $\Var \left[ \widehat{\tau}^{B}_{\mid \Z=\z } \right]
    \sim \dfrac{16 \int \Kc^{2}}{nh^d f_{\Z}(\z)}
    \left(Q_{B,0 \mid \Z=\z} - P_{\mid \Z=\z}^2\right),$
    
    \item$\Var \left[ \widehat{\tau}^{R}_{\mid \Z=\z } \right]
    \sim \dfrac{16 \int \Kc^{2}}{nh^d f_{\Z}(\z)}
    \left(Q_{R,1 \mid \Z=\z} - P_{\mid \Z=\z}^2\right),$
    
     \item $\Var \left[ \widehat{\tau}^{D}_{\mid \Z=\z } \right]
    \sim \dfrac{16 \int \Kc^{2}}{nh^d f_{\Z}(\z)}
    \left(Q_{D,0 \mid \Z=\z} - P_{\mid \Z=\z}^2\right),$
    
    \item $\Var \left[ \widehat{\tau}^{U}_{\mid \Z=\z } | \W \right]
    \sim \dfrac{16 \int \Kc^{2}}{nh^d f_{\Z}(\z)}
    \left(Q_{\W,0 \mid \Z=\z} - P_{\mid \Z=\z}^2\right),$
    
    \item $\Var \left[ \widehat{\tau}^{U}_{\mid \Z=\z } \right]
    \sim \dfrac{16 \int \Kc^{2}}{nh^d f_{\Z}(\z)}
    \left(Q_{B,0 \mid \Z=\z} - P_{\mid \Z=\z}^2\right),$
\end{itemize}
assuming respectively that
$Q_{j_1,j_2 \mid \Z=\z},
Q_{B,0 \mid \Z=\z},
Q_{R,1 \mid \Z=\z},
Q_{D,0 \mid \Z=\z},
Q_{\W,0 \mid \Z=\z},
Q_{B,0 \mid \Z=\z}$
are strictly larger than $P_{\mid \Z=\z}^2$,
where conditional versions of $P$, $Q$ and their averages are defined with the same expression as in the previous section, but applied to the conditional law $\Pb_{\X|\Z=\z}$.

If these assumptions are not met, the corresponding variances converges to $0$ at a rate faster than $O(1/(nh^d))$.
\end{corollary}

As seen before, we remark that the asymptotic variances have analogous expressions to that of their unconditional counterparts. 
Therefore, all averaging estimators exhibit a lower asymptotic variance than the naive conditional Kendall's tau estimator. 
Also, the row averaging estimator intuitively performs worse than the block, diagonal and random estimators and for growing dimensions the block, diagonal and random  estimator perform (almost) equally in the limit, assuming that the averages of $U$ are not far apart.
Again, we can greatly reduce computation time by using either one of the diagonal or random averaging estimator instead of the block averaging estimator, since then only part of all conditional Kendall's taus have to be computed. 
Lastly, it again holds that if only part of the row or diagonal is averaged, the asymptotic variances of the resulting estimators do not change. By doing so we can further decrease computation time, but at the cost of attaining the limiting variances at slower rates.

\section{Simulation Study}
\label{sec:SS}

We perform a simulation study to assess the finite sample properties of our estimators.
First, in Section \ref{sec:simstuKT}, we compare the unconditional estimators by studying their variances and computation times for varying block and sample sizes.
In Section \ref{sec:simstuCKT}, we focus on the conditional versions of the diagonal and block estimators and we let the Kendall's taus depend on a one-dimensional covariate. 
Similarly, we compare their accuracy and computational efficiency for varying sample size and block dimensions.
In addition, we examine the estimators' optimal bandwidths under varying conditional dependencies of the Kendall's tau matrix. 
The simulations are all executed with the help of the statistical environment \texttt{R} \cite{Renv} on the DelftBlue supercomputer~\cite{DHPC2022}.
For simplicity, we choose $N = \min(b_1, b_2)$, so that diagonal, row and random estimators average over the same number of terms.

\subsection{Unconditional Kendall's Tau}
\label{sec:simstuKT}

In the unconditional framework, we compare the block, row, diagonal, random and naive Kendall's tau matrix estimators.
We will examine how the estimators' variance changes as a function of the block dimensions and the sample size. For this purpose, we consider mean squared errors (MSEs), which is a measure of variance here as all unconditional estimators are unbiased. 
Furthermore, we measure computation times for comparing the computational efficiency. For computing the pairwise sample Kendall's taus, we use the function \texttt{wdm} in the \texttt{R} package \texttt{wdm} \cite{wdm}, which can efficiently calculate sample and weighted Kendall's tau with time complexity~$O(n \log(n))$. 
The row, diagonal, random and block estimators are now available as part of the package \texttt{ElliptCopulas} \cite{ElliptCopulas}, and can be computed using the function \texttt{KTMatrixEst}.

\medskip

In each simulation, data is generated using a meta-elliptical copula~\cite{abdous2005dependence,derumigny2022identifiability,fang2002meta,genest2007metaelliptical}. A copula is said to be meta-elliptical if it is the copula of a distribution with density $ \lvert \bs{\Sigma} \rvert^{-1/2} g(\x^\top \Sigma \x)$ for a covariance matrix $\Sigma$ (here chosen to be a correlation matrix) and a function $g: \Rb_+ \to \Rb_+$, called the generator of the meta-elliptical copula.
This means that we simulate data from the Gaussian copula, or from other meta-elliptical copulas with different generators. 
Note that for meta-elliptical copulas, the matrix $\Sigma$ can be directly obtained from the Kendall's tau matrix.
We let the underlying Kendall's tau matrix be block-structured corresponding to two groups of equal size, which we will refer to as the block size. This results in two diagonal blocks and a single distinct off-diagonal block due to symmetry. 

\medskip

\noindent\textbf{Open problem 2.} 
Simulating from meta-elliptical copulas is easy, for example using the \texttt{ElliptCopulas} package \cite{ElliptCopulas} as they rely on elliptical distributions which are well-understood distributions. Constructing explicit non-meta-elliptical models of dependence that satisfy Assumption~\ref{as:simp_structural} seems difficult (unless both groups are independent) and is left for future research. Indeed, even in the low-dimensional case where $b_1 = b_2 = 2$
(so that the two blocks are $\{1, 2\}$ and $\{3, 4\}$), a D-vine (for example) would give a decomposition of the copula $c_{1,2,3,4} = c_{1,2} c_{2,3} c_{3,4} c_{1,3|2} c_{2,4|3} c_{1,4|2,3}$ of $(X_1, X_2, X_3, X_4)$.
So Kendall's tau $\tau_{2,3}$ can be easily chosen through the specification of the copula $c_{2,3}$. 
However, this would not give an explicit expression for the other interblocks Kendall's taus
$\tau_{1,3}, \tau_{1,4}, \tau_{2,4}$.
Indeed, they depend in a complicated way of all the copulas' and conditional copulas' parameters through integration.

\medskip

As obtained in Theorem \ref{th:var}, the variances depend on the averages of the auxiliary quantities $P,Q,$ and $S$ of pairs either along the row, the diagonal or over the entire block.
For a fair comparison of the different estimators, we need all different averages of these auxiliary quantities to be equal. As such, in addition to having identical Kendall's tau values in the off-diagonal block, we take the values within the diagonal blocks to be identical as well.
In that case, all auxiliary quantities are independent of the choice of pairs within the off-diagonal block, and moreover the Partial Exchangeability Assumption holds.

\medskip

For performance analysis, we will focus on estimates of the single off-diagonal block, as all estimators treat the diagonal blocks equally. As such, computation times and MSEs result from only estimating the single off-diagonal block.

\subsubsection{Effect of the sample size}
\label{subsubcec:uncondvarsamp}

In the first experiment, we study the dependency of the MSE on the sample size. To this end, the sample size is varied and the block size is fixed to $32 \times 32$.
The true Kendall's tau values are fixed in the following way:
the intragroup Kendall's tau are $\tau_{1,1} = \tau_{2,2} = 0.5$
and the intergroup Kendall's tau is $\tau_{1,2} = 0.3$.
We examine data generated from the Gaussian distribution.
For each estimator, the MSEs are calculated using 3000 replications. The results can be found on a log-log scale in Figure~\ref{fig:MSE_n}. 
\begin{figure}[hbt]
    \centering
    \includegraphics[width = \textwidth]{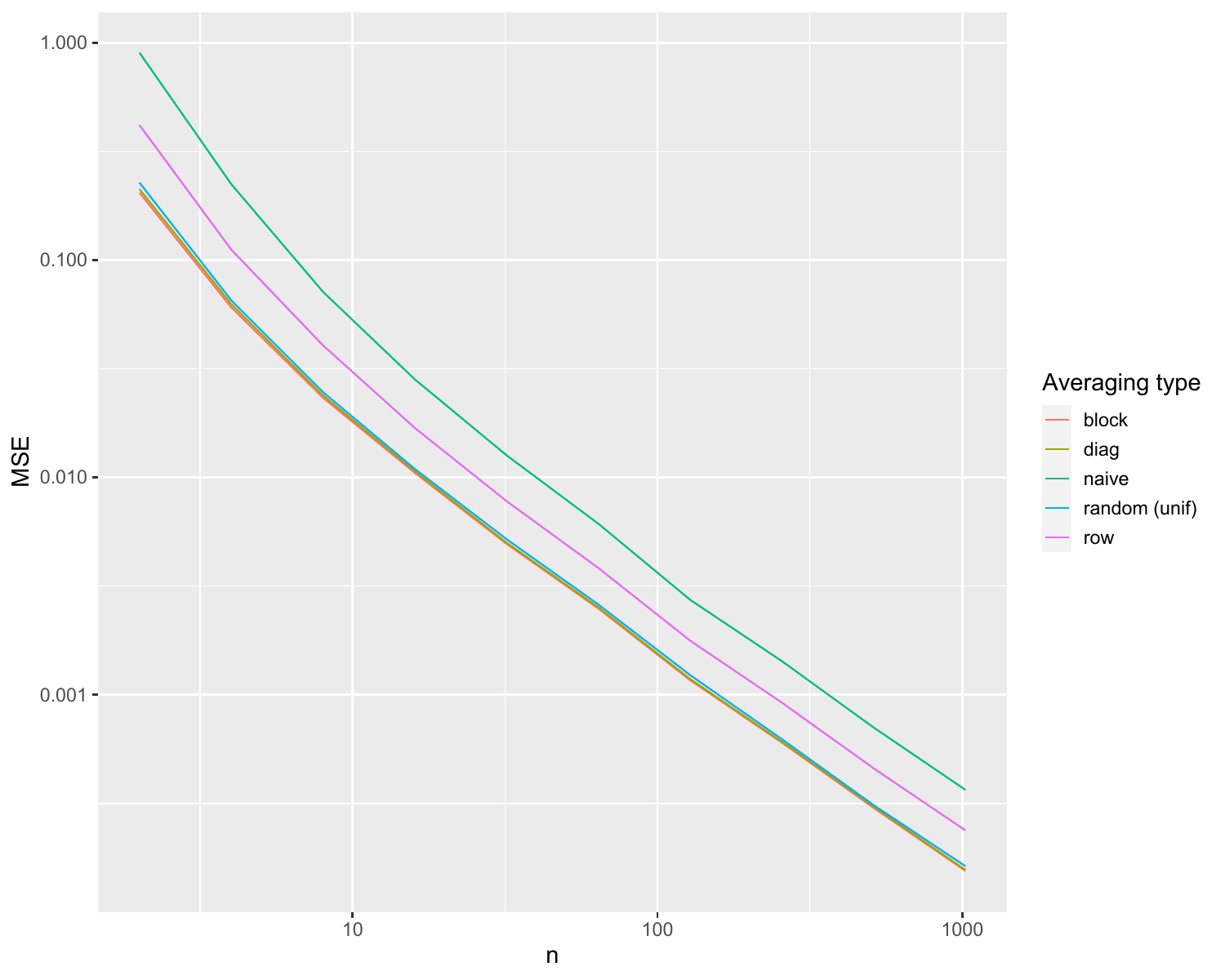}
    \caption{Log-log plots of
    the mean squared error of the estimators
    $\widehat{\tau}_{j_1,j_2}$ (``naive''), 
    $\widehat{\tau}^B$ (``block''), 
    $\widehat{\tau}^R$ (``row''),
    $\widehat{\tau}^D$ (``diag'')
    and 
    $\widehat{\tau}^U$ (``random'' with uniform selection of pairs)
    as a function of the sample size. The diagonal block Kendall's taus are set at $0.5$ and the off-diagonal block values at $0.3$. }
    \label{fig:MSE_n}
\end{figure}

In Figure~\ref{fig:MSE_n} we clearly observe almost straight lines for all of the estimators, with slopes indicating an inverse relationship. 
This not only confirms that the limiting variances are inversely proportional to the sample size, but also that this applies accurately for small sample sizes. 
In addition, we see that averaging the sample Kendall's taus does indeed lead to better estimates.
This applies to any given sample size, as all estimates depend equally on it.
As expected, the block estimator behaves best, only closely followed by the diagonal estimator. 

\medskip

Next, we study the dependency of the computation times on sample size. 
For this experiment, we compute the average computation time and the $5\%$ and $95\%$ quantile of the computation time for each experiment. The results are shown in Figure \ref{fig:meanRealtime_n} on a log-log scale.
It shows that the computation times gradually increase with the sample size, to a point where they appear to scale almost linearly with each other.
These observations are in line with the computation time $O(n\log (n))$ of the pairwise sample Kendall's tau estimator.  
As expected, the computation times of the block and sample Kendall's tau matrix estimator are very similar, as are the computation times of the row and diagonal estimators, with the latter two being significantly more efficient for any given sample size. 

\begin{figure}[hbt]
    \centering
    \includegraphics[width = \textwidth]{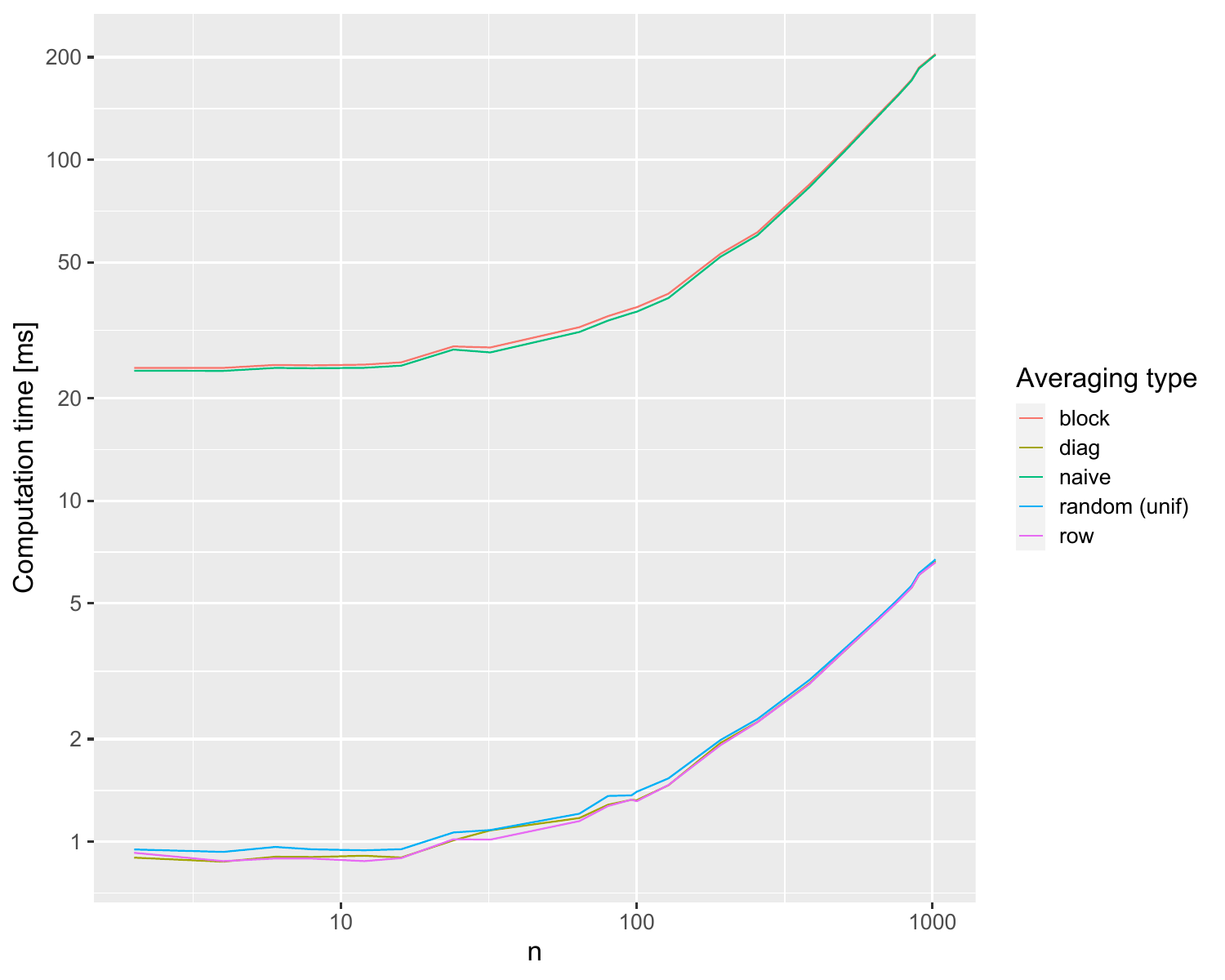}
    \caption{Log-log plot of 
    the mean computation time [ms] of the estimators
    $\widehat{\tau}_{j_1,j_2}$ (``naive''), 
    $\widehat{\tau}^B$ (``block''), 
    $\widehat{\tau}^R$ (``row''),
    $\widehat{\tau}^D$ (``diag'')
    and 
    $\widehat{\tau}^U$ (``random'' with uniform selection of pairs)
    as a function of the sample size, calculated using a block size of $32$.}
    \label{fig:meanRealtime_n}
\end{figure}

\subsubsection{Effect of the block size}
\label{subsubsec:ucondvarblock}

We first study the behavior of the MSE with respect to varying block sizes with off-diagonal block Kendall's taus of $0.3$ and diagonal block Kendall's taus of $0.5$.
In this experiment, we set the sample size to $4$ to reduce the computational cost of running a sufficient number of replications.
Again, we examine data generated from the Gaussian distributions. 
The MSEs are calculated using 3000 replications.
See Figure \ref{fig:MSE_vs_dim_unconditional} for a log-log plot of the MSEs as a function of the block size. 

\begin{figure}[hbt]
    \centering
    \includegraphics[width = \textwidth]{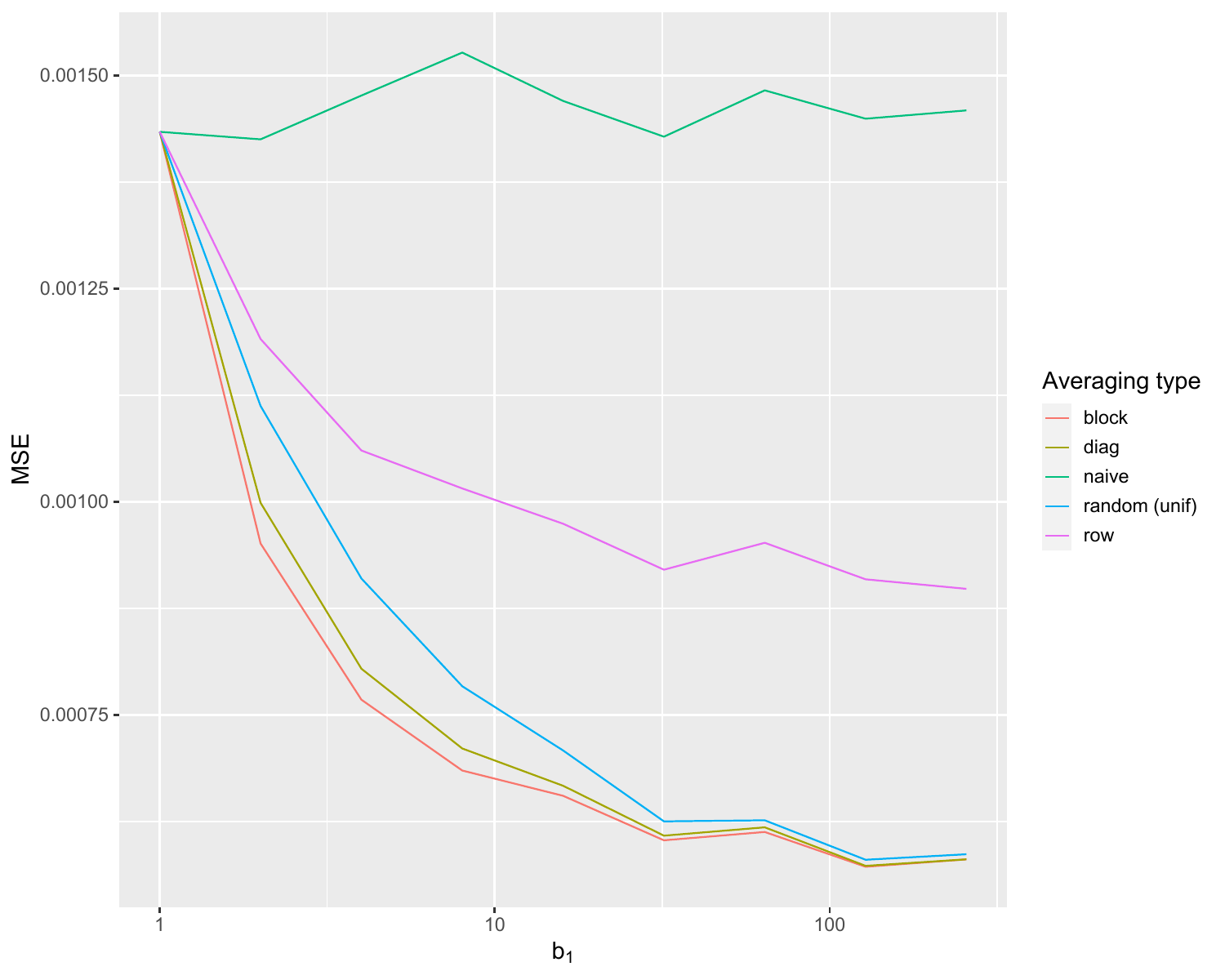}
    \caption{Log-log plots of 
    the mean squared error of the estimators
    $\widehat{\tau}_{j_1,j_2}$ (``naive''), 
    $\widehat{\tau}^B$ (``block''), 
    $\widehat{\tau}^R$ (``row''),
    $\widehat{\tau}^D$ (``diag'')
    and 
    $\widehat{\tau}^U$ (``random'' with uniform selection of pairs)
    as a function of the block size. The diagonal block Kendall's taus are set at $0.5$ and the off-diagonal block values at $0.3$.}
    \label{fig:MSE_vs_dim_unconditional}
\end{figure}

\medskip

Figure~\ref{fig:MSE_vs_dim_unconditional} shows that all of the averaging estimators perform increasingly better than the sample Kendall's tau estimator for growing block dimensions. 
For large block dimensions the MSEs seem to reach constant values, confirming that the asymptotic variances do not depend on block dimensions.
As expected, the block and diagonal averaging estimators both converge to the lowest limiting variance, approached fastest by the block averaging estimator. The row and the random averaging estimator performs considerably less.



Furthermore, we find that the relative difference between the diagonal and block estimator is largest for small dimensions, but they are still well within a factor of $1.5$ of each other. 
As the dimension increase, the MSEs of the diagonal estimator converge rapidly to that of the block estimator, again confirming that  the block and diagonal estimators have close variances for large block dimensions.
A more detailed presentation of the interplay between the sample size $n$ and the size $b_1$ and $b_2$ of the blocks is available in the Appendix, Figure~\ref{MSEnb1b2}.

\subsubsection{Effect of the value of the true Kendall's taus}

In this section, we fix the sample size $n = 16$ and the block sizes $b_1 = b_2 = 4$, and we change the value of Kendall's tau.
The MSE is displayed in Figure~\ref{fig:MSE_tau_reduced} for different combinations of $\tau_{1,1}, \tau_{1,2}$ and $\tau_{2,2}$. Note that some combinations are not present due to the constraints presented in Section~\ref{sec:consequences_corMatrix}.

\begin{figure}[hbt]
    \centering
    \includegraphics[width = \textwidth]{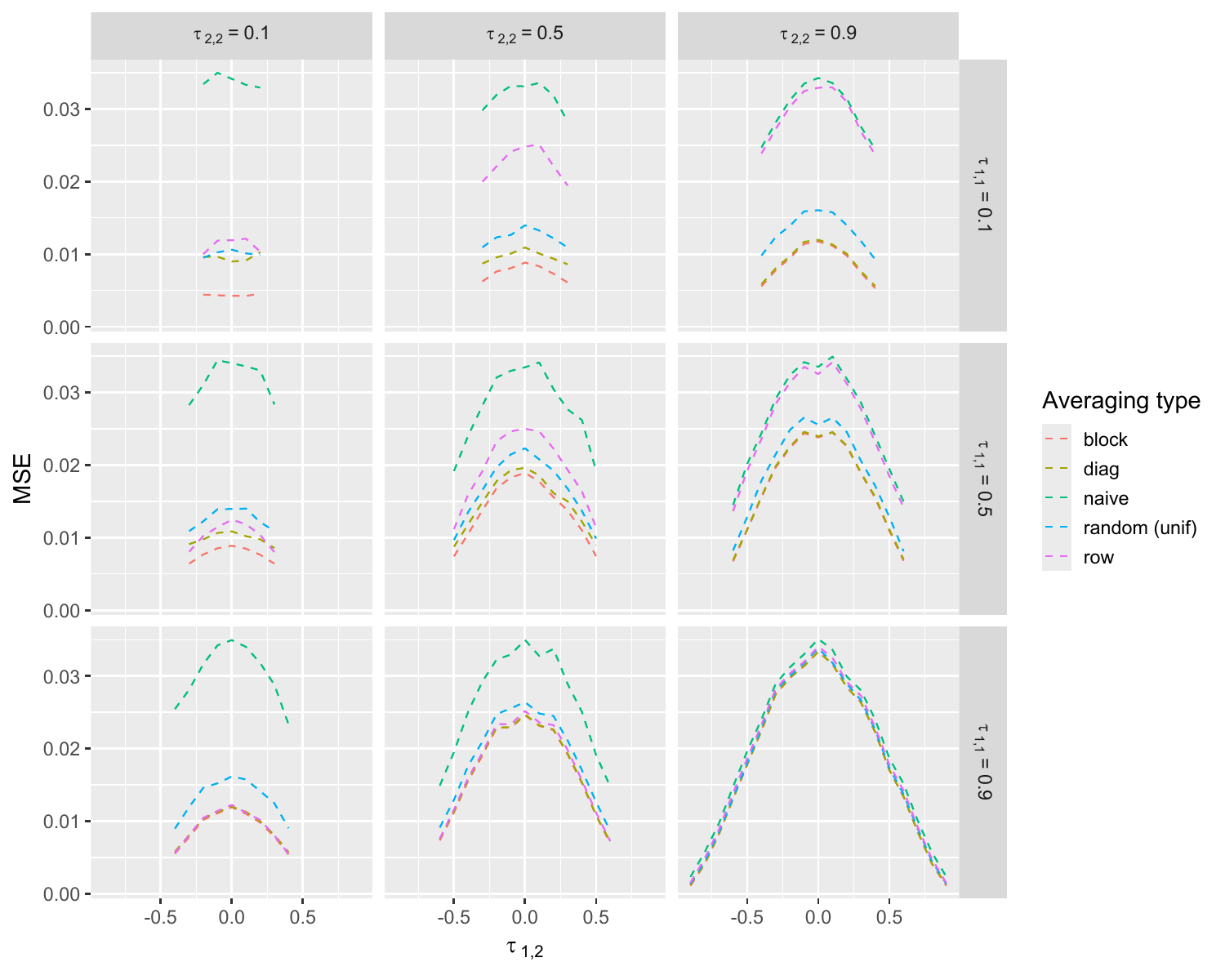}
    \caption{MSE as a function of the intergroup Kendall's tau $\tau_{1,2}$ for different combinations of $\tau_{1,1}$ and $\tau_{2,2}$. We use the sample size $n = 16$ and the block sizes $b_1 = b_2 = 4$.}
    \label{fig:MSE_tau_reduced}
\end{figure}

As expected, the relative order of the estimator is the same for all values of the Kendall's tau.
A more detailed version of this figure is available
in the Appendix, Figure~\ref{MSEtau},
showing the same phenomena for a larger range of value for the Kendall's taus.
Interestingly, the performance of all estimators is very similar when both intragroup Kendall's taus are equal to $0.9$.
Indeed, in this case the variables in each blocks are mostly the same, and then averaging does not change the situation anymore.

\subsubsection{Effect of the copula}

In this section, we fix the sample size, block sizes and Kendall's tau value.
We vary instead the copula of the distribution.
For this, we use different meta-elliptical copulas, because of their natural relationships between Kendall's tau and the underlying correlation matrix.
These meta-elliptical copulas are respectively defined as the copulas of the distributions with densities
$ \lvert \bs{\Sigma} \rvert^{-1/2} g(\x^\top \Sigma \x)$ for a covariance matrix $\Sigma$ (here chosen to be a correlation matrix) and a function $g: \Rb_+ \to \Rb_+$ respectively chosen as
\begin{itemize}
    \item $x/(1+x^3)$,
    \item $1/(1+x^2)$,
    \item $\exp(-x) \times |\cos(x)|$,
    \item $\exp(-x/2) + \exp(-x) \times |\cos(x)|$,
    \item $\exp(-x) + \exp(-x/3) \times \cos(x)^2$.
\end{itemize}
The results are displayed in Figure~\ref{fig:Var_gen}.
We can observe that the relative order of the estimators is mostly the same as before, with the averaging estimator the best and the no-averaging estimator having the highest mean square error.

\begin{figure}[hbt]
    \centering
    \includegraphics[width = \textwidth]{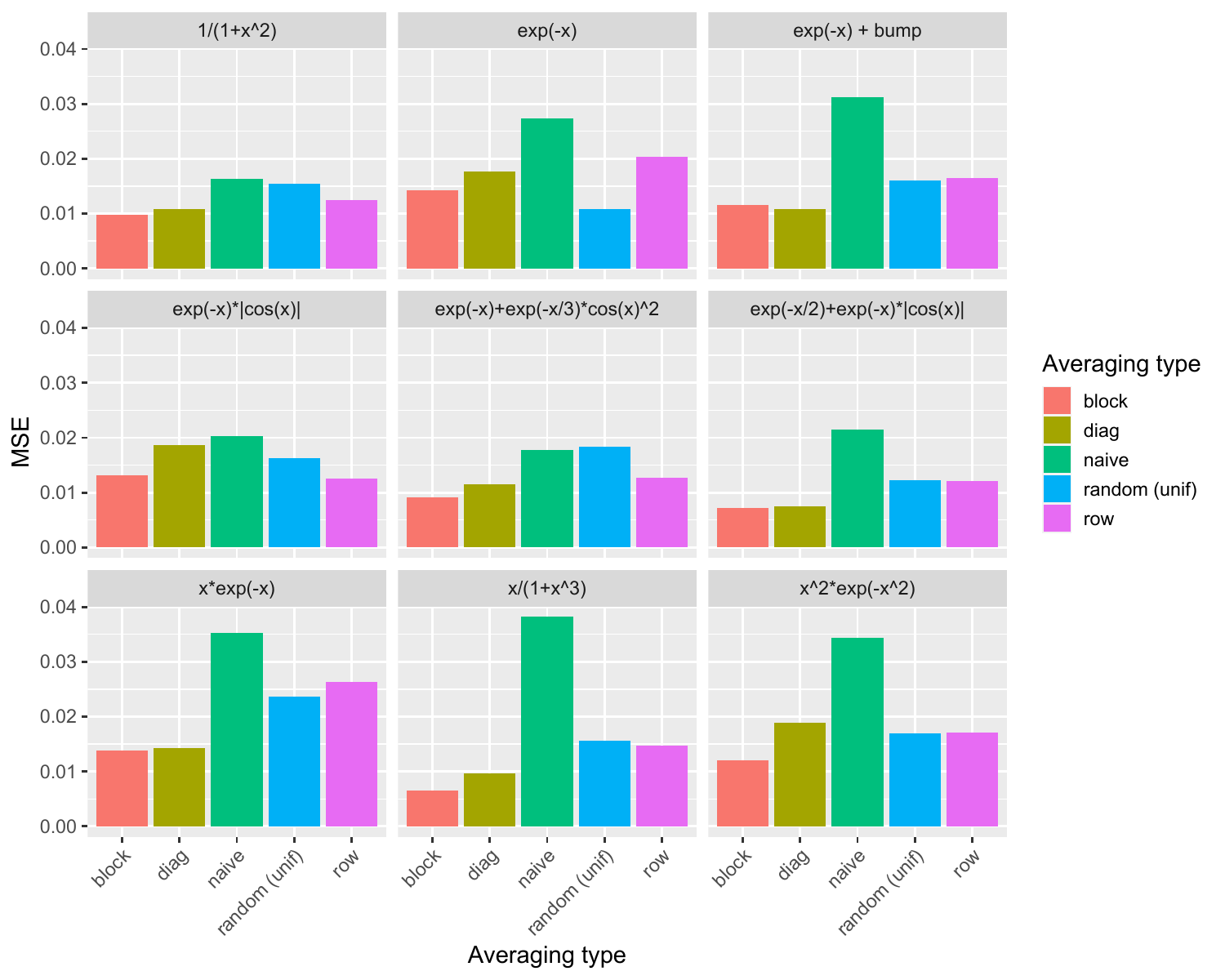}
    \caption{MSE for different meta-elliptical copulas and different estimators.}
    \label{fig:Var_gen}
\end{figure}

\subsection{Conditional Kendall's Tau}
\label{sec:simstuCKT}

In this section, we study the conditional versions of the block and diagonal estimators.
Since the estimators make use of kernel regression, a larger sample size is needed for obtaining stable results.
We therefore consider only a one-dimensional covariate $Z$, so that we do not need to increase the sample size even further and can run a sufficient number of replications. 
Kernel estimation is carried out with the Epanechnikov kernel~\cite{tsybakov2003introduction} and the estimation procedures are now available in the function \texttt{CKTmatrix.kernel} of the \texttt{R} package \texttt{CondCopulas}~\cite{CondCopulas}.

\medskip

In each of the experiments, we let the covariate $Z$ be uniformly distributed on the interval $[0,1]$.
We will estimate conditional Kendall's taus for points $z$ ranging from 0 to 1 in steps of 0.1. 
We generate data with the Gaussian distribution, as other distribution yield similar results.  
All variables will have a mean of $Z$ and variance of $1+Z^2$.
The Kendall's tau matrix is again block-structured corresponding to two groups of equal size.
Similarly to the unconditional case, we only focus on the estimates of the single off-diagonal block.
We set all Kendall's taus within the diagonal blocks to a constant value of $0.3$, which is independent of $Z$.
Finally, we let the Kendall's taus within the off-diagonal block depend on the covariate $Z$. 

\medskip

In Section \ref{subsubsec:condsamplesize}, we examine the accuracy and computational efficiency of the estimates under varying sample size.
A similar analysis for the effect of the block size is done in Appendix  \ref{subsubsec:condblocksize}.
To this end, we set Kendall's tau in the off-diagonal blocks to $0.1\,Z$. As $Z$ is distributed on $[0,1]$, the conditional Kendall's taus range from $0$ to $0.1$. As such, the underlying variables are again partially exchangeable conditionally to $Z=z$ for any $z \in [0,1]$.
It follows that the biases of the pairwise estimates in the off-diagonal block are all equal and thus that averaging over them does not change the total bias.
Since therefore all estimators have equal biases, we focus on the sample variances instead of the MSEs for a comparison of accuracies. 

\medskip

Then, in Section \ref{subsubsec:bandwidthselection} we study optimal bandwidths where we vary the way in which the off-diagonal block Kendall's taus depend on $Z$. 
We consider a model in which we let the off-diagonal block conditional Kendall's taus be given by
$$\left[\T_{\vert Z = z}\right]_{\mathcal{B}_{1,2}} = 0.1(\cos(0.5 \pi \omega z) + 1) \one , $$
with frequencies $\omega$ in $\{1, 2, 3, 4\}$. 
As such, these conditional Kendall's taus range from $0$ until $0.2$. 
For comparing the accuracies under varying bandwidths, we study mean integrated squared errors (MISE) computed by averaging the MSEs of conditional estimates computed at conditioning points ranging from $0$ to 1 in steps of 0.1.

\subsubsection{Effect of the sample size}
\label{subsubsec:condsamplesize}

In this experiment, we study the dependency of the variances on the sample size. To this end, we vary the sample size under a fixed block size of $4$ and a bandwidth of $0.5$. 
We use this relatively large bandwidth to ensure stable results even at lower sample sizes.
The integrated variance $\mathrm{IVAR} := \int_{z = 0}^1 \Var[\widehat\tau_{| Z = z}] dz$ is estimated as the average of sample variances for the grid points $z_i = i/n$, $i = 0, \dots, 10$, using 3000 replications.
Supplementary plots showing the influence of the sample size on each term $\Var[\widehat\tau_{| \, \Z = z_i}]$ are available in the Appendix~\ref{sec:additional_fig}, Figure~\ref{fig:MSE_vs_dim_conditional_wholegrid}.

\begin{figure}[hbt]
    \centering
    \includegraphics[width = \textwidth]{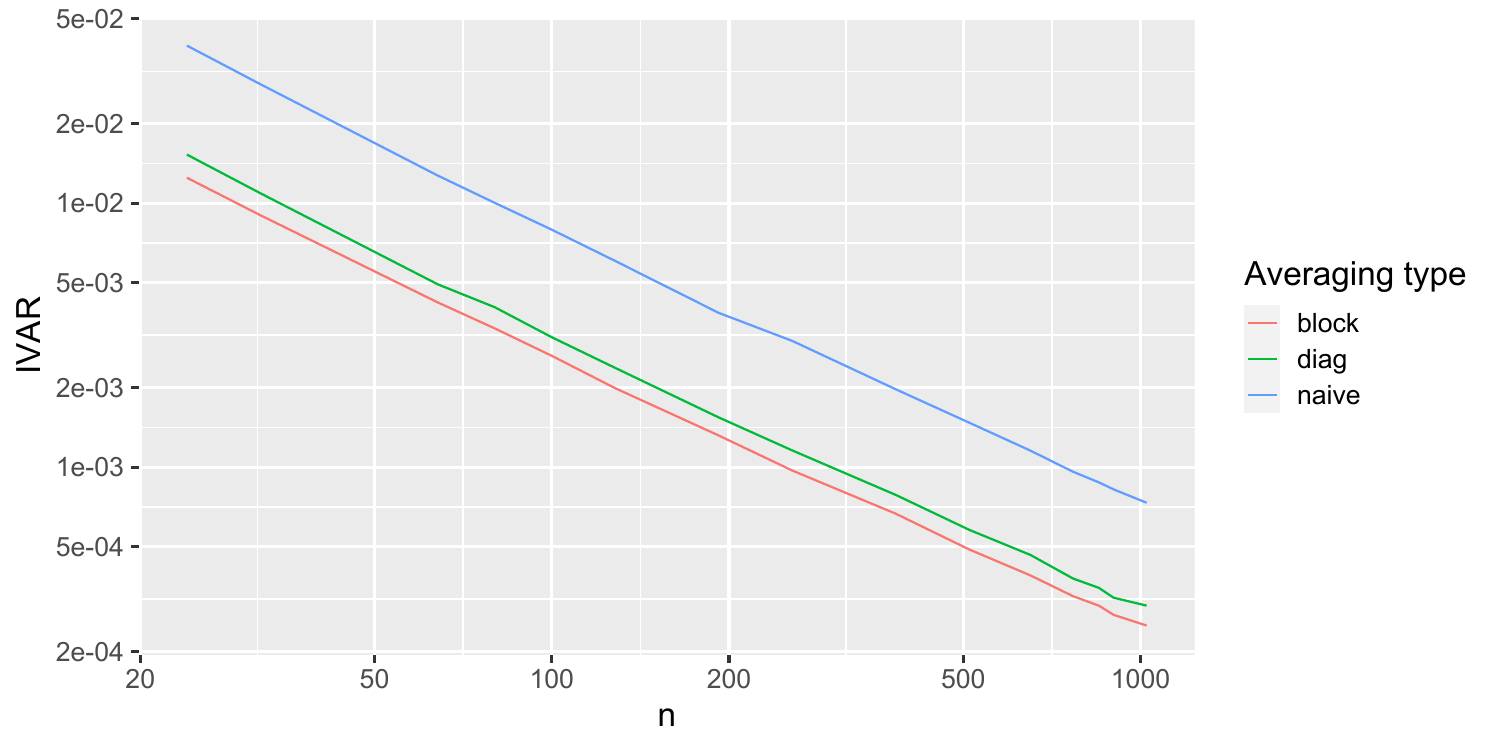}
    \caption{Log-log plots of the conditional estimators' integrated variance as a function of the sample size, using a block size of $4$ and a bandwidth of 0.5.}
    \label{fig:fig_IVAR_CKT_n}
\end{figure}

\medskip

Unsurprisingly, the conditional variances are also inversely related to the sample size. It follows that if bandwidths are kept constant, MSEs converge to the bias.
As such, appropriate bandwidths are naturally smaller for larger sample sizes. Furthermore, it is seen that the estimates near the edges of the interval $[0,1]$ are less accurate than those in the middle. 
This can be attributed to the fact that there are fewer observations of $Z$ near grid points close to the edges than near grid points in the middle, since the observations can be found there on both sides.
Evidently, a change in the distribution of $Z$ also changes the level of the variances.

\medskip

Next, let us study the dependency of the computation time on the sample size. We leave the setting unchanged, though the results correspond to the calculation of the conditional block estimates on a single grid point. 
The results are computed using 500 replications and are represented on log-log scale in Figure~\ref{fig:Comp_time_vs_timepoints_conditional}.
Here it is seen that the computation times gradually increase with the sample sample to a point where they appear to scale quadratically with each other.
This behavior follows from the fact that the conditional estimates require the calculation of a double sum of $n$ terms.
Note that the computation times of the diagonal and block estimators are relatively close since only a block size of $4$ is used here.

\begin{figure}[hbtp]
    \centering
    \includegraphics[width = \textwidth]{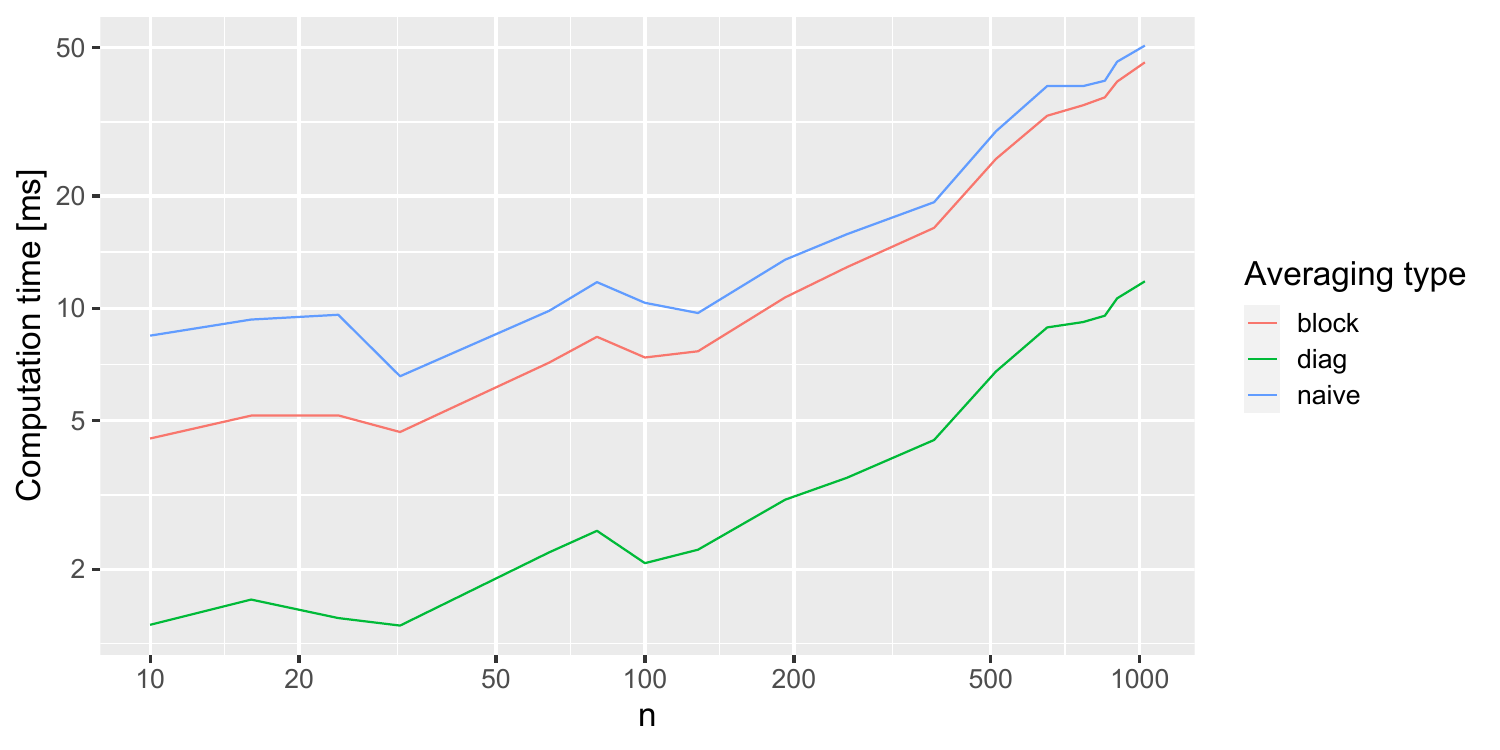}[]
    \caption{Log-log plot of the estimated mean computation time [ms] of the conditional estimator as a function of the sample size, for a block size of 4.}
    \label{fig:Comp_time_vs_timepoints_conditional}
\end{figure}

\subsubsection{Bandwidth selection}
\label{subsubsec:bandwidthselection}

\begin{figure}[hbt]
    \centering
    \includegraphics[scale=1]{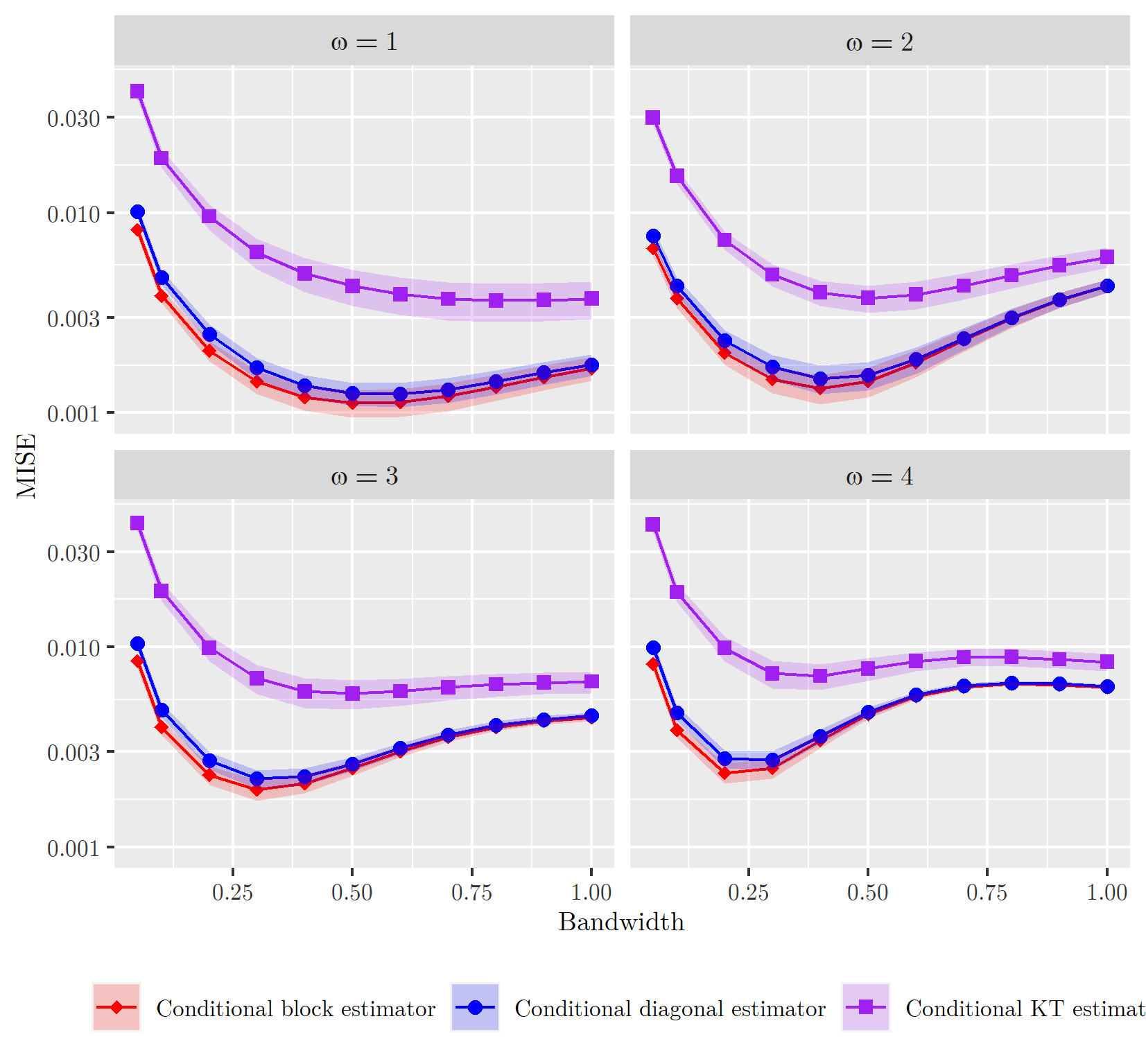}
    \caption{Log-plots of the conditional estimators' MISEs as a function of the bandwidth for different frequencies $\omega$ including 95\% confidence intervals, for a sample size of 200 and a block size of 8. ``Conditional KT estimate'' refers to the naive estimator of conditional Kendall's tau $\widehat{\tau}_{1,2\mid \Z=\z}$}.
    \label{fig:MISE_vs_h}
\end{figure} 

Let us compare the estimators' MISEs for different bandwidths. In this experiment, we set the diagonal block Kendall's taus to 0.3 and the off-diagonal block Kendall's taus conditionally at $Z = z$ to 
$$0.1(\cos(0.5 \pi \omega z) + 1),$$
with frequencies $\omega \in \{1,2,3,4\}$. The block size is fixed at $8$ and the sample size at $200$. The MISEs and 95\% confidence intervals are estimated using 100 replications, see Figure \ref{fig:MISE_vs_h}.  

\medskip

The figure confirms that indeed the averaging estimators have smaller optimal bandwidths than the naive estimator.
It should be noted that only a block size of $8$ is used here, and that the optimal bandwidth decreases with block size until the limit values are reached.
Furthermore, the figure shows that as the frequency increases, the optimal bandwidth is reduced. 
This is fully consistent with kernel regression theory: increasing the frequency increases the difference in Kendall's tau values conditionally on adjacent points of $z$, and therefore we need to pick a smaller bandwidth.
Lastly, it should be noted that as the bandwidth increases the effect of averaging is less and less visible. 
This can be attributed to the fact that by increasing the bandwidth, the variance term within the MISE becomes less and less prominent, while the bias term generally increases.

\section{Application to Real Data}
\label{sec:RD}

In this section, we study the behavior of the estimators under real data conditions and provide value at risk (VaR) computations of a large stock portfolio as an example of possible applications.
In Section~\ref{subsec:estpro}, we describe the methods used to estimate the VaR input parameters. 
The results are presented in Section~\ref{subsec:results}, where backtesting is applied to asses the viability.
All computations have been done using the \texttt{R} statistical environment \cite{Renv}.

\subsection{Value at Risk for Elliptical Distributions}
\label{subsec:VaRtheory}

The Value at Risk (VaR) is a widely used risk measure in a variety of financial fields, ranging from auditing and financial reporting to risk management and the calculation of regulatory capital \cite{VaRintro}. 
It is used to quantify potential losses over a specific time frame of some financial entity or portfolio of assets.
We will follow the approach of \cite{pimenova2012semi,VaRell}, in which explicit expressions for the VaR of elliptical distributions was derived. For the reader's convenience, we recall these expressions in the present section.

\medskip

Let $X$ be the loss of a given portfolio, i.e. $X > 0$ means that the portfolio manager is loosing $X$ euros.
The VaR at level $\alpha \in (0,1)$ is defined as the quantile of $X$ at level $(1 - \alpha)$.
To estimate the VaR of a given portfolio of assets, it is often assumed that the portfolio's profits and losses are a linear function of the returns of the individual constituents.
More formally, a portfolio with value $\Pi(t)$ at time $t$ is called linear if its profit and loss $\Delta\Pi(t) = \Pi(t)-\Pi(0)$ over a time window $[0, \mathrm{t}]$ is a linear function of the returns $X_{1}(t), \ldots, X_{p}(t)$:
$$\Delta \Pi(t)=\delta_{1} X_{1}(t)+\delta_{2} X_{2}(t)+\ldots+\delta_{p} X_{p}(t).$$ 
This clearly applies to any common stock portfolio by using the ordinary returns of the individual shares and when considering the log returns, this holds to a good approximation provided that the time window $[0,t]$ is small, e.g. for daily log returns.
The time window $t$ will be kept constant and will therefore be omitted from future notations.

\medskip

Furthermore, we will assume that the $X_j$ are elliptically distributed with mean $\bs{\mu}$, covariance matrix $\bs{\Sigma}$ with Cholesky decomposition $\bs{\Sigma} = \mb{A}\mb{A}^T$ and density generator $g$.
Thus, the probability density function $f_{\X}$ of
$\X = (X_1, \ldots, X_p)$ is given by
$$f_{\X}(\x) = \lvert \bs{\Sigma} \rvert^{-1/2} g\left( (\x -\bs{\mu})^T \bs{\Sigma}^{-1} (\x -\bs{\mu}) \right). $$
When considering elliptically distributed risk factors, we cannot simply use the Delta-Normal approach to calculate the VaR, as it relies on the stronger assumption of normality. 
A generalization of the Delta-Normal method was derived for the class of elliptical distributions in \cite{VaRell}.

\medskip

Let us start by noting that the VaR of the portfolio profits and losses $\Delta \Pi(t)$ can be rewritten as
$\mathbb{P}\left( \Delta \Pi < -\operatorname{VaR}_{\alpha} \right) = \alpha.$
Then, given the linearity of the portfolio and the fact that $\X$ follows an elliptical distribution, the VaR is obtained by solving the following equation:
\begin{equation*}
    \alpha=|\bs{\Sigma}|^{-1 / 2} \int_{\left\{\bs{\delta}  \x^T \leq-\operatorname{VaR}_{\alpha}\right\}} g\left((\x-\bs{\mu}) \bs{\Sigma}^{-1}(\x-\bs{\mu})^{T}\right) \mathrm{d} \x,
\end{equation*}
where $\bs{\delta}$ denotes the vector of weights $(\delta_1, \ldots, \delta_p)$.
After several changes of variables, we obtain
\begin{equation}
    \alpha = \left| S_{p-2} \right| \int_{0}^{\infty} r^{p-2} \int_{-\infty}^{\frac{-\bs{\delta}\bs{\mu}^T-\operatorname{VaR}_{\alpha}}{|\bs{\delta}\mb{A}| } } g\left(z_1^2 + r^2\right)\mathrm{d} z_1 \mathrm{d} r, \label{eq:trEq1}
\end{equation}
where $\left| S_{p-2} \right| := 2\pi^\frac{p-1}{2} / \Gamma(\frac{p-1}{2}).$
Let us now introduce the function
\begin{align}
\label{eq:trEq}
    G(s)
    &= \frac{2\pi^\frac{p-1}{2}}{\Gamma(\frac{p-1}{2})} \int_{-\infty}^{-s} \int_0^\infty r^{n-2} g(z_1^2 + r^2) \mathrm{d}r \mathrm{d}z_1
    = \frac{\pi^\frac{p-1}{2}}{\Gamma(\frac{p-1}{2})} \int_{s}^{\infty} \int_{z}^\infty (u-z^2)^{\frac{p-3}{2}} g(u) \mathrm{d}u \mathrm{d}z,
\end{align}
where we have changed variables to $u = r^2 + z_1^2$ and $z = - z_1$. 
Let us denote by $q_{\alpha,p}^g$ the unique solution of the transcendental equation
\begin{equation}
    \alpha = G(q_{\alpha,p}^g).
\end{equation}
It then finally follows from expressions \eqref{eq:trEq1} and \eqref{eq:trEq} that the Delta-Elliptic VaR is given by
\begin{align}
    \operatorname{VaR}_{\alpha}
    &= -\bs{\delta}\bs{\mu}^T + q_{\alpha,p}^g |\bs{\delta} \mb{A}| 
    = -\bs{\delta}\bs{\mu}^T + q_{\alpha,p}^g \sqrt{\bs{\delta} \bs{\Sigma} \bs{\delta}^T}. \label{eq:VaRcalc}
\end{align}
Note that this equation has a clear financial interpretation: the portfolio's average return is given by $\bs{\delta}\bs{\mu}^T$ and the portfolio's standard deviation by $ \sqrt{\bs{\delta} \bs{\Sigma} \bs{\delta}^T} $. Further note that the result is analogous to that of the Delta-Normal VaR, in which we simply replace $q_{\alpha,p}^g$ with the $1-\alpha$ quantile of the standard-normal distribution.

\subsection{Estimation Procedure}
\label{subsec:estpro}

In order to test the estimators in real data conditions, we consider a portfolio consisting of $240$ different stocks. All stocks are listed on the Euronext markets and data has been downloaded from Yahoo Finance. The complete list of all shares involved is available in Appendix~\ref{sec:stocksLists}.
We will estimate the portfolio's daily VaR assuming that the price is set at a level of 100 and that all stocks in the portfolio are equally weighted.
To this end, we model the daily log returns of the individual stocks, assuming they follow an elliptical distribution.

\medskip

In order to achieve a proper clustering, we compute the pairwise Kendall's tau matrix over a long time period from 01 January 2007 to 14 January 2022, after which we reorder the variables in order to obtain the intended block structure.
Since we have not proposed a clustering method, we simply use the method \texttt{GW\_Ward} method from package \texttt{seriation} \cite{seriation}, along with a few manual adjustments. 
The resulting reordering corresponds to four large groups, which are specified further in Appendix~\ref{sec:stocksLists}. 
See Figure~\ref{fig:clustering} for a heatmap of the pairwise Kendall's tau matrix before and after reordering the variables by group. 
To indicate the groups, lines have been drawn around the diagonal blocks.
It should be noted that, if studied carefully, the large groups can be broken down into smaller and more accurate groups.
Nevertheless, these large groups already seem to be quite useful and therefore we will simply use them for our further analysis.

\medskip

Based on the groups displayed in Figure~\ref{fig:clustering}, the objective is to compute the VaR at 30 June 2017, leaving sufficient future data for backtesting the results. 
To this end, we estimate the Kendall's tau matrix of the log returns using the block, row, diagonal and sample Kendall's tau matrix estimators using data points over the period 01 August 2015 to 30 June 2017. To estimate the standard deviations and averages over the same period, we use the sample mean and sample standard deviation.

\medskip 

Following the elliptical assumption, we can now obtain covariance matrix estimates from each of the Kendall's tau matrix estimates.
Subsequently, we can compute nonparametric estimates of the density generator for each of these inputs.
To this end, we make use of the function \texttt{EllDistrEst} from the \texttt{ElliptCopulas} package \cite{ElliptCopulas} which implements Liebscher's procedure \cite{liebscher}.

\medskip 

For the density generator estimation we require a complete data set with no missing values. 
As such, the interval on which we estimate the density generator will be chosen as shorter (01 June 2016 to 30 June 2017). 
The kernel function will be chosen as the Epanechnikov kernel.
Choice of tuning parameters in this setting is discussed in \cite{ryan2024choice}.
For simplicity, we use Silverman’s rule of thumb for bandwidth selection to estimate elliptical density generators~\cite{pimenova2012semi}, which for a sample size of $n$ is given by 
\begin{equation}
    h = 1.06\sqrt{ \operatorname{Var}\left[ \widehat{\bs{\xi}} \right] n^{1/5}},
\end{equation}
where 
$$\widehat{\xi}_i = -1 + \left(1+((\x_i-\bs{\mu}) \widehat{\bs{\Sigma}}^{-1} (\x_i-\bs{\mu})   )^{p / 2}\right)^{2 / p},$$ 
for $i = 1, \ldots, n$ and $p = 240$.
Here, $\x_i$ stands for the vector of log returns at the $i$th date, $\widehat{\bs{\Sigma}}$ stands for one of the covariance matrix estimates and $\bs{\mu}$ stands for the log returns' sample mean. 
Clearly, by using this bandwidth selection method, the use of different Kendall's tau matrix estimators yields different values for the bandwidth. 
In order to get a better idea of the effects of the bandwidth choice, we also consider several deterministic bandwidths, and compare the performance of the estimators for each of them.  

\medskip

Finally, we can numerically solve the transcendental equation as given in \eqref{eq:trEq1} to arrive at the corresponding quantiles. 
As such, we have discussed all ingredients for calculating the VaR as in \eqref{eq:VaRcalc}.
In order to test the results, we perform backtesting on two intervals, one in the future from 01 July 2017 to 14 January 2022 and one during the period on which the estimations are based, from 01 August 2015 to 30 June 2017.

\subsection{Results}
\label{subsec:results}

We compute the portfolio's 5\% and 10\% VaR values by following the estimation procedure described in Section \ref{subsec:estpro}.  
Table \ref{tab:quant} shows the quantile estimates obtained by solving the transcendental equation for each of the different density generator estimates. 
The density generators were estimated using each of the block, row, diagonal and naive Kendall's tau matrix estimators and using varying values of the bandwidth.

\medskip

The table shows that the averaging estimators yield very similar quantiles which are all relatively constant for different choices of the bandwidth. 
In contrast, the quantiles of the naive estimator lie substantially higher and vary significantly for the different bandwidths. 
In that sense, the estimates obtained with the averaging estimators seem to be much more stable.
Moreover, the Silverman's bandwidths of the averaging estimators are also all very similar, while that of the naive estimator is again considerably larger.

\begin{table}[hbtp]
  \begin{center}
 \caption{Estimated quantiles corresponding to the 5\% and 10\% VaRs calculated by estimating an elliptical distribution for the daily log returns using each of the different Kendall's tau matrix estimates and several values of the bandwidth.}
  \label{tab:quant}
\noindent \begin{tabular}{ll|cccc}  \toprule \toprule
     \multicolumn{2}{l|}{\textbf{Quantiles} $\bs{q_{\alpha, p}^{\widehat{g}_h}}$} & \multicolumn{4}{l}{Estimated}  \\
    $\alpha$ & Estimator & {$h = 20$} & {$h = 40$} & {$h = 100$} & {Silverman's $h$}   \\ \midrule
    \multirow{4}*{5\%}  & Naive & 2.11 & 1.94 & 1.98 & 2.12 ($h=586.8$)   \\
      & Block  & 1.60 & 1.60  & 1.60 & 1.60 ($h = 40.8$)   \\
      & Row  & 1.60 & 1.60  & 1.60 & 1.60 ($h = 41.1$)    \\
      & Diagonal & 1.59 & 1.59  & 1.60 & 1.59 ($h = 40.5$)         \\ \midrule
    \multirow{4}*{10\%}  & Naive   & 1.48 & 1.38   & 1.40 & 1.53 ($h=586.8$)     \\
     & Block  & 1.23 & 1.23  & 1.23 & 1.23 ($h = 40.8$)     \\
     & Row  & 1.24 & 1.24  & 1.23 & 1.24 ($h = 41.1$)      \\
     & Diagonal  & 1.23 & 1.23  & 1.23 & 1.23 ($h = 40.5$)        \\  \bottomrule \bottomrule
\end{tabular}
\end{center}
\end{table}

\medskip

Table \ref{tab:VaR} shows the VaR estimates for each of the different estimators and bandwidths, and also the backtested VaR values. As discussed in Section \ref{subsec:estpro}, backtests were conducted at two intervals, interval 1 refers to the upcoming interval from 01 July 2017 until 14 January 2022, and interval 2 refers to the interval on which the estimation is based, from 01 August 2015 until 30 June 2017.

\begin{table}[hbtp]
  \begin{center}
\caption{Estimated 5\% and 10\% VaRs including the corresponding backtesting results on two intervals. Interval 1 corresponds to 01 July 2017 until 14 January 2022 and interval 2 to 01 August 2015 until 30 June 2017.}
     \label{tab:VaR}
 \begin{tabular}{ll|cccc|cc} \toprule \toprule
     \multicolumn{2}{l|}{$\bs{\operatorname{VaR}}$} & \multicolumn{4}{l|}{Estimated} & \multicolumn{2}{l}{Backtested} \\
    $\alpha$ & Estimator & {$h = 20$} & {$h = 40$} & {$h = 100$} & {Silverman's $h$} & {Interval 1} & {Interval 2} \\ \midrule
    \multirow{4}*{5\%}  & Naive & 1.647 & 1.512 & 1.544 & 1.655 ($h=586.8$) & \multirow{4}*{1.392} & \multirow{4}*{1.262} \\
      & Block  & 1.320 & 1.320  & 1.320 & 1.320 ($h=40.8$)\\
      & Row  & 1.332 & 1.332 & 1.332 & 1.332 ($h=41.1$)\\
      & Diagonal & 1.284 & 1.284 & 1.292 & 1.284 ($h=40.5$)  \\ \midrule 
   \multirow{4}*{10\%}  & Naive   & 1.147 & 1.083 & 1.068 & 1.187 ($h=586.8$)  & \multirow{4}*{0.861} & \multirow{4}*{0.839} \\
     & Block  & 1.008 & 1.008 & 1.008 & 1.008 ($h=40.8$)\\
     & Row  & 1.026 & 1.017  & 1.026 & 1.026 ($h=41.1$)\\
     & Diagonal  & 0.987 & 0.987 & 0.987 & 0.987 ($h=40.5$)\\ \bottomrule \bottomrule 
\end{tabular}
\end{center}
\end{table}

\begin{table}[htbp]
  \begin{center}
  \caption{The number of exceedances of the estimated 5\% and 10\% VaRs during backtesting interval 1, from 01 July 2017 until 14 January 2022.}
  \label{tab:exc1}
 \begin{tabular}{ll|cccc|c} \toprule \toprule
     \multicolumn{2}{l|}{\textbf{\# Exceedances}  } & \multicolumn{4}{l|}{Estimated} & \multicolumn{1}{l}{Backtested} \\
    $\alpha$ & Estimator & {$h = 20$} & {$h = 40$} & {$h = 100$} & {Silverman's $h$} & {Interval 1}  \\ \midrule
    \multirow{4}*{5\%}  & Naive & 47  & 53 & 53 &  46 ($h=586.8$) & \multirow{4}*{58}   \\
      & Block  & 58 & 58 & 58 & 58 ($h=40.8$)  \\
      & Row  & 58 & 58 & 58 & 58 ($h=41.1$) \\
      & Diagonal   & 61 & 61 & 59 & 61 ($h=40.5$) \\ \midrule 
    \multirow{4}*{10\%}  & Naive   & 76 & 85  & 83 & 72 ($h=586.8$) & \multirow{4}*{116} \\
     & Block  & 94 &  94 & 94 & 94 ($h=40.8$) \\
     & Row  & 91 & 91 & 92 & 91 ($h=41.1$)\\
     & Diagonal  & 100 & 100 & 100 & 100 ($h=40.5$)       \\  \bottomrule \bottomrule
\end{tabular}
\end{center}
\end{table}

\begin{table}[htbp]
  \begin{center}
  \caption{The number of exceedances of the estimated 5\% and 10\% VaRs during backtesting interval 2, from 01 August 2015 until 30 June 2017.}
  \label{tab:exc2}
\noindent \begin{tabular}{ll|cccc|c} \toprule \toprule
     \multicolumn{2}{l|}{\textbf{\# Exceedances}  } & \multicolumn{4}{l|}{Estimated} & \multicolumn{1}{l}{Backtested} \\
    $\alpha$ & Estimator & {$h = 20$} & {$h = 40$} & {$h = 100$} & {Silverman's $h$} & {Interval 2}  \\ \midrule
    \multirow{4}*{5\%}  & Naive & 9  & 12 & 12 &  9 ($h=586.8$) & \multirow{4}*{25}   \\
      & Block  & 20 & 20 & 20 & 20 ($h=40.8$)  \\
      & Row  & 20 & 20 & 20 & 20 ($h=41.1$)   \\
      & Diagonal & 22 & 22 & 21 & 22 ($h=40.5$)   \\ \midrule
    \multirow{4}*{10\%}  & Naive   & 30 &  32 & 32 & 30 ($h=586.8$) &  \multirow{4}*{49}   \\
     & Block  & 35 & 35  & 35 & 35 ($h=40.8$)     \\
     & Row  & 35 & 35 & 35 & 35 ($h=41.1$)  \\
     & Diagonal  & 35 & 35 & 35 & 35 ($h=40.5$)        \\  \bottomrule \bottomrule
\end{tabular}
\end{center}
\end{table}

\medskip

This clearly shows that the averaging estimators have performed significantly better than the naive estimator when compared to both backtesting intervals.
For both $\alpha$-levels, it can be seen that the VaRs generated using the naive estimator are considerably larger than those using the averaging estimators, which themselves produce relatively similar values. 
Furthermore, it can be seen that the 5\% VaRs of the averaging estimators agree fairly well with the results of the backtesting, unlike those of the naive estimator.

\medskip

However, the 10\% VaR estimates are not as accurate and all estimators yield considerably higher VaRs than those obtained by backtesting. 
This could indicate that the log returns are not elliptically distributed, or that the interval at which we estimate the density generator is too short.
Recall that the interval on which we estimate the density generator is merely from 01 June 2016 until 30 June 2017.
This lack of performance is hard to relate directly to the block sizes. Indeed, the ``naive'' estimator of Kendall's tau
(i.e., without any averaging)
corresponds to the case where all the blocks have size 1, so all block sizes are as small as possible. Still the VaR estimates from this estimator are the worst.

\medskip

To get a better understanding of how well the VaR estimates correspond with the backtesting results, we examine how often the estimates are exceeded by the portfolio's losses in each of the backtesting periods. Table \ref{tab:exc1} and \ref{tab:exc2} show the number of exceedances in interval 1 and interval 2 respectively. 

\medskip

Both tables show that the difference between the theoretical and the observed number of exceedances is much larger when using the naive sample Kendall's tau matrix estimator than when using any of the averaging estimators and this applies to both $\alpha$-levels as well as to all bandwidths. 
As such, the averaging estimators are overall significantly better performers than the naive estimator. 
In addition, although there are subtle differences in the performance of the block, row and diagonal estimators, there is no clear winner in this example. 
This shows that computing all Kendall's tau using the block estimators incur no clear additional benefits compared to using only the row or diagonal estimators, that are computationally much cheaper.

\section{Conclusion}
\label{sec:conclusions}

We have provided an alternative approach to the generally challenging task of estimating Kendall's tau and conditional Kendall's tau matrices in high-dimensional settings. By imposing structural assumptions on the underlying (conditional) Kendall's tau matrix, we have introduced new estimators that have significantly reduced computational costs without much loss in performance.

\medskip

For the unconditional case, a model was studied in which the set of variables could be grouped in such a way that the Kendall's taus of variables from different groups depend only on the group numbers.
After reordering the variables by group, the underlying Kendall's tau matrix is then block-structured with constant values in the off-diagonal blocks.
We have proposed several (unbiased) estimators that take advantage of this block-structure by averaging over the usual pairwise Kendall's tau estimates in each of the off-diagonal blocks: the block estimator averages over all pairwise estimates, whereas the row, the diagonal and the random estimators only average over part of the off-diagonal blocks (respectively, over the pairs on the first row, on the first diagonal and over a random selection of pairs). 

\medskip

We have formally derived variance expressions, which showed not only that all estimators are improvements over the usual sample Kendall's tau matrix estimator, but also, interestingly, that the asymptotic variances do not depend on the block dimensions.  
Furthermore, we have seen that the block, the diagonal and the random estimators have similar asymptotic variances, whereas that of the row estimator was different.
In most examples, the diagonal estimator performed the best,
but a formal characterization of the set of such copulas is left for future work.
Under light assumptions, we have shown that asymptotic variances are equal, and that it is approached fastest by the block estimator, followed by the diagonal estimator and then the random estimator.
Hence, if the computational costs were to be reduced, the diagonal estimator is preferable to both the random and the row estimator.

\medskip

Furthermore, a model was studied in which the Kendall's taus depend on a conditioning variable.
Here it was assumed that the conditional Kendall's tau matrix has the above-mentioned block structure and, moreover, that it is preserved under fluctuations of the conditioning variable. 
We have adopted nonparametric, kernel-based estimates of the conditional Kendall's tau in order to construct the conditional versions of the block, row, diagonal and random estimators.  
Under some additional regularity assumptions, we have shown that the estimators are all asymptotically normal conditionally to different values of the covariate.
Following from these expressions, we have seen that the asymptotic variances have analogous expressions to their unconditional counterparts.
As such, all estimators are again improvements over the naive estimator, with the block estimator having the best performance. 
Similarly, if computational costs were to be reduced, the diagonal estimator is preferable to both the random and the row estimator.
Moreover, the reduction of computing costs becomes particularly relevant in the conditional setting, as the use of kernel smoothing introduces additional complexity.

\medskip

We have performed a simulation study in order to support the theoretical findings.
In the unconditional setting, simulations were performed with different meta-elliptical copulas. 
It was furthermore confirmed that the diagonal and the block estimator indeed have the lowest asymptotic variance in most cases, with the block estimator converging the fastest, though closely followed by the diagonal estimator. 
This emphasizes the practical use of the diagonal estimator.

\medskip

We remarked again that the conditional estimators' variances decrease in a similar fashion for growing block dimensions.
As a consequence, the averaging estimators allow for a reduced optimal bandwidth; this was indeed confirmed in the simulations.
This makes the averaging estimators perfectly suited for practical applications, as reducing the bandwidth goes hand in hand with reducing the estimation bias.  

\medskip

Lastly, we have demonstrated the use of the estimators in a real world application.
The estimators were used to model the daily log returns of a large stock portfolio consisting of 240 Euronext listed stocks.
After clustering the sample Kendall's tau matrix, the proposed block structure was clearly visible.
Building on these groups, robust estimates of the correlation matrix were obtained by assuming that the log returns follow an elliptical distribution.
Using each of these estimates, the portfolio's 5\% and 10\% VaR values were estimated.
The results of the averaging estimators were much more stable under changes in the bandwidth used for the estimation of the density generator.
Moreover, the averaging VaRs were significant improvements over the naive estimates.
This example confirmed that the proposed block structures are well reflected in real data conditions and that the averaging estimators lead to significantly more stable and accurate results. 

%

\bigskip

\textbf{Acknowledgments.} The authors thanks Thomas Nagler for useful comments on a previous draft, and Dorota Kurowicka for a discussion and references that lead to Section 2.2. The authors also thank the Associate Editor and two anonymous reviewers for their useful comments which significantly improved the manuscript.


\bigskip

\bibliographystyle{abbrv}
\bibliography{main}{}

\bigskip

\newpage

\appendix

\section{Additional figures}
\label{sec:additional_fig}




\begin{figure}[!hb]
    \centering
    \includegraphics[width = 0.99\textwidth]{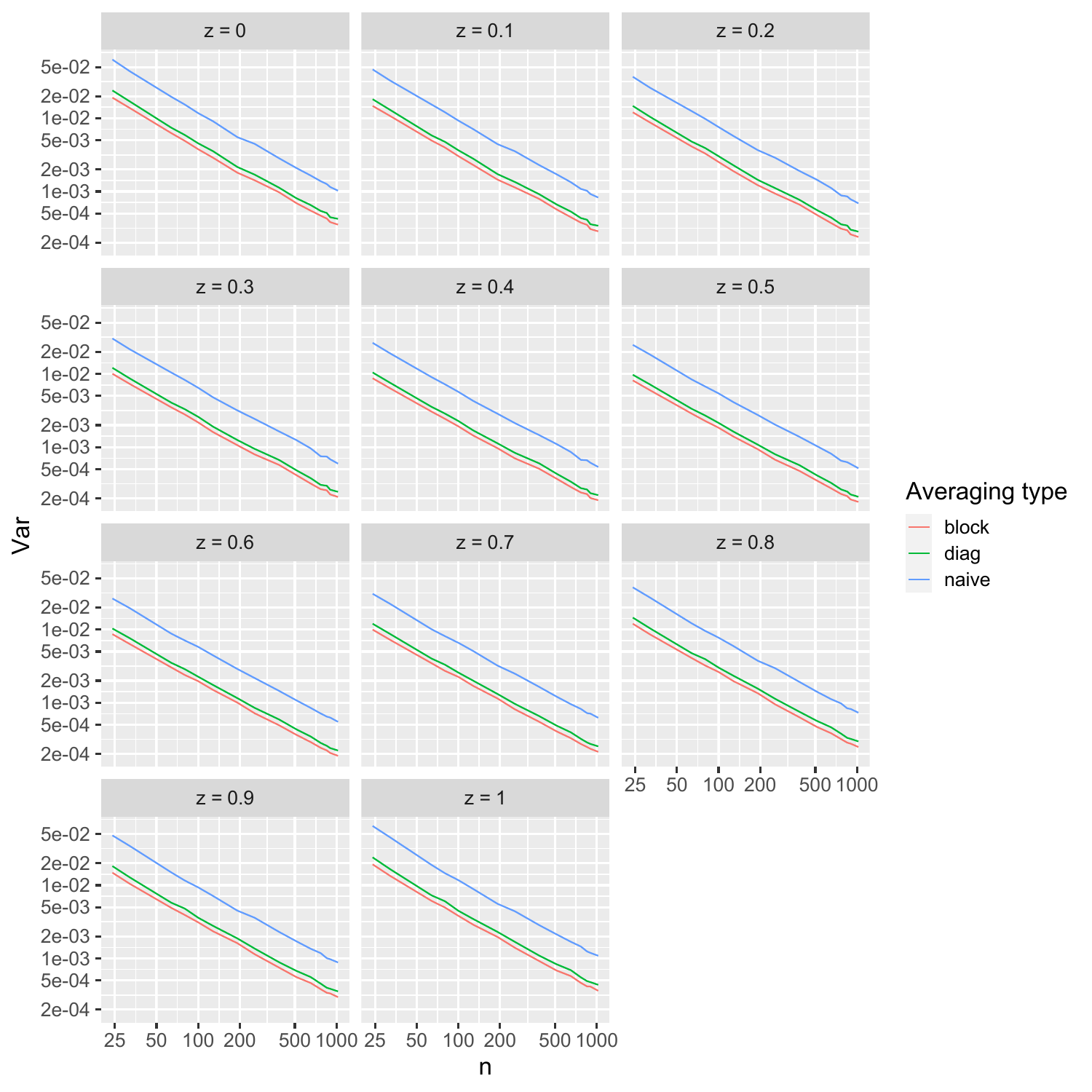}
    \caption{Log-log plots of the conditional estimators' variances as a function of the sample size on several conditioning points, using a block size of $4$ and a bandwidth of 0.5.}
    \label{fig:MSE_vs_dim_conditional_wholegrid}
\end{figure}

\subsection{Effect of the block size in the conditional framework (Section~\ref{sec:simstuCKT}) }
\label{subsubsec:condblocksize}

We first study the estimators' variance under varying block dimensions. In order to run a sufficient number of replications we set the sample size to 20 and consequently the bandwidth to 0.5. The variances and 95\% confidence intervals are estimated using 30000 replications.
For each grid point $z$, the resulting sample variances are displayed on log-log scale in Figure~\ref{fig:MSE_vs_dim_conditional_gauss}.

From the figure we observe that the estimators' variances behave similarly to the unconditional setting under varying block dimensions, for each of the grid points. That is, both estimators are improvements over the naive estimator, both limiting variances are identical, and the block estimator converges slightly faster than the diagonal estimator.
It further follows that since averaging reduces variance, it also reduces the optimal bandwidth.
This is studied in Section \ref{subsubsec:bandwidthselection}.
Again, as there are fewer observations of $Z$ near grid points close to the edges of $[0,1]$, the variance levels vary slightly over the different grid points.

\begin{figure}[hbtp]
    \centering
    \includegraphics[scale = 0.9]{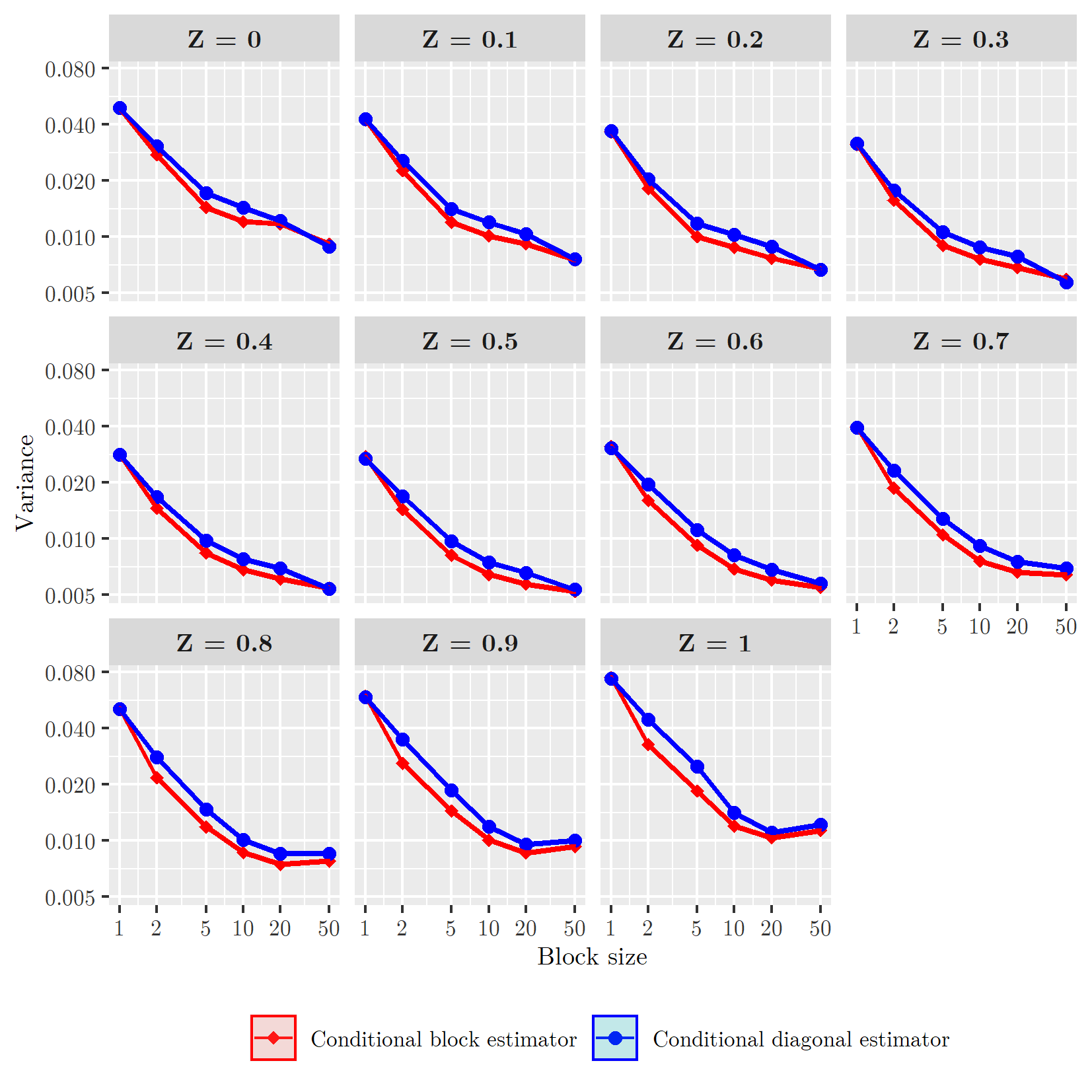}
    \caption{Log-log plots of the conditional estimators' variances as a function of the block size on several conditioning points including 95\% confidence intervals, for a sample size of $20$ and a bandwidth of 0.5.}
    \label{fig:MSE_vs_dim_conditional_gauss}
\end{figure} 

As for the computation times, there is clearly no fundamental change in how these depend on the block size when compared to the unconditional setting.
However, since the conditional estimators are kernel-based, it should be noted that they generally require more computation time than their unconditional counterparts, as was also seen in Figure \ref{fig:Comp_time_vs_timepoints_conditional}. 
For the sake of completeness, we still include a plot of the average computation time against the block size, see Figure \ref{fig:Comp_time_vs_dim_conditional}. 
The results correspond to estimating the off-diagonal block conditional Kendall's taus simultaneously on the $11$ grid points, and follow from $10000$ replications with a sample size of $150$. 
As expected, the block estimator scales quadratically with block size, while the diagonal estimator scales linearly with block size.
Therefore, as in the unconditional case, one may prefer the diagonal estimator over the block estimator to gain substantial computational efficiency and lose only little precision.

\begin{figure}[hbtp]
    \centering
    \includegraphics{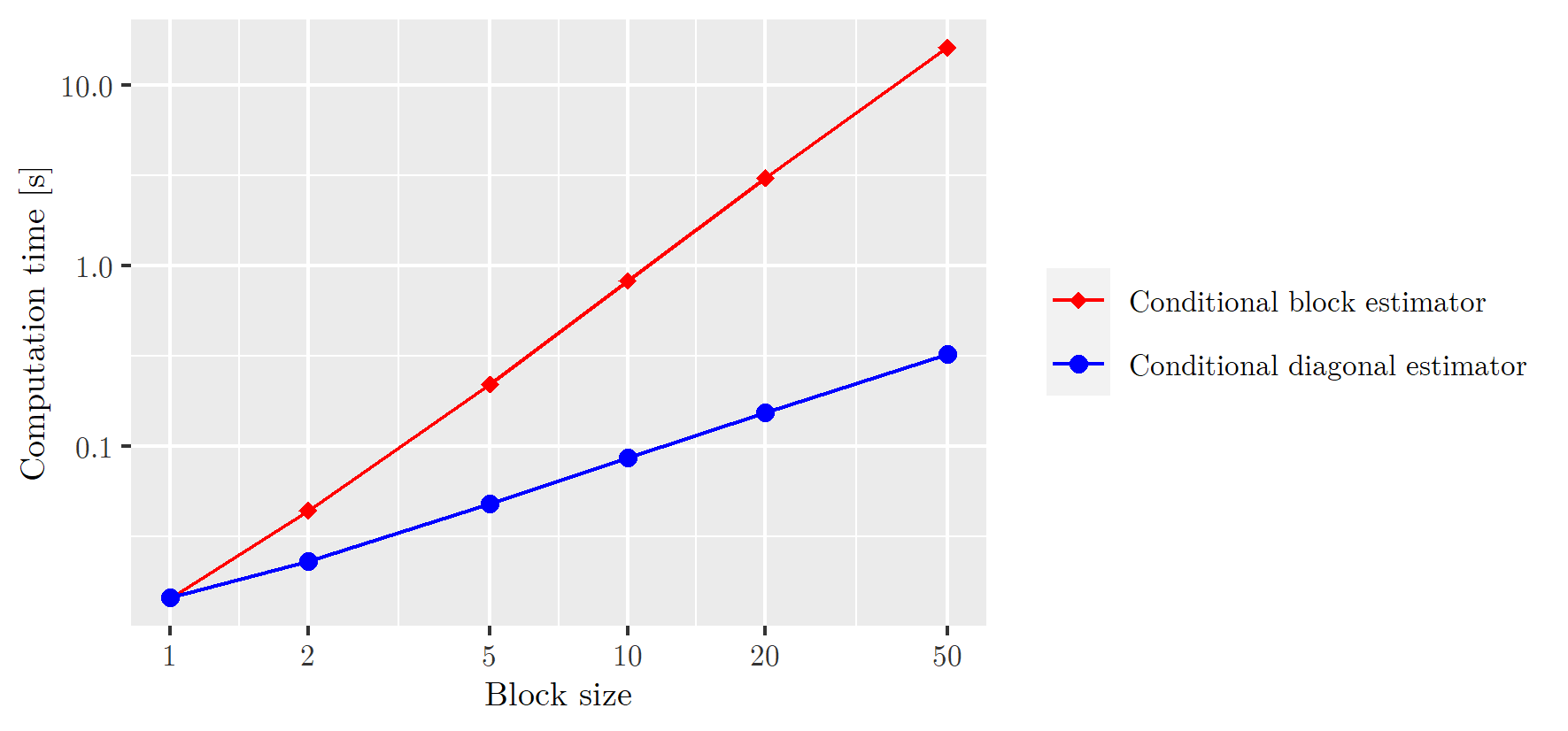}
    \caption{Log-log plot of the conditional estimators' mean computation time [s] as a function of the block size, for a sample size of $150$.}
    \label{fig:Comp_time_vs_dim_conditional}
\end{figure}

\begin{figure}[p]
    \centering
    \includegraphics[width = 0.99\textwidth]{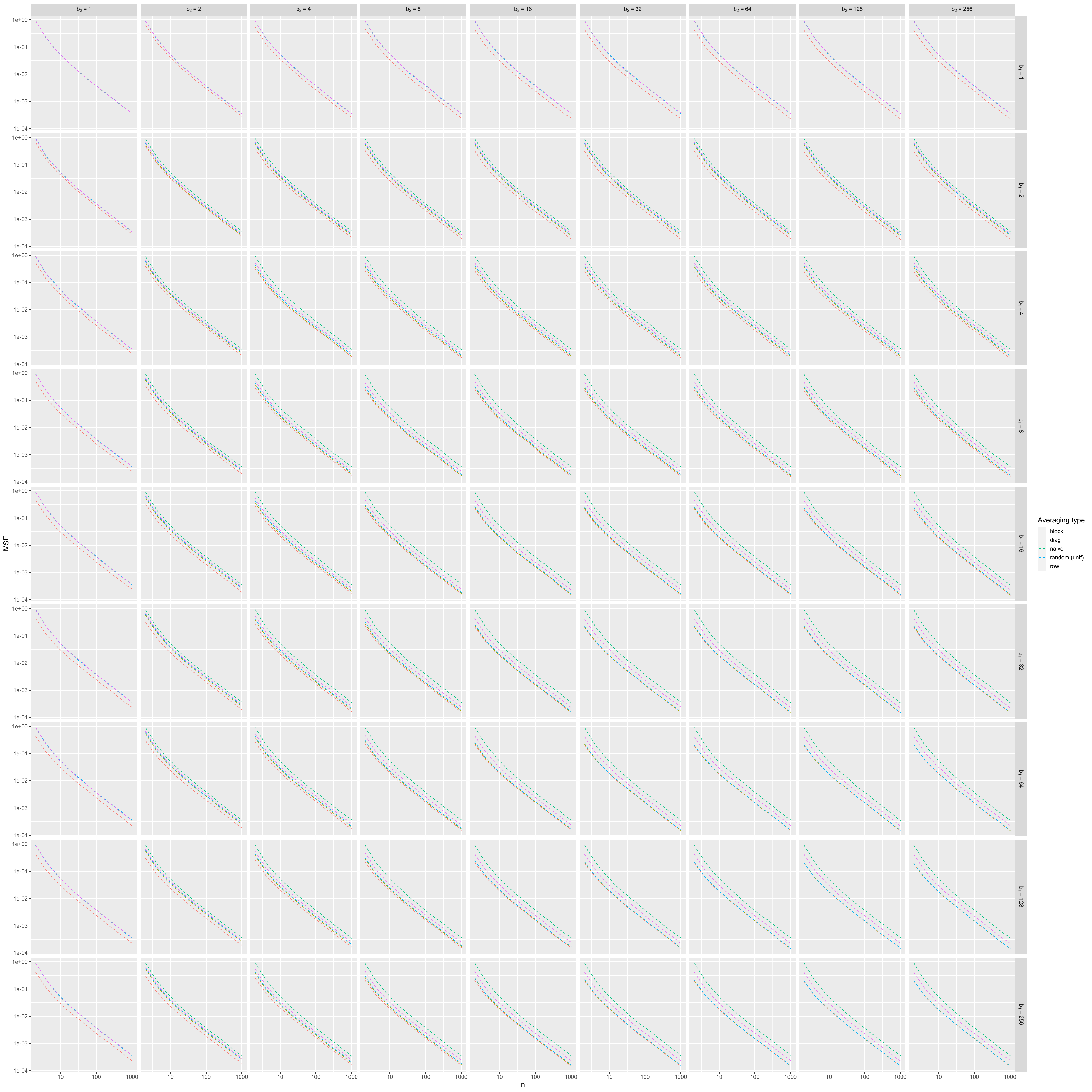}
    \caption{MSE of the unconditional estimators, for different combinations of block sizes}
    \label{MSEnb1b2}
\end{figure}

\begin{figure}[p]
    \centering
    \includegraphics[width = 0.99\textwidth]{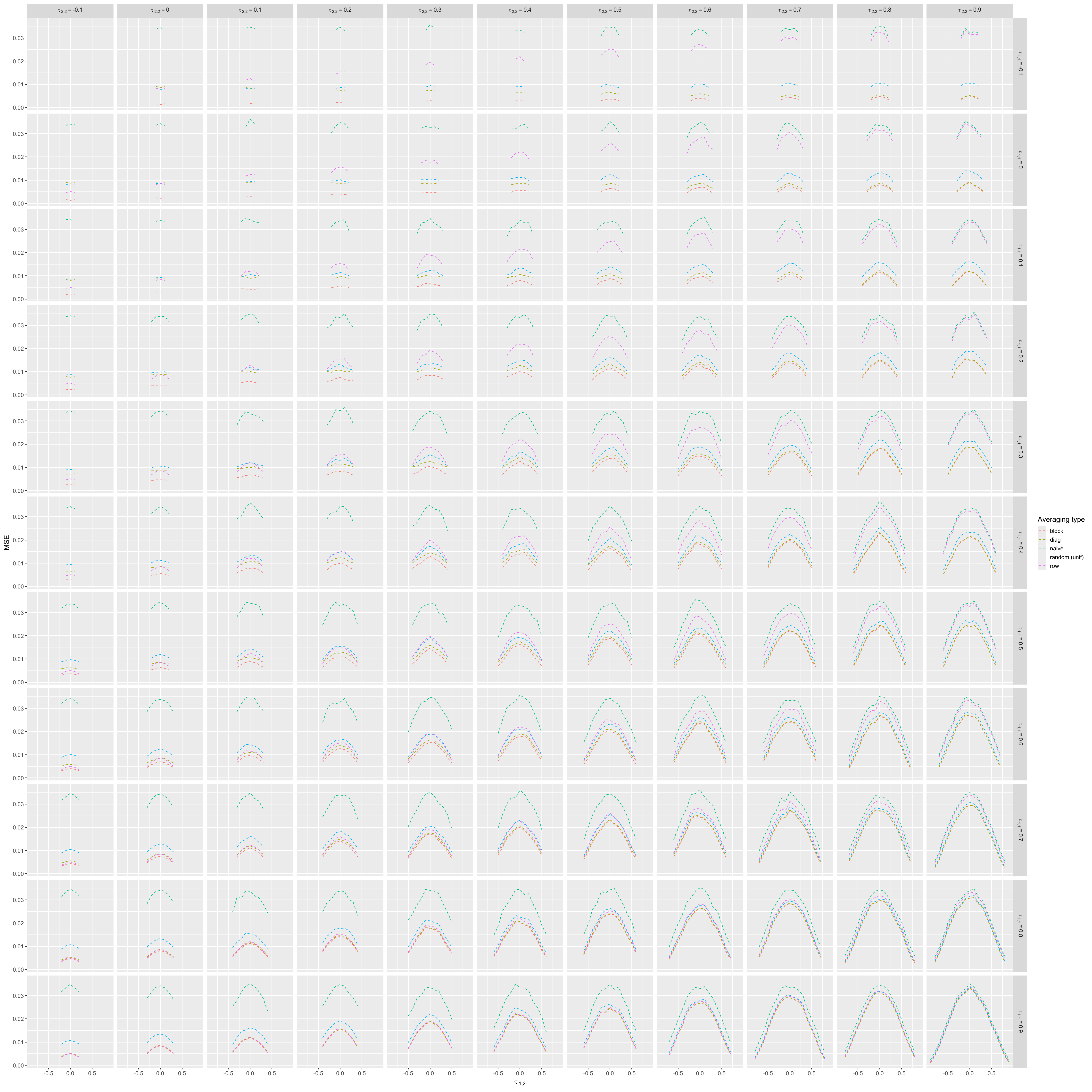}
    \caption{MSE as a function of the intergroup Kendall's tau}
    \label{MSEtau}
\end{figure}


\section{Proofs}

\subsection{Proofs for Section~\ref{sec:consequences_corMatrix}}
\label{proof:prop:correlation_matrix}

\begin{proof}[Proof of Proposition~\ref{prop:correlation_matrix}]
In this proof, we will need the following notation: for an integer $i \in \{1, \dots, p\}$, we define the vector $\1_{i}$ as the vector with a $1$ at the $i$-th component and $0$ elsewhere. For a set $I \subset \{1, \dots, p\}$, we define $\1_I := \sum_{i \in I} \1_i$ which the vector with $1$ at the components in $I$ and $0$ elsewhere.
Note that 
\begin{align*}
    &M (\1_{1}-\1_{2})
    = (1 - \rho_1, \rho_1 - 1, 0, \dots, 0) 
    = (1-\rho_1) (\1_{1}-\1_{2}) \\
    &M (\1_{b_1+1}-\1_{b_1+2})
    = (0, \dots, 0, 1 - \rho_2, \rho_2 - 1, 0, \dots, 0) 
    = (1 - \rho_2) (\1_{b_1+1}-\1_{b_1+2})
\end{align*}
This gives us a number of $(b_1 - 1) + (b_2 - 1)$ eigenvectors with eigenvalues $1 - \rho_1$ and $1 - \rho_2$ that are positive since $\rho_1, \rho_2$ are smaller than $1$.
Moreover, remark that
\begin{align*}
    &M \1_{1:b_1}
    = (1 + (b_1 - 1)) \rho_1 \1_{1:b_1}
    + b_1 \rho_3 \1_{b_1 + (1:b_2)} \\
    &M \1_{b_1 + (1:b_2)}
    = b_2 \rho_3 \1_{1:b_1}
    + (1 + (b_2 - 1) \rho_2) \1_{b_1 + (1:b_2)},
\end{align*}
so the eigenvalues of the matrix
\begin{align*}
    \left( \begin{array}{cc}
        (1 + (b_1 - 1) \rho_1 & b_1 \rho_3 \\
        b_2 \rho_3  & 1 + (b_2 - 1) \rho_2
    \end{array} \right)
\end{align*}
are also eigenvalues of $M$.
These eigenvalues are
\begin{align*}
    1+\frac{\left(b_{1}-1\right) \rho_{1}}{2}+\frac{\left(b_{2}-1\right) \rho_{2}}{2}
    \pm \frac{\sqrt{\left(\left(b_{1}-1\right) \rho_{1}-\left(b_{2}-1\right) \rho_{2}\right)^{2}+4b_{1} \rho_{3}^{2} b_{2}}}{2}
\end{align*}
The smallest eigenvalue is positive if and only if
\begin{align*}
    \left(1+\frac{\left(b_{1}-1\right) \rho_{1}}{2}+\frac{\left(b_{2}-1\right) \rho_{2}}{2} \right)^2 
    > \frac{\left(\left(b_{1}-1\right) \rho_{1}-\left(b_{2}-1\right) \rho_{2}\right)^{2}+4b_{1} \rho_{3}^{2} b_{2}}{4}.
\end{align*}
i.e.
\begin{align*}
    \Big(\big(b_{2} b_{1}-b_{1}-b_{2}+1\big) \rho_{2}
    + b_{1} - 1 \Big) \rho_{1} - b_{1} \rho_{3}^{2} b_{2} + (b_{2}-1) \rho_{2} + 1 > 0
\end{align*}
A sufficient condition is:
\begin{align*}
    &\rho_{1}
    > \frac{b_{1} b_{2} \rho_{3}^{2} - (b_{2} - 1) \rho_{2} - 1}{
    \left(b_{2} b_{1}-b_{1}-b_{2}+1\right) \rho_{2} + b_{1} - 1}, \\
    &\rho_{2}
    > - \frac{b_{1}-1}{b_{2} b_{1} - b_{1} - b_{2} + 1}.
\end{align*}
\end{proof}

\begin{proof}[Proof of Proposition~\ref{proof:lowerbound_rho_K}]
Assume that $M$ is a correlation matrix, and let $\X \sim \mathcal{N}(0, M)$. Take one random variable from each block. Their correlation matrix is $(I+\rho J)$ and must therefore be positive semidefinite. This yields the constraint $\rho \geq -1 / (K-1)$. 

\medskip

Conversely, this bound is reached by choosing the correlation matrices $\Sigma_k = \1$ for all $k = 1, \dots, K$ and considering $M$ as the correlation matrix of $(X_1, \dots, X_1, X_2, \dots, X_2, \dots, X_K, \dots, X_K) \in \Rb^{b_1 + b_2 + \dots + b_K}$, where $(X_1, \dots, X_K)$ follows an exchangeable normal distribution with correlation arbitrarily close to $-1/(K-1)$.
\end{proof}

\subsection{Derivation of Equation (\ref{eq:Q_copulaExpression})}
\label{proof:eq:Q_copulaExpression}

We have
\begingroup
\begin{align*}
    Q_{j_1,j_2} &=  \mathbb{P}\big( 
    \X_{1,(j_1,j_2)} < \X_{2,(j_1,j_2)} ,
    \X_{1,(j_1,j_2)} < \X_{3,(j_1,j_2)}  \big) \\
    &+  \mathbb{P}\big( \X_{1,(j_1,j_2)} < \X_{2,(j_1,j_2)} ,
    \X_{1,(j_1,j_2)} > \X_{3,(j_1,j_2)}  \big)   \\
    &+ \mathbb{P}\big( \X_{1,(j_1,j_2)} > \X_{2,(j_1,j_2)} ,
    \X_{1,(j_1,j_2)} > \X_{3,(j_1,j_2)}  \big) \\
    &+  \mathbb{P}\big( \X_{1,(j_1,j_2)} > \X_{2,(j_1,j_2)} ,
    \X_{1,(j_1,j_2)} < \X_{3,(j_1,j_2)}  \big)
\end{align*}
\endgroup
Let us write these probabilities in terms of the copula $C_{j_1,j_2}$ of $\X_{(j_1,j_2)} = (X_{j_1},X_{j_2})$. This gives
\begin{align*}
    Q_{j_1,j_2} &=  \int_{[0,1]^2} \int_{(u_1,u_2)}^{(1,1)} \int_{(u_1,u_2)}^{(1,1)} \mathrm{d}C_{j_1,j_2}(u_5,u_6)\mathrm{d}C_{j_1,j_2}(u_3,u_4) \mathrm{d}C_{j_1,j_2}(u_1,u_2)\\
    &+ \int_{[0,1]^2} \int_{(0,0)}^{(u_1,u_2)}  \int_{(u_1,u_2)}^{(1,1)} \mathrm{d}C_{j_1,j_2}(u_5,u_6) \mathrm{d}C_{j_1,j_2}(u_3,u_4) \mathrm{d}C_{j_1,j_2}(u_1,u_2)\\
    &+ \int_{[0,1]^2}   \int_{(u_1,u_2)}^{(1,1)} \int_{(0,0)}^{(u_1,u_2)} \mathrm{d}C_{j_1,j_2}(u_5,u_6) \mathrm{d}C_{j_1,j_2}(u_3,u_4) \mathrm{d}C_{j_1,j_2}(u_1,u_2)\\
    &+ \int_{[0,1]^2}   \int_{(0,0)}^{(u_1,u_2)} \int_{(0,0)}^{(u_1,u_2)} \mathrm{d}C_{j_1,j_2}(u_5,u_6) \mathrm{d}C_{j_1,j_2}(u_3,u_4) \mathrm{d}C_{j_1,j_2}(u_1,u_2) \displaybreak[0] \\
    &= \int_{[0,1]^2} \bar{C}_{j_1,j_2}^2(u_1,u_2) \mathrm{d}C_{j_1,j_2}(u_1,u_2)+ \int_{[0,1]^2} \bar{C}_{j_1,j_2}(u_1,u_2) C_{j_1,j_2}(u_1,u_2) \mathrm{d}C_{j_1,j_2}(u_1,u_2)\\
    &+ \int_{[0,1]^2}  C_{j_1,j_2}(u_1,u_2) \bar{C}_{j_1,j_2}(u_1,u_2)  \mathrm{d}C_{j_1,j_2}(u_1,u_2)+ \int_{[0,1]^2}   C_{j_1,j_2}^2(u_1,u_2) \mathrm{d}C_{j_1,j_2}(u_1,u_2) \displaybreak[0] \\
    &= \int_{[0,1]^2} (C_{j_1,j_2}(u_1,u_2) + \bar{C}_{j_1,j_2}(u_1,u_2))^2 \mathrm{d}C_{j_1,j_2}(u_1,u_2),
\end{align*}
as claimed.

\subsection{Proof of Lemma \ref{lem:Ustat}}
\label{proof:lem:Ustat}

\begin{proof}
First check that, trivially, the sample Kendall's tau $\widehat{\tau}_{j_1,j_2}$ is a U-statistic of order 2 with (symmetric) kernel
\begin{align*}
    g^*\left( \X_{i_1, (j_1, j_2)} ,
    \X_{i_2, (j_1, j_2)} \right)
    := \operatorname{sign}
    \left(\left(X_{i_1, j_1}-X_{i_2, j_1}\right)
    \left(X_{i_1, j_2}-X_{i_2, j_2}\right)\right).
\end{align*}
Consequently,
\begin{align*}
    \widehat{\tau}^B &= \frac{1}{b_1 b_2}
    \sum_{j_1 = 1}^{b_1} \sum_{j_2 = b_1 + 1}^p \widehat{\tau}_{j_1,j_2}\\
    &= \frac{1}{b_1 b_2}
    \sum_{j_1 = 1}^{b_1} \sum_{j_2 = b_1 + 1}^p
    \frac{2}{n(n-1)} \sum_{i_1<i_2}
    g^* \left(
    \X_{i_1, (j_1, j_2)} ,
    \X_{i_2, (j_1, j_2)} \right) \\
    &= \frac{2}{n(n-1)} \sum_{i_1<i_2}
    \frac{1}{b_1 b_2}
    \sum_{j_1 = 1}^{b_1} \sum_{j_2 = b_1 + 1}^p
    g^*\left( \X_{i_1, (j_1, j_2)} ,
    \X_{i_2, (j_1, j_2)} \right),
\end{align*}
and it follows that $\widehat{\tau}^B$ is a U-statistic with kernel
\begin{align*}
    g^{B} \left( \X_{i_1}, \X_{i_2} \right)
    = \frac{1}{b_1 b_2}
    \sum_{j_1 = 1}^{b_1} \sum_{j_2 = b_1 + 1}^p
    g^* \left( \X_{i_1, (j_1, j_2)} ,
    \X_{i_2, (j_1, j_2)} \right).
\end{align*}
In a similar manner, it is easily seen that $\widehat{\tau}^{R}$, $\widehat{\tau}^{D}$ and $\widehat{\tau}^{U}$ are all U-statistics as well with respective kernels
\begingroup
\allowdisplaybreaks
\begin{align*}
    g^{R} \left( \X_{i_1}, \X_{i_2} \right)
    &= \frac{1}{N} \sum_{j=1}^N 
    g^*\left( \X_{i_1, (1, j)} ,
    \X_{i_2, (1, j)} \right), \\
    g^{D} \left( \X_{i_1}, \X_{i_2} \right)
    &= \frac{1}{N} \sum_{j=1}^N 
    g^*\left( \X_{i_1, (j, b_1 + j)} ,
    \X_{i_2, (j, b_1 + j)} \right), \\
    g^{U} \left( \X_{i_1}, \X_{i_2} \right)
    &= \frac{1}{b_1 b_2}
    \sum_{j_1 = 1}^{b_1} \sum_{j_2 = b_1 + 1}^p
    W_{j_1,j_2} g^*\left( \X_{i_1, (j_1, j_2)} ,
    \X_{i_2, (j_1, j_2)} \right).
\end{align*}
\endgroup
Note that the kernel $g_{k_1,k_2}^{U}$ is random by depending on the weights $\W$.
\end{proof}

\subsection{Proof of Theorem~\ref{th:var}}
\label{proof:th:var}

Recall from Lemma \ref{lem:Ustat} that $\widehat{\tau}_{j_1,j_2},  \widehat{\tau}^{B}, \widehat{\tau}^{R},  \widehat{\tau}^{D}, \widehat{\tau}^{U}$ can all be written as U-statistics of order 2 with the symmetric kernels defined just above.
To compute the variance of these U-statistics, we will use Hoeffding's formula \cite{hoeffding} (see also \cite[Section 5.2.1]{serfling2009approximation}) that we recall for reader's convenience.
For a second-order U-statistic
$U_n := \binom{n}{2}^{-1} \sum_{1 \leq i_1 \neq i_2 \leq n} g(\X_{i_1},\X_{i_2})$ with symmetric kernel $g: \Rb^p \to \Rb$ satisfying
$\E \lvert g(\X_1, \X_2)\rvert < +\infty $,
the variance is given by
\begin{align}
\label{eq:varU}
    \Var[U_n] &= \binom{n}{2}^{-1} \sum_{c=1}^2 \binom{2}{c}\binom{n-2}{2-c}  \zeta_c \nonumber \\
    &= \frac{2}{n(n-1)} \left( 2(n-2) \zeta_1 + \zeta_2\right),
\end{align}
where $\zeta_1 := \Var\left[ g_1 (\X)\right]$,
$\zeta_2 := \Var\left[ g (\X_1, \X_2)\right]$,
and $g_1 (\x) := \E[ g(\X_1, \x ) ]$.
Further, note that $\E\left[g_1(\X)\right]
= \E\left[g \left(\X_1, \X_2  \right) \right]
= \E\left[ U_n \right]$.

\medskip

We proceed to evaluate $\zeta_1$ and $\zeta_2$ for the different kernels, and then substitute them into \eqref{eq:varU}. Since the kernels $g^*$, $g^B$, $g^R$ and $g^D$ are all deterministic, this leaves us with the variances of the corresponding estimators.
First we prove items (i)-(iv) of Theorem~\ref{th:var} and then proceed with the proof of (v), where we deal with the randomness of $g^{U}$.

\subsubsection{Proof of (i)}

Under Assumption \ref{as:simp_structural}, we have for every $(j_1,j_2) \in \{1, \dots, b_1\} \times \{ b_1 + 1, \dots, p\}$,
\begin{equation*}
    \E[ g^* \left( \X_{1, (j_1, j_2)} ,
    \X_{2, (j_1, j_2)} \right) ]
    = \tau = 2 P_{j_1 , j_2} - 1 .
\end{equation*}
Also,
\begin{align*}
    g^*_{1}(x_{j_1}, x_{j_2})
    &= \E\Big[2\big( \mathds{1} \left\{
    x_{j_1} < X_{1,j_1} , x_{j_2} < X_{1,j_2}  \right\}
    + \mathds{1} \left\{
    X_{1,j_1} < x_{j_1} , X_{1,j_2} < x_{j_2} \right\}  \big) - 1 \Big]
    = 2 P^{c}(x_{j_1}, x_{j_2}) -1,
\end{align*}
where $P^{c}(x_{j_1}, x_{j_2})$ denotes the probability of concordance of two versions of $(X_{j_1}, X_{j_2})$ given that one pair equals $(x_{j_1},x_{j_2})$. Then,
\begin{align*}
    \zeta_1 &= \Var \left[ g^*_{1} (X_{j_1}, X_{j_2}) \right]
    \\[1em]
    &= \E \left[ (2 P^{c}(X_{j_1}, X_{j_2}) - 1 - \tau )^2   
    \right] \\[1em]
    &= 4 \E\left[ P^{c}(X_{j_1}, X_{j_2}) ^2 \right]
    - 4 (1 + \tau) \E\left[ P^{c}(X_{j_1}, X_{j_2}) \right]
    + (1 + \tau)^2.   
\end{align*}
Note that
$\E\left[ P^{c}(X_{j_1}, X_{j_2}) \right] 
= P_{j_1,j_2}$ and
$\E\left[ P^{c}(X_{j_1}, X_{j_2})^2 \right]
= Q_{j_1,j_2}$.
Furthermore, substitution of $\tau = 2 P_{j_1,j_2} - 1 $ gives us
\begin{equation}
    \label{eq:unz1}
    \zeta_1 = 4(Q_{j_1,j_2} - P_{j_1, j_2}^2).
\end{equation}

For $\zeta_2$ we find
\begin{align*}
    \zeta_2 &= \Var \left[ g^* \left( \X_{1, (j_1, j_2)} ,
    \X_{2, (j_1, j_2)} \right)  \right] \\[1em]
    %
    &=\begin{multlined}[t]
    \E \Big[\big( 2\left( \mathds{1}\left\{
    X_{1, j_1} < X_{2, j_1} , X_{1, j_2} < X_{2, j_2}  \right\} + \mathds{1}\left\{
    X_{2, j_1} < X_{1, j_1} , X_{2, j_2} < X_{1, j_2}
    \right\} \right)
    - 1 - \tau \big)^2 \Big]
    \end{multlined} \\[1em]
    &= 4 \E \left[ \left( \mathds{1}\left\{
    X_{1, j_1} < X_{2, j_1} , X_{1, j_2} < X_{2, j_2}  \right\}
    + \mathds{1}\left\{
    X_{2, j_1} < X_{1, j_1} , X_{2, j_2} < X_{1, j_2}
    \right\} \right)^2  \right] \\
    &- 4 (1 + \tau) \E\left[ \mathds{1}\left\{
    X_{1, j_1} < X_{2, j_1} , X_{1, j_2} < X_{2, j_2}  
    \right\}
    + \mathds{1}\left\{
    X_{2, j_1} < X_{1, j_1} , X_{2, j_2} < X_{1, j_2}
    \right\}  \right] \\
    &+ (1 + \tau)^2.
\end{align*}
Furthermore, note that  
\begin{align*}
    \mathds{1}\left\{
    X_{1,j_1} < X_{2,j_1} , X_{1,j_2} < X_{2,j_2}
    \right\} + \mathds{1}\left\{
    X_{2,j_1} < X_{1,j_1} , X_{2,j_2} < X_{1,j_2}
    \right\} \in \{0,1\},
\end{align*}
and that therefore the expression is equal to its square. We obtain
\begin{align}
    \zeta_2
    &= 4 \E \left[ \left( \mathds{1}\left\{
    X_{1,j_1} < X_{2,j_1} , X_{1,j_2} < X_{2,j_2}
    \right\} + \mathds{1}\left\{
    X_{2,j_1} < X_{1,j_1} , X_{2,j_2} < X_{1,j_2}
    \right\} \right)^2  \right] \nonumber\\
    & - 4 (1 + \tau)
    \E\left[ \mathds{1}\left\{
    X_{1,j_1} < X_{2,j_1} , X_{1,j_2} < X_{2,j_2}
    \right\} + \mathds{1}\left\{
    X_{2,j_1} < X_{1,j_1} , X_{2,j_2} < X_{1,j_2}
    \right\} \right] \nonumber \\
    &+ (1 + \tau)^2
    \nonumber  \displaybreak[0] \\[1em]
    &= - 4\tau \E\left[ \mathds{1}\left\{
    X_{1,j_1} < X_{2,j_1} , X_{1,j_2} < X_{2,j_2}
    \right\} + \mathds{1}\left\{
    X_{2,j_1} < X_{1,j_1} , X_{2,j_2} < X_{1,j_2}
    \right\}  \right] \nonumber \\
    &+ (1 + \tau)^2 \nonumber \\[1em]
    &= - 4 (2 P_{j_1, j_2} - 1) P_{j_1,j_2}
    + (1 + 2P_{j_1, j_2}-1)^2 \nonumber \\[1em]
    &= 4 (P_{j_1,j_2} - P_{j_1, j_2}^2), \label{eq:unz2}
\end{align}
where in the second step we have used that 
\begin{align*}
    \E\left[ \mathds{1}\left\{
    X_{1,j_1} < X_{2,j_1} , X_{1,j_2} < X_{2,j_2}
    \right\} + \mathds{1}\left\{
    X_{2,j_1} < X_{1,j_1} , X_{2, j_2} < X_{1,j_2}
    \right\} \right]
    = P_{k_1,k_2}.
\end{align*}
Substitution of \eqref{eq:unz1} and \eqref{eq:unz2} into (\ref{eq:varU}) gives us the final expression
\begin{align*}
    \Var \left[ \widehat{\tau}_{j_1,j_2} \right]
    = \frac{8}{n(n-1)} \big(2(n-2) \left(
    Q_{j_1, j_2} - P_{j_1, j_2}^2 \right)
    + P_{j_1, j_2} - P_{j_1, j_2}^2 \big).
\end{align*}

\subsubsection{Proof of (ii)}

We have
$$\E[ g^{B}\left( \X_{1}, \X_{2}  \right) ] 
= 2 P^{B,1} - 1, $$
and for any $\x = (x_1, \dots, x_p) \in \Rb^p$,
\begin{align*}
    g_1^{B}(\x) &= \E\left[  \frac{1}{b_1 b_2}
    \sum_{j_1 = 1}^{b_1} \sum_{j_2 = b_1 + 1}^p
    g^* \left( \X_{i_1, (j_1, j_2)} ,
    \x_{(j_1, j_2)} \right) \right]
    \\
    &= \frac{1}{b_1 b_2}
    \sum_{j_1 = 1}^{b_1} \sum_{j_2 = b_1 + 1}^p
    2 P^{c}(x_{j_1}, x_{j_2}) - 1.
\end{align*}
For $\zeta_1$, we then obtain
\begingroup
\allowdisplaybreaks
\begin{align*}
    \zeta_1
    &= \Var \left[ g_1^{B}(\X) \right]
    = \E \left[ \left( \frac{1}{b_1 b_2}
    \sum_{j_1 = 1}^{b_1} \sum_{j_2 = b_1 + 1}^p
    2 P^{c}(X_{j_1}, X_{j_2}) - \tau_{j_1, j_2} - 1
    \right)^2   \right] \\
    &= \frac{1}{b_1^2 b_2^2}
    \sum_{j_1 = 1}^{b_1} \sum_{j_2 = b_1 + 1}^p
    \Bigg( \E \left[ \left(2 P^{c}(X_{j_1},X_{j_2}) - 1 - \tau_{j_1, j_2} \right)^2   \right] \\
    &+ \sum_{\substack{(j_3,j_4)
    \in \{1, \dots, b_1\} \times \{b_1+1, \dots, p\} \\
    \{j_1,j_2\} \cap \{j_3,j_4\} = \emptyset } }
    \E \Big[
    \left(2 P^{c}(X_{j_1},X_{j_2}) - 1 - \tau_{j_1, j_2} \right)
    \left(2 P^{c}(X_{j_3},X_{j_4}) - 1 - \tau_{j_3, j_4} \right)  
    \Big] \\
    &+ \sum_{\substack{(j_3,j_4)
    \in \{1, \dots, b_1\} \times \{b_1+1, \dots, p\} \\
    \lvert \{j_1,j_2\}\cap \{j_3,j_4\} \rvert =1 } }
    \E \Big[
    \left(2 P^{c}(X_{j_1},X_{j_2}) - 1 - \tau_{j_1, j_2} \right) 
    \left(2 P^{c}(X_{j_3},X_{j_4}) - 1 - \tau_{j_3, j_4} \right)
    \Big] \Bigg).
\end{align*}
\endgroup
Now check that for any $j_1, j_2, j_3, j_4 \in \{1, \dots, p\}$, 
\begin{align*}
    \E\left[ P^{c}(X_{j_1},X_{j_2}) P^{c}(X_{j_3},X_{j_4})  
    \right] = Q_{j_1,j_2,j_3,j_4}.
\end{align*}
We then have
\begin{align*}
    \zeta_1
    &= \frac{1}{b_1^2 b_2^2}
    \sum_{j_1 = 1}^{b_1} \sum_{j_2 = b_1 + 1}^p
    \Bigg( 4(Q_{j_1, j_2} - P_{j_1,j_2}^2)
    + \sum_{\substack{(j_3,j_4)
    \in \{1, \dots, b_1\} \times \{b_1+1, \dots, p\} \\
    \{j_1,j_2\} \cap \{j_3,j_4\} = \emptyset } }
    4 (Q_{j_1,j_2,j_3,j_4} - P_{j_1,j_2}P_{j_3,j_4}) \\
    &+ \sum_{\substack{(j_3,j_4)
    \in \{1, \dots, b_1\} \times \{b_1+1, \dots, p\} \\
    \lvert \{j_1,j_2\}\cap \{j_3,j_4\} \rvert =1 } }
    4 (Q_{j_1,j_2,j_3,j_4} - P_{j_1,j_2}P_{j_3,j_4}) \Bigg).
\end{align*}
Therefore, we have
\begin{align}
    \zeta_1
    &= \frac{4}{b_1 b_2} \big( Q_{B,2} - S_{B,2}
    + (b_1 + b_2 - 2) (Q_{B,1} - S_{B,1})
    + (b_1 - 1) (b_2 - 1) (Q_{B,0} - S_{B,0}) \big).
    \label{eq:zeta_1_B}
\end{align}

For $\zeta_2$ we have
\begingroup
\allowdisplaybreaks
\begin{align*}
    \zeta_2
    &= \Var \left[ g^{B} \left( \X_{1}, \X_{2} ) \right)  
    \right] \\[1em]
    &= \E \left[ \left( \frac{1}{b_1 b_2}
    \sum_{j_1 = 1}^{b_1} \sum_{j_2 = b_1 + 1}^p
    g^* \left( \X_{1, (j_1, j_2)} ,
    \X_{2, (j_1, j_2)} \right) - \tau_{j_1,j_2}
    \right)^2 \right] \\[2em]
    &= \frac{1}{b_1^2 b_2^2}
    \sum_{j_1 = 1}^{b_1} \sum_{j_2 = b_1 + 1}^p
    \Bigg( \E \left[ \left(
    g^* \left( \X_{1, (j_1, j_2)} , \X_{2, (j_1, j_2)} \right)
    - \tau_{j_1, j_2} \right)^2   \right] \\
    &+ \sum_{\substack{(j_3,j_4)
    \in \{1, \dots, b_1\} \times \{b_1+1, \dots, p\} \\
    \{j_1,j_2\} \cap \{j_3,j_4\} = \emptyset } }
    \E \Big[ \left(
    g^* \left( \X_{1, (j_1, j_2)} , \X_{2, (j_1, j_2)} \right)
    - \tau_{j_1, j_2} \right)
    \left( g^* \left( \X_{1, (j_3, j_4)} ,
    \X_{2, (j_3, j_4)} \right)
    - \tau_{j_3, j_4} \right)  
    \Big] \\
    &+ \sum_{\substack{(j_3,j_4)
    \in \{1, \dots, b_1\} \times \{b_1+1, \dots, p\} \\
    \lvert \{j_1,j_2\}\cap \{j_3,j_4\} \rvert =1 } }
    \E \Big[ \left(
    g^* \left( \X_{1, (j_1, j_2)} , \X_{2, (j_1, j_2)} \right)
    - \tau_{j_1, j_2} \right)
    \left( g^* \left( \X_{1, (j_3, j_4)} ,
    \X_{2, (j_3, j_4)} \right)
    - \tau_{j_3, j_4} \right)
    \Big].
\end{align*}
\endgroup
Remark that
\begin{align*}
    \E \Big[ 
    g^* \left( \X_{1, (j_1, j_2)} , \X_{2, (j_1, j_2)} \right)
    g^* \left( \X_{1, (j_3, j_4)} , \X_{2, (j_3, j_4)} \right)
    \Big] = 4 P_{j_1, j_2, j_3, j_4}.
\end{align*}
Therefore,
\begingroup
\allowdisplaybreaks
\begin{align*}
    \zeta_2
    &= \frac{1}{b_1^2 b_2^2}
    \sum_{j_1 = 1}^{b_1} \sum_{j_2 = b_1 + 1}^p
    \Bigg( 4(P_{j_1,j_2} - P_{j_1,j_2}^2) 
    + \sum_{\substack{(j_3,j_4)
    \in \{1, \dots, b_1\} \times \{b_1+1, \dots, p\} \\
    \{j_1,j_2\} \cap \{j_3,j_4\} = \emptyset } }
    4(P_{j_1,j_2,j_3,j_4} - P_{j_1,j_2}P_{j_3,j_4}) \\
    &+ \sum_{\substack{(j_3,j_4)
    \in \{1, \dots, b_1\} \times \{b_1+1, \dots, p\} \\
    \lvert \{j_1,j_2\}\cap \{j_3,j_4\} \rvert =1 } }
    4(P_{j_1,j_2,j_3,j_4} - P_{j_1,j_2}P_{j_3,j_4})
    \Bigg).
\end{align*}
\endgroup
and we obtain the expression
\begin{align}
    \zeta_2 &= \frac{4}{b_1 b_2}
    \Big( P_{B,2} - S_{B,2}
    + (b_1 + b_2 - 2) (P_{B,1} - S_{B,1})
    + (b_1 - 1)(b_2 - 1) (P_{B,0} - S_{B,0}) \Big)
    \label{eq:zeta_2_B}
\end{align}
Finally, by substituting \eqref{eq:zeta_1_B}
and \eqref{eq:zeta_2_B} into (\ref{eq:varU}) we find
\begin{align*}
    \Var \left[ \widehat{\tau}^{B} \right]
    = \frac{8}{b_1 b_2 n(n-1)} \Big(
    2(n-2)
    &\big( Q_{B,2} - S_{B,2}
    + (b_1 + b_2 - 2) (Q_{B,1} - S_{B,1})
    + (b_1 - 1) (b_2 - 1) (Q_{B,0} - S_{B,0}) \big) \\
    + &\big( P_{B,2} - S_{B,2}
    + (b_1 + b_2 - 2) (P_{B,1} - S_{B,1})
    + (b_1 - 1)(b_2 - 1) (P_{B,0} - S_{B,0}) \big) \Big).
\end{align*}

\subsubsection{Proof of (iii)}

In a similar manner to the proof of (ii) we obtain 
\begin{align*}
    \zeta_1
    &= \frac{1}{N^2}
    \sum_{j_1=1}^N
    \Bigg( 4(Q_{1, b_1 + j_1} - P_{1, b_1 + j_1}^2)
    + \sum_{j_2=1, \, j_2 \neq j_1}^N
    4 (Q_{1, b_1 + j_1, 1, b_1 + j_2}
    - P_{1, b_1 + j_1} P_{1, b_1 + j_2}) \Bigg).
\end{align*}
Note that there is one less summation compared to the expressions in (ii) since 
$\{1, b_1 + j_1\} \cap \{1, b_1 + j_2\} \neq \emptyset$
for every $j_1, j_2$. Therefore
\begin{align*}
    \zeta_1 = \frac{4}{N}
    \big( Q_{R,2} - S_{R,2}
    + (N - 1) (Q_{R,1} - S_{R,1}) \big).
\end{align*}
Similarly, we find that
\begin{align*}
    \zeta_2 = \frac{4}{N}
    \big( P_{R,2} - S_{R,2}
    + (N - 1) (P_{R,1} - S_{R,1}) \big).
\end{align*} 
Hence,
\begin{align*}
    \Var \left[ \widehat{\tau}^{R} \right]
    = \frac{8}{N n(n-1)} \Big(
    2(n-2) \big( Q_{R,2} - S_{R,2}
    + (N - 1) (Q_{R,1} - S_{R,1}) \big)
    + \big( P_{R,2} - S_{R,2}
    + (N - 1) (P_{R,1} - S_{R,1}) \big) \Big).
\end{align*}

\subsubsection{Proof of (iv)}

Again, in a similar manner to the proof of (ii) we obtain 
\begin{align*}
    \zeta_1
    &= \frac{1}{N^2}
    \sum_{j_1=1}^N
    \Bigg( 4(Q_{j_1, b_1 + j_1} - P_{j_1, b_1 + j_1}^2)
    + \sum_{j_2=1, \, j_2 \neq j_1}^N
    4 (Q_{j_1, b_1 + j_1, j_1, b_1 + j_2}
    - P_{j_1, b_1 + j_1} P_{j_1, b_1 + j_2}) \Bigg).
\end{align*}
Note that there is one less summation compared to the expressions in (ii), as in (iii).
Therefore
\begin{align*}
    \zeta_1 = \frac{4}{N}
    \big( Q_{D,2} - S_{D,2}
    + (N - 1) (Q_{D,1} - S_{D,0}) \big).
\end{align*}
Similarly, we find that
\begin{align*}
    \zeta_2 = \frac{4}{N}
    \big( P_{D,2} - S_{D,2}
    + (N - 1) (P_{D,0} - S_{D,0}) \big).
\end{align*}
Hence,
\begin{align*}
    \Var \left[ \widehat{\tau}^{D} \right]
    = \frac{8}{N n(n-1)} \Big(
    2(n-2) \big( Q_{D,2} - S_{D,2}
    + (N - 1) (Q_{D,0} - S_{D,0}) \big)
    + \big( P_{D,2} - S_{D,2}
    + (N - 1) (P_{D,0} - S_{D,0}) \big) \Big).
\end{align*}

\subsubsection{Proof of (v)}

This proof is very similar as the one of (ii), except that we replace all expectations by conditional expectations given $\W = \w$.
Note that
$\widehat{\tau}_{k_1,k_2}^{U} \lvert \W$ is a U-statistic with (deterministic) kernel $g_{k_1,k_2}^{U}\lvert \W$.
As before we obtain the corresponding $\zeta_1$ and $\zeta_2$,
\begingroup
\allowdisplaybreaks
\begin{align*}
    \zeta_1 &= \frac{1}{b_1^2 b_2^2}
    \sum_{j_1 = 1}^{b_1} \sum_{j_2 = b_1 + 1}^p
    \Bigg( 4 W_{j_1, j_2}^2 (Q_{j_1, j_2} - P_{j_1,j_2}^2)
    + \sum_{\substack{(j_3,j_4)
    \in \{1, \dots, b_1\} \times \{b_1+1, \dots, p\} \\
    \{j_1,j_2\} \cap \{j_3,j_4\} = \emptyset } }
    4 W_{j_1, j_2} W_{j_3, j_4}
    (Q_{j_1,j_2,j_3,j_4} - P_{j_1,j_2}P_{j_3,j_4}) \\
    &+ \sum_{\substack{(j_3,j_4)
    \in \{1, \dots, b_1\} \times \{b_1+1, \dots, p\} \\
    \lvert \{j_1,j_2\}\cap \{j_3,j_4\} \rvert =1 } }
    4 W_{j_1, j_2} W_{j_3, j_4}
    (Q_{j_1,j_2,j_3,j_4} - P_{j_1,j_2}P_{j_3,j_4}) \Bigg),
\end{align*}
\endgroup
and
\begin{align*}
    \zeta_2
    &= \frac{1}{b_1^2 b_2^2}
    \sum_{j_1 = 1}^{b_1} \sum_{j_2 = b_1 + 1}^p
    \Bigg( 4 W_{j_1, j_2}^2 (P_{j_1,j_2} - P_{j_1,j_2}^2) 
    + \sum_{\substack{(j_3,j_4)
    \in \{1, \dots, b_1\} \times \{b_1+1, \dots, p\} \\
    \{j_1,j_2\} \cap \{j_3,j_4\} = \emptyset } }
    4 W_{j_1, j_2} W_{j_3, j_4}
    (P_{j_1,j_2,j_3,j_4} - P_{j_1,j_2}P_{j_3,j_4}) \\
    &+ \sum_{\substack{(j_3,j_4)
    \in \{1, \dots, b_1\} \times \{b_1+1, \dots, p\} \\
    \lvert \{j_1,j_2\}\cap \{j_3,j_4\} \rvert =1 } }
    4 W_{j_1, j_2} W_{j_3, j_4}
    (P_{j_1,j_2,j_3,j_4} - P_{j_1,j_2}P_{j_3,j_4})
    \Bigg),
\end{align*}
which lead to the result as claimed.

\subsubsection{Proof of (vi)}

Lastly, for obtaining the variance of $\widehat{\tau}^{U}$ we need to deal with the random kernel $g^{U}$. To this end, we use the law of total variance and find 
\begin{align*}
    \Var\left[ \widehat{\tau}^{U} \right]
    &= \Var \left[ \mathbb{E} \left[ \widehat{\tau}^{U} \lvert \W\right] \right]
    + \mathbb{E} \left[ \Var \left[\widehat{\tau}^{U} \lvert \W \right] \right] \\[1em]
    &= \Var_{J_1, J_2} \left[ \tau_{J_1, J_2} \right]
    + \mathbb{E} \left[ \Var \left[\widehat{\tau}_{k_1,k_2}^{U} \lvert \W \right] \right] \\[1em]
    &= \Var_{J_1, J_2} \left[ \tau_{J_1, J_2} \right]
    + \mathbb{E} \left[ \Var \left[\widehat{\tau}_{k_1,k_2}^{U} \lvert \W \right] \right].
\end{align*}
We can thus obtain the desired variance by first evaluating the variance of $\tau^{U}$ under a given $\W$ and then by taking the expectation with respect to $\W$.
Therefore, 
\begin{align}
\label{eq:pfitemv}
    &\mathbb{E} \left[ \Var \left[\widehat{\tau}_{k_1,k_2}^{U} \lvert \W \right] \right]
    = \frac{2}{n(n-1)} \left( 2(n-2) \mathbb{E}\left[ \zeta_1 \right] + \E\left[ \zeta_2 \right] \right).
\end{align}
Inserting the expression of $\zeta_1$ and $\zeta_2$ found in (v) into (\ref{eq:pfitemv}), we get
\begingroup \allowdisplaybreaks
\begin{align}
    \mathbb{E} \left[ \Var \left[\widehat{\tau}^{U} \lvert \W \right] \right]
    &= \frac{8}{N^2 n (n - 1)} \Bigg( 2(n-2) \Bigg(
    \frac{1}{b_1^2 b_2^2}
    \sum_{j_1 = 1}^{b_1} \sum_{j_2 = b_1 + 1}^p
    \bigg( 4 \mathbb{E}[W_{j_1, j_2}^2] (Q_{j_1, j_2} - P_{j_1,j_2}^2) 
    \nonumber \\
    &+ \sum_{\substack{(j_3,j_4)
    \in \{1, \dots, b_1\} \times \{b_1+1, \dots, p\} \\
    \{j_1,j_2\} \cap \{j_3,j_4\} = \emptyset } }
    4 \mathbb{E}[W_{j_1, j_2} W_{j_3, j_4}]
    (Q_{j_1,j_2,j_3,j_4} - P_{j_1,j_2}P_{j_3,j_4})
    \nonumber \\
    &+ \sum_{\substack{(j_3,j_4)
    \in \{1, \dots, b_1\} \times \{b_1+1, \dots, p\} \\
    \lvert \{j_1,j_2\}\cap \{j_3,j_4\} \rvert =1 } }
    4 \mathbb{E}[W_{j_1, j_2} W_{j_3, j_4}]
    (Q_{j_1,j_2,j_3,j_4} - P_{j_1,j_2}P_{j_3,j_4}) \bigg)
    \Bigg) \nonumber\\
    &+ 
    \frac{1}{b_1^2 b_2^2}
    \sum_{j_1 = 1}^{b_1} \sum_{j_2 = b_1 + 1}^p
    \bigg( 4 \mathbb{E}[W_{j_1, j_2}^2] (P_{j_1,j_2} - P_{j_1,j_2}^2)
    \nonumber \\
    &+ \sum_{\substack{(j_3,j_4)
    \in \{1, \dots, b_1\} \times \{b_1+1, \dots, p\} \\
    \{j_1,j_2\} \cap \{j_3,j_4\} = \emptyset } }
    4 \mathbb{E}[W_{j_1, j_2} W_{j_3, j_4}]
    (P_{j_1,j_2,j_3,j_4} - P_{j_1,j_2}P_{j_3,j_4})
    \nonumber \\
    &+ \sum_{\substack{(j_3,j_4)
    \in \{1, \dots, b_1\} \times \{b_1+1, \dots, p\} \\
    \lvert \{j_1,j_2\}\cap \{j_3,j_4\} \rvert =1 } }
    4 \mathbb{E}[W_{j_1, j_2} W_{j_3, j_4}]
    (P_{j_1,j_2,j_3,j_4} - P_{j_1,j_2}P_{j_3,j_4})
    \bigg)
    \Bigg) \label{eq:uruvar}
\end{align}
\endgroup
Recall that we select $N$ pairs out of the $b_1 b_2$ possible pairs with uniform probability and without replacement.
Therefore, for every distinct pairs $(j_1,j_2)$ and $(j_3, j_4)$, we can compute the following expectations
\begin{align*}
    \mathbb{E}\left[W_{j_1,j_2}^2\right]
    &= \mathbb{E}\left[W_{j_1,j_2}\right]
    = \frac{N}{b_1 b_2}, \\[1em]
    \mathbb{E}\left[ W_{j_1,j_2}W_{j_3,j_4} \right]
    &= \frac{N(N - 1)}{b_1 b_2 (b_1 b_2 - 1)}.
\end{align*}
Lastly, by combining this with Equation~\eqref{eq:uruvar}, we establish the desired formula
\begin{align*}
    &\mathbb{E} \left[
    \Var \left[\widehat{\tau}^{U} \lvert \W \right] \right]
    = \frac{8}{b_1 b_2 n(n-1)} \Big(
    2(n-2) \\
    &\times \big( Q_{B,2} - S_{B,2}
    + \frac{N - 1}{b_1 b_2 - 1}
    (b_1 + b_2 - 2) (Q_{B,1} - S_{B,1})
    + \frac{N - 1}{b_1 b_2 - 1}
    (b_1 - 1) (b_2 - 1) (Q_{B,0} - S_{B,0}) \big) \\
    & + \big( P_{B,2} - S_{B,2}
    + \frac{N - 1}{b_1 b_2 - 1}
    (b_1 + b_2 - 2) (P_{B,1} - S_{B,1})
    + \frac{N - 1}{b_1 b_2 - 1}
    (b_1 - 1)(b_2 - 1) (P_{B,0} - S_{B,0}) \big) \Big).
\end{align*}

\subsection{Proof of Theorem~\ref{th:varconditional}}
\label{proof:th:varconditional}

\begin{proof}
In \cite{derumigny2019kernel} (see p. 299) it was already shown that the conditional Kendall's tau estimator defined in \eqref{eq:def:est_ckt} is asymptotically normal at different points of the conditioning variable.
However, the asymptotic normalities of the averaging estimators remain to be proven. 
To this end, we follow their approach of studying the joint distribution of U-statistics at several conditioning points,
and give a detailed proof for completeness.
The asymptotic covariance matrices are then obtained by combining the results with the appropriate kernels under Assumption \ref{as:structuralconditional}.

Remember that the averaged conditional Kendall's tau are conditional U-statistics with kernels given by Lemma~\ref{lem:Ustat}, which treats the unconditional case.
Furthermore, for any measurable function $g: \Rb^{2d} \to \Rb$, let us define the second order U-statistic
\begin{align}
\label{eq:defUstat}
    U_{n, j'}(g)
    &:= \frac{1}{n(n-1)}  \sum_{1\leq i_1 \neq i_2 \leq n} g_{i_1,i_2},
\end{align}
where
\begin{align*}
    g_{i_1,i_2}
    := \frac{ g\left(\X_{i_{1}}, \X_{i_{2}}\right) \Kc_h\left(\z'_{j'}-\Z_{i_{1}}\right) \Kc_h\left(\z'_{j'}-\Z_{i_{2}}\right)}{ \E\left[\Kc_h\left(\z'_{j'}-\Z\right)\right]^{2}}
\end{align*}

It follows easily that the averaging estimators can be written in terms of $U_{n,j'}$ by 
\begin{equation}
\label{eq:pfUexp}
    \widehat{\tau}_{\mid \Z=\z'_{j'}}
    = \frac{U_{n,j'}(g)}{U_{n,j'}(1)
    + \epsilon_{n,j'}},
\end{equation}
where we write $\widehat{\tau}_{\mid \Z=\z'_{j'}}$ for any of the estimators
$\widehat{\tau}^{B}_{\mid \Z=\z'_{j'}}$, $\widehat{\tau}^{R}_{\mid \Z=\z'_{j'}}$, $\widehat{\tau}^{D}_{\mid \Z=\z'_{j'}}$, $\widehat{\tau}^{U}_{\mid \Z=\z'_{j'}}$
with $g$ given by, respectively, $g^B$, $g^R$, $g^D$, $g^U$. 
The residual term $\epsilon_{n,j'}$ is given by 
$$\epsilon_{n,j'}:=\frac{\sum_{i=1}^{n} \Kc_h^{2}\left(\z'_{j'}-\Z_{i}\right)}{ n(n- 1) \E\left[\Kc_h(\z'_{j'}-\Z)\right]^{2}} .$$ 

\medskip

Further, we set 
\begin{equation}
\label{eq:pfUtilde}
    \widetilde{\tau}_{\mid \Z=\z'_{j'}}
    := \frac{U_{n,j'}(g)}{U_{n,j'}(1) },
\end{equation}
to be the equivalent of \eqref{eq:pfUexp}
with the term $\epsilon_{n,j'}$ removed.
This is a simpler version of \eqref{eq:pfUexp} that will be easier to analyze theoretically.
We now show that $\widehat{\tau}_{\mid \Z=\z'_{j'}}$
and $\widetilde{\tau}_{\mid \Z=\z'_{j'}}$
are close.
Under Assumption \ref{as:kernel},
we replace both
$\widehat{\tau}_{\mid \Z=\z'_{j'}} = U_{n,j'}(g) / U_{n,j'}(1)$ and
$\widetilde{\tau}_{\mid \Z=\z'_{j'}} = U_{n,j'}(g) / (U_{n,j'}(1) + \epsilon_{n,j'})$ by their expressions to obtain 
\begin{align*}
    \E\left[\frac{1}{\epsilon_{n,j'}}(\widetilde{\tau}_{\mid \Z=\z'_{j'}}-\widehat{\tau}_{\mid \Z=\z'_{j'}}) \right]
    &= \E\left[\frac{1}{\epsilon_{n,j'}}\left( \frac{U_{n,j'}(g)}{U_{n,j'}(1) }
    - \frac{U_{n,j'}(g)}{U_{n,j'}(1) + \epsilon_{n,j'}}     \right)  \right] \\
    &= \E\left[\frac{
    U_{n,j'}(g) (U_{n,j'}(1) + \epsilon_{n,j'})
    - U_{n,j'}(g) U_{n,j'}(1)
    }{\epsilon_{n,j'} U_{n,j'}(1)
    (U_{n,j'}(1) + \epsilon_{n,j'})} \right] \\
    &= \E\left[\frac{
    U_{n,j'}(g) U_{n,j'}(1)
    + U_{n,j'}(g) \epsilon_{n,j'}
    - U_{n,j'}(g) U_{n,j'}(1)
    }{\epsilon_{n,j'} U_{n,j'}(1)
    (U_{n,j'}(1) + \epsilon_{n,j'})} \right] \\
    &= \E\left[\frac{
    U_{n,j'}(g) \epsilon_{n,j'}
    }{\epsilon_{n,j'} U_{n,j'}(1)
    (U_{n,j'}(1) + \epsilon_{n,j'})} \right] \\
    &= \E\left[\frac{1}{U_{n,j'}(1)}
    \frac{U_{n,j'}(g)}{U_{n,j'}(1) + \epsilon_{n,j'}}   \right]\\
    &= \E\left[\frac{1}{U_{n,j'}(1)}
    \widehat{\tau}_{ \mid \Z=\z'_{j'}} \right]\\
    &= O(1),
\end{align*}
because $\big|\widehat{\tau}_{ \mid \Z=\z'_{j'}}\big|$ is bounded by $1$ and $U_{n,j'}(1)$ is asymptotically normal, hence convergent by Lemma 17 of \cite{derumigny2019kernel}.
Therefore, $ \widehat{\tau}_{ \mid \Z=\z'_{j'}}-\widetilde{\tau}_{ \mid \Z=\z'_{j'}}  = O_P(\epsilon_{n,j'})$ using Markov's inequality.
By Assumption \ref{as:kernel}(c) and by Bochner's lemma (see e.g.~\cite{tsybakov2003introduction}) we see that $\epsilon_{n,j'} = O_P((nh^d)^{-1})$. It then follows by Assumption~\ref{as:bandwidth} that $$(nh^d)^{1/2}(\widehat{\tau}_{ \mid \Z=\z'_{j'}}-\widetilde{\tau}_{ \mid \Z=\z'_{j'}}) = O_P((nh^d)^{1/2} \epsilon_{n,j'}) = o_P(1).$$ 
It therefore suffices to obtain the limiting law of $ (nh^d)^{1/2}  (\widetilde{\tau}_{ \mid \Z=\z'_{j'}}-\tau_{ \mid \Z=\z'_{j'}})$ as $n\to \infty$. 

Now let us apply Lemma 17 again from \cite{derumigny2019kernel} on the joint asymptotic law of U-statistics of the form $U_{n, j'}$. That is, under Assumptions \ref{as:kernel}-\ref{as:bandwidth} and for any two bounded measurable functions $g_1$ and $g_2$
\begin{align}
\left(n h^{d}\right)^{1 / 2}\bigg(\Big(U_{n, j'}\left(g_{1}\right)&-\E\left[U_{n, j'}\left(g_{1}\right)\right]\Big)_{j'=1, \ldots, n^{\prime}},\Big(U_{n, j'}\left(g_{2}\right)-\E\left[U_{n, j'}\left(g_{2}\right)\right]\Big)_{j'=1, \ldots, n^{\prime}}\bigg) \nonumber \\
& \stackrel{\textnormal{law}}{\longrightarrow}  \mathcal{N}\left(0,\left[\begin{array}{cc}
M_{\infty}\left(g_{1}\right) & M_{\infty}\left(g_{1}, g_{2}\right) \\
M_{\infty}\left(g_{1}, g_{2}\right) & M_{\infty}\left(g_{2}\right)
\end{array}\right]\right), \label{eq:Ulimit}
\end{align}
as $n \rightarrow \infty$, where
\begin{align}
    \left[M_{\infty}\left(g_{1}, g_{2}\right)\right]_{j_1', j_2'}:=\frac{4 \int \Kc^{2} \mathds{1}_{\{\z_{j_1'}^{\prime}=\z_{j_2'}^{\prime}\}}}{f_{\Z}\left(\z_{j_1'}^{\prime}\right)} &\int g_{1}\left(\x_{1}, \x\right) g_{2}\left(\x_{2}, \x\right)  \\
    &\times f_{\X \mid \Z=\z_{j_1'}^{\prime}}(\x) f_{\X \mid \Z=\z_{j_1'}^{\prime}}\left(\x_{1}\right) f_{\X \mid \Z=\z_{j_1'}^{\prime}}\left(\x_{2}\right) \mathrm{d} \x \mathrm{d} \x_{1} \mathrm{d} \x_{2} . \nonumber
\end{align}
Let us investigate the expectation of  $U_{n,j'} (g)$. We write by \eqref{eq:defUstat} 
\begin{equation*}
    \E\left[ U_{n,j'} (g)  \right] = \frac{1 }{ \E\left[\Kc_h(\z'_{j'}-\Z)\right]^{2} } \E\Big[g\left(\X_{1}, \X_{2}\right) \Kc_h\left(\z'_{j'}-\Z_{1}\right) \Kc_h\left(\z'_{j'}-\Z_{2}\right)\Big].
\end{equation*}
Further, by a change of variable we find
\begin{align}
\label{eq:denominatorUstat}
    &\E\Big[g \left(\X_{1}, \X_{2}\right) \Kc_h\left(\z'_{j'}-\Z_{1}\right) \Kc_h\left(\z'_{j'}-\Z_{2}\right) \Big] \nonumber \\
    &= \int g \left(\x_{1}, \x_{2}\right) \Kc_h\left(\z'_{j'}-\Z_{1}\right) \Kc_h\left(\z'_{j'}-\Z_{2}\right)
    f_{\X, \Z} (\mb{x_1}, \Z_1)
    f_{\X, \Z} (\mb{x_2}, \Z_2)
    \mathrm{d}\x_1 \mathrm{d}\x_2 \mathrm{d}\Z_1 \mathrm{d}\Z_2 \nonumber \\
    &= \int g \left(\x_{1}, \x_{2}\right)
    \Kc\left(\mb{u}_1\right)
    \Kc\left(\mb{u}_2\right)
    f_{\X, \Z} (\mb{x_1}, \z'_{j'} + h u_1)
    f_{\X, \Z} (\mb{x_2}, \z'_{j'} + h u_2)
    \mathrm{d}\x_1 \mathrm{d}\x_2 \mathrm{d}\mb{u}_1 \mathrm{d}\mb{u}_2,
\end{align}
with the change of variable
$\Z_1 = \z'_{j'} + h u_1$
and $\Z_2 = \z'_{j'} + h u_2$.
Let us define the function $\phi_{\x_{1}, \x_{2}, \mb{u}_1, \mb{u}_2}(t):=f_{\X, \Z}\left(\x_{1}, \z'_{j'} +t h \mb{u}_1\right) f_{\X, \Z}\left(\x_{2}, \z'_{j'} +t h \mb{u}_2\right)$ for $t \in [0,1]$. By Assumption \ref{as:conditionalvariate}, $\phi_{\x_{1}, \x_{2}, \mb{u}_1, \mb{u}_2}(t)$ is $\alpha$ times differentiable, allowing us to apply the Taylor-Lagrange formula. This gives
\begin{equation}
\label{eq:phi}
    \phi_{\x_{1}, \x_{2}, \mb{u}_1, \mb{u}_2}(t) = \sum_{k=0}^{\alpha-1} \frac{1}{k!}  \phi_{\x_{1}, \x_{2}, \mb{u}_1, \mb{u}_2}^{(k)}(0) + \frac{1}{\alpha!}  \phi_{\x_{1}, \x_{2}, \mb{u}_1, \mb{u}_2}^{(\alpha)}(t^*),
\end{equation}
for some $t^* \in [0,1]$ and where $\phi_{\x_{1}, \x_{2}, \mb{u}_1, \mb{u}_2}^k(t)$ is equal to
\begin{align*}
     \sum_{l=0}^{k} \binom{k}{l} \sum_{j_{1}, \ldots, j_{k}=1}^{d} &h^{k} u_{j_{1},1} \ldots u_{j_{l},1} u_{j_{l+1},2} \ldots u_{j_{k},2}\\
    & \frac{\partial^{k} f_{\X, \Z}}{\partial z_{j_{1}} \ldots \partial z_{j_{l}}}\left(\x_{1}, \z'_{j'}+t h \mb{u}_1\right) \frac{\partial^{k-l} f_{\X, \Z}}{\partial z_{j_{l+1}} \ldots \partial z_{j_{k}}}\left(\x_{2}, \z'_{j'} + t h \mb{u}_2\right).
\end{align*}

After substituting \eqref{eq:phi} into \eqref{eq:denominatorUstat} we obtain
\begin{align*}
    &\int g \left(\x_{1}, \x_{2}\right)
    \Kc\left(\mb{u}_1\right)
    \Kc\left(\mb{u}_2\right)
    \left(\sum_{k=0}^{\alpha-1} \frac{1}{k!}  \phi_{\x_{1}, \x_{2}, \mb{u}_1, \mb{u}_2}^{(k)}(0) + \frac{1}{\alpha!}  \phi_{\x_{1}, \x_{2}, \mb{u}_1, \mb{u}_2}^{(\alpha)}(t^*) \right)
    \mathrm{d}\x_1 \mathrm{d}\x_2 \mathrm{d}\mb{u}_1 \mathrm{d}\mb{u}_2\\
    &= \int g \left(\x_{1}, \x_{2}\right)
    \Kc\left(\mb{u}_1\right)
    \Kc\left(\mb{u}_2\right)
    \left(\phi_{\x_{1}, \x_{2}, \mb{u}_1, \mb{u}_2}(0) + \frac{1}{\alpha!}  \phi_{\x_{1}, \x_{2}, \mb{u}_1, \mb{u}_2}^{(\alpha)}(t^*) \right)
    \mathrm{d}\x_1 \mathrm{d}\x_2 \mathrm{d}\mb{u}_1 \mathrm{d}\mb{u}_2\\
    &=  \int g \left(\x_{1}, \x_{2}\right)
    \Kc\left(\mb{u}_1\right)
    \Kc\left(\mb{u}_2\right)
    f_{\X, \Z} (\mb{x_1}, \z'_{j'})
    f_{\X, \Z} (\mb{x_2}, \z'_{j'})
    \mathrm{d}\x_1 \mathrm{d}\x_2 \mathrm{d}\mb{u}_1 \mathrm{d}\mb{u}_2\\
     &+ \frac{1}{\alpha!} \int g \left(\x_{1}, \x_{2}\right)
    \Kc\left(\mb{u}_1\right)
    \Kc\left(\mb{u}_2\right)
      \phi_{\x_{1}, \x_{2}, \mb{u}_1, \mb{u}_2}^{(\alpha)}(t^*) 
    \mathrm{d}\x_1 \mathrm{d}\x_2 \mathrm{d}\mb{u}_1 \mathrm{d}\mb{u}_2\\
    &= f_{\Z}^2( \z'_{j'} ) \E\Big[g\left(\X_{1}, \X_{2}\right) \mid  \Z_{1}=\Z_{2}=\z'_{j'}\Big] \\
    &+ \frac{1}{\alpha!} \int g \left(\x_{1}, \x_{2}\right)
    \Kc\left(\mb{u}_1\right)
    \Kc\left(\mb{u}_2\right)
      \phi_{\x_{1}, \x_{2}, \mb{u}_1, \mb{u}_2}^{(\alpha)}(t^*) 
    \mathrm{d}\x_1 \mathrm{d}\x_2 \mathrm{d}\mb{u}_1 \mathrm{d}\mb{u}_2,
\end{align*}
where in the first equality we have used the fact that $\int \Kc(\mb{u}) u_{j_{1}} \ldots u_{j_{k}} d \mb{u}=0$ for all $k=1, \ldots, \alpha-1$ as stated in Assumption \ref{as:kernel}(b), which results in the elimination of all the terms in the sum except the first one.
In the second equality, we replaced $\phi$ and $\phi^{(\alpha)}$ by their expressions.
In the last equality, we factor out $f_{\Z}^2( \z'_{j'} )$ and recognize that the corresponding integral is a conditional expectation.

We now bound the second term of the previous display.
By Assumption \ref{as:conditionalvariate} we have
\begin{align*}
    & \frac{1}{\alpha!}  \bigg|  \int g \left(\x_{1}, \x_{2}\right)
    \Kc\left(\mb{u}_1\right)
    \Kc\left(\mb{u}_2\right)
      \phi_{\x_{1}, \x_{2}, \mb{u}_1, \mb{u}_2}^{(\alpha)}(t^*) 
    \mathrm{d}\x_1 \mathrm{d}\x_2 \mathrm{d}\mb{u}_1 \mathrm{d}\mb{u}_2 \bigg| 
    \\& \leq \int 
    \lvert \Kc \rvert \left(\mb{u}_1\right)
    \lvert \Kc\rvert \left(\mb{u}_2\right)
     \lvert \phi_{\x_{1}, \x_{2}, \mb{u}_1, \mb{u}_2}^{(\alpha)}\rvert (t^*) 
    \mathrm{d}\x_1 \mathrm{d}\x_2 \mathrm{d}\mb{u}_1 \mathrm{d}\mb{u}_2   \\
    &\leq C_{\X, \Z} h^\alpha. 
\end{align*}
Therefore, 
$$ \E\Big[g\left(\X_{1}, \X_{2}\right) \Kc_h\left(\z'_{j'}-\Z_{1}\right) \Kc_h\left(\z'_{j'}-\Z_{2}\right)\Big]  =  f_{\Z}^2(\Z'_{j})  \E\Big[g\left(\X_{1}, \X_{2}\right) \mid  \Z_{1}=\Z_{2}=\z'_{j'}\Big]  + O( h^\alpha),  $$
and by the same reasoning we obtain
\begin{align*}
 \E\left[\Kc_h(\z'_{j'}-\Z)\right]  =  f_{\Z}^2(\Z'_{j})  + O( h^\alpha). 
\end{align*}
Consequently, we find that
\begin{align}
    \E\left[ U_{n,j'} (g)  \right] &= \frac{1 }{ \E\left[\Kc_h(\z'_{j'}-\Z)\right]^{2} } \E\Big[g\left(\X_{1}, \X_{2}\right) \Kc_h\left(\z'_{j'}-\Z_{1}\right) \Kc_h\left(\z'_{j'}-\Z_{2}\right)\Big] \nonumber \\
    &= \E\Big[g\left(\X_{1}, \X_{2}\right) \mid  \Z_{1}=\Z_{2}=\z'_{j'}\Big] + r_{n,j'},
\end{align}
where $\lvert r_{n,j'}\rvert \leq C_0h^{\alpha}$ for some constant $C_0$ independent of $j'$. Then, by Assumption \ref{as:bandwidth}
\begin{align}
    (nh^d)^{1/2}\left(\E[U_{n,j'}(g)] -  \E[g(\X_1,\X_2) \mid \Z_1 = \Z_2=\z'_{j'}]\right) = O((nh^d)^{1/2} h^\alpha) = o(1).
    \label{eq:control_bias}
\end{align}
Therefore, the asymptotic law of \eqref{eq:Ulimit} still holds after replacing $\E[U_{n,j'}(g)]$ with $\E[g(\X_1,\X_2) \mid \Z_1 = \Z_2=\z'_{j'}]$. As such,
\begin{align}
    \left(n h^{d}\right)^{1 / 2}\bigg(\Big(U_{n, j'}\left(g\right)&-\tau_{\mid \Z=\z'_{j'}}\Big)_{j'=1, \ldots, n^{\prime}},\Big(U_{n, j'}\left(1\right)-1\Big)_{j'=1, \ldots, n^{\prime}}\bigg) \nonumber \\
    &\stackrel{\textnormal{law}}{\longrightarrow} \mathcal{N}\left(0,\left[\begin{array}{cc}
    M_{\infty}\left(g,g\right) & M_{\infty}\left(g, 1\right) \\
    M_{\infty}\left(g, 1\right) & M_{\infty}\left(1,1\right)
    \end{array}\right]\right), \label{eq:Ulimit2}
\end{align}
as $n \to \infty$, where we have used that $\E[g(\X_1,\X_2) \mid \Z_1 = \Z_2=\z'_{j'}]
= \tau_{\mid \Z=\z'_{j'}}$ under Assumption~\ref{as:structuralconditional}. 

\medskip

In order to derive the asymptotic law of $ (nh^d)^{1/2}  (\widetilde{\tau}_{\mid \Z=\z'_{j'}}-\tau_{ \mid \Z=\z'_{j'}}) $
we apply the Delta-method on \eqref{eq:Ulimit2} with the function $\gamma(\x,\mb{y}) :=\x/\mb{y}$, that divides two real vectors $\x,\mb{y}$ of size $n'$ component-wise. The corresponding Jacobian is given by the $n'\times 2n'$ matrix 
$$
J_{\gamma}(\x, \mb{y})=\left[\operatorname{Diag}\left(y_{1}^{-1}, \ldots,  y_{n^{\prime}}^{-1}\right), \operatorname{Diag}\left(-x_{1} y_{1}^{-2}, \ldots, -x_{n^{\prime}} y_{n^{\prime}}^{-2}\right)\right].
$$
Hence, as $n \to \infty$
$$
\left(n h_{n}^{d}\right)^{1 / 2}\left(\widehat{\tau}_{\mid \Z=\z'_{j'}}-\tau_{\mid \Z=\z'_{j'}}\right)_{j'=1, \ldots, n^{\prime}} \stackrel{\textnormal{law}}{\longrightarrow}  \mathcal{N}\left(0, \mb{H}\right),
$$
setting
$$
\mb{H}:=J_{\gamma}(\vec{\tau}_{|\Z}, \mb{e})\left[\begin{array}{cc}
M_{\infty}\left(g\right) & M_{\infty}\left(g, 1\right) \\
M_{\infty}\left(g, 1\right) & M_{\infty}(1)
\end{array}\right] J_{\gamma}(\vec{\tau}_{|\Z}, \mb{e})^{T},
$$
where $\vec{\tau}_{|\Z}$ and $\mb{e}$ denote $n'$-dimensional vectors filled with respectively $\tau_{\mid \Z=\z'_{j'}}$ and $1$. This gives
\begin{align*}
    \mb{H}&=M_{\infty}\left(g,g\right)-\operatorname{Diag}(\vec{\tau}_{|\Z}) M_{\infty}\left(g, 1\right)-M_{\infty}\left(g, 1\right) \operatorname{Diag}(\vec{\tau}_{|\Z})\\
    &+\operatorname{Diag}(\vec{\tau}_{|\Z}) M_{\infty}(1,1) \operatorname{Diag}(\vec{\tau}_{|\Z})
\end{align*}
and for $1\leq j_1', j_2'\leq n'$, we find
\begin{align*}
     \left[M_{\infty}\left(g, g\right)\right]_{j_1', j_2'}&=\frac{4 \int \Kc^{2} \mathds{1}_{\{\z_{j_1'}^{\prime}=\z_{j_2'}^{\prime}\}}}{f_{\Z}\left(\z_{j_1'}^{\prime}\right)} \mathbb{E}\left[g\left(\X_{1}, \X\right) g\left(\X_{2}, \X\right) \mid \Z=\z_{1}=\Z_{2}=\z_{j_1'}^{\prime}\right],\\
    \left[\operatorname{Diag}(\vec{\tau}_{|\Z}) M_{\infty}\left(g, 1\right)\right]_{j_1', j_2'}&=\frac{4 \int \Kc^{2} \mathds{1}_{\{\z_{j_1'}^{\prime}=\z_{j_2'}^{\prime}\}}}{f_{\Z}\left(\z_{j_1'}^{\prime}\right)}  \tau_{|\Z=\z'_{j'}} \mathbb{E}\left[g\left(\X_{1}, \X\right) \mid \Z=\z_{1}=\z_{j_1'}^{\prime}\right]  \\
    &= \frac{4 \int \Kc^{2} \mathds{1}_{\{\z_{j_1'}^{\prime}=\z_{j_2'}^{\prime}\}}}{f_{\Z}\left(\z_{j_1'}^{\prime}\right)} \tau_{|\Z=\z'_{j'}}^2  \\
    &= \left[M_{\infty}\left(g, 1\right) \operatorname{Diag}(\vec{\tau}_{|\Z})\right]_{j_1',j_2'}\\ &=\left[\operatorname{Diag}(\vec{\tau}_{|\Z}) M_{\infty}\left(1, 1\right) \operatorname{Diag}(\vec{\tau}_{|\Z})\right]_{j_1',j_2'},
\end{align*}
and thus,
\begin{align*}
    [\mb{H}]_{j_1', j_2'}&=\frac{4 \int \Kc^{2} \mathds{1}_{\left\{\z_{j_1'}^{\prime}=\z_{j_2'}^{\prime}\right\}}}{f_{\Z}\left(\z_{j_1'}^{\prime}\right)} \left(\mathbb{E}\left[g\left(\X_{1}, \X_2\right) g\left(\X_{1}, \X_3\right) \mid \Z_{1}=\Z_{2}=\Z_{3}=\z_{j_1'}^{\prime}\right]-\tau_{\mid \Z=\z'_{j'}}^{2}\right).
\end{align*}
Lastly, by substituting the appropriate kernels and by similar steps as in the derivation of the corresponding $\zeta_1$'s from the proof of Theorem \ref{th:var}, it is easily seen that under Assumption \ref{as:structuralconditional} we obtain the desired asymptotic covariance matrices. 

\end{proof}

\section{List of Stocks Used}

\label{sec:stocksLists}

\textbf{Group 1}
\begin{multicols}{2}
  \begin{enumerate}
    \item Cafom (CAFO)  \item Techstep (TECH) \item Gold by Gold (ALGLD) \item Fonciere Inea (INEA) \item NSC Groupe (ALNSC) \item Hofseth BioCare (HBC)
    \item GC Rieber Shipping (RISH) \item Aega (AEGA) \item i2S (ALI2S) \item Moury Construct (MOUR) \item Gascogne (ALBI) \item Thunderbird (TBIRD) 
    \item Hydratec Industries (HYDRA) \item Sparebank 1 Ostfold Akershus (SOAG) \item Altareit (AREIT) \item Unibel (UNBL) \item Cheops Technology France (MLCHE) \item Zenobe Gramme Cert (ZEN) 
    \item Indel Fin (INFE) \item Artois Nom (ARTO) \item IDS (MLIDS) \item Musee Grevin (GREV) \item Robertet (CBE) \item Aurskog Sparebank (AURG)
    \item Alliance Developpement Capital (ALDV) \item Fonciere Atland (FATL) \item FREYR Battery (FREY) \item Hotels de Paris (HDP) \item Phone Web (MLPHW) \item Maroc Telecom (IAM) 
    \item Sporting (SCP) \item MG International (ALMGI) \item Ucar (ALUCR) \item Cumulex (CLEX) \item Televerbier (TVRB) \item Alan Allman Associates (AAA) 
    \item Serma Group (ALSER) \item Planet Media (ALPLA) \item Philly Shipyard (PHLY) \item Augros Cosmetic Packaging (AUGR) \item Sequa Petroleum (MLSEQ) \item EMOVA Group (ALEMV) 
    \item Streamwide (ALSTW) \item Accentis (ACCB) \item Smalto (MLSML) \item Signaux Girod (ALGIR)
  \end{enumerate}
\end{multicols}

\noindent \textbf{Group 2}
\begin{multicols}{2}
  \begin{enumerate}
    \setcounter{enumi}{46}
    \item Interoil (IOX)  \item Ensurge Micropower (ENSU) \item Idex Biometrics (IDEX) \item SD Standard Drilling (SDSD) \item SpareBank 1 Nord-Norge (NONG) \item Bonheur (BONHR)
    \item Eidesvik Offshore (EIOF) \item DOF (DOF) \item Solstad Offshore (SOFF) \item Havila Shipping (HAVI) \item Awilco LNG (ALNG) \item FLEX LNG (FLNG) 
    \item Avance Gas Holding (AGAS) \item Hunter Group (HUNT) \item Itera (ITERA) \item Q-Free (QFR) \item Photocure (PHO) \item PCI Biotech Holding (PCIB)
    \item Hexagon Composites (HEX) \item Nel (NEL) \item McPhy Energy (MCPHY) \item Vow (VOW) \item Axactor (ACR) \item Magseis Fair Fairfield (MSEIS)
    \item NRC Group (NRC) \item Petrolia (PSE) \item Ctac (CTAC) \item StrongPoint (STRO) \item Crescent (OPTI) \item Magnora (MGN)
    \item Rec Silicon (RECSI) \item Questerre Energy Corp (QEC) \item ElectroMagnetic GeoServices (EMGS) \item ABG Sundal Collier Holding (ABG) \item Nekkar (NKR) \item SeaBird Exploration (GEG) 
    \item Wilh. Wilhelmsen Holding (WWI) \item Golden Ocean Group (GOGL) \item Frontline (FRO) \item Euronav (EURN) \item Norwegian Air Shuttle (NAS) \item Atea (ATEA)
    \item Vopak (VPK) \item Orkla (ORK) \item Corbion (CRBN) \item Otello Corporation (OTEC) 
  \end{enumerate}
\end{multicols}

\noindent \textbf{Group 3}
\begin{multicols}{2}
  \begin{enumerate}
    \setcounter{enumi}{92}
    \item Aures Technologies (AURS) \item Keyrus (ALKEY) \item Nextedia (ALNXT) \item Cabasse Group (ALCG) \item Groupe Guillin (ALGIL) \item Guillemot (GUI)
    \item Solutions 30 (S30) \item Esker (ALESK) \item Wavestone (WAVE) \item Groupe Open (OPN) \item Envea (ALTEV) \item Stern Groep (STRN)
    \item IT Link (ALITL)    \item Lectra (LSS) \item Groupe CRIT (CEN)   \item Aubay (AUB)   \item Sword Group (SWP) \item NRJ Group (NRG)
    \item Van de Velde (VAN) \item Hunter Douglas (HDG) \item Oeneo (SBT) \item Axway Software (AXW) \item SES-imagotag (SESL) \item Ateme (ATEME)
    \item Infotel (INF) \item Sergeferrari Group (SEFER) \item Umanis (ALUMS) \item Corticeira Amorim (COR) \item  Pharmagest Interactive (PHA) \item Asetek (ASTK)
    \item  ID Logistics Group (IDL) \item Scana (SCANA) \item Acheter Louer fr (ALOLO) \item Adomos (ALADO) \item Glintt (GLINT) \item Inapa (INA)
    \item  Cegedim (CGM) \item Lavide Holding (LVIDE) \item TIE Kinetix (TIE) \item Alumexx (ALX) \item Ober (ALOBR) \item Cibox Interactive (CIB)
    \item Evolis (ALTVO) \item Proactis (PROAC) \item Visiodent (SDT) \item Fashion B Air (ALFBA) \item Adthink (ALADM) \item Innelec Multimedia (ALINN)
    \item Herige (ALHRG) \item Egide (GID) \item U10 Corp (ALU10) \item Mr. Bricolage (ALMRB) \item Coheris (COH) \item Pcas (PCA)
    \item Rosier (ENGB) \item Itesoft (ITE) \item Gea Grenobl.Elect. (GEA) \item Immobel (IMMO) \item IGE+XAO Group (IGE) \item Koninklijke Brill (BRILL) 
    \item Argan (ARG) \item Fonciere Lyonnaise (FLY) \item Covivio Hotels (COVH) \item Electricite de Strasbourg (ELEC) \item Robertet (RBT) \item Norway Royal Salmon (NRS) \item Olympique Lyonnais Groupe (OLG) \item GeoJunxion (GOJXN) \item Hybrid Software Group (HYSG) \item Cast (CAS) \item Acteos (EOS) \item HF Company (ALHF) 
    \item Vranken-Pommery Monopole (VRAP) \item Generix Group (GENX) \item  Union Technologies Infor. (FPG) \item Diagnostic Medical Systems (DGM) \item Capelli (CAPLI) \item EXEL Industries (EXE) 
    \item Groupe LDLC (ALLDL) \item genOway (ALGEN) \item CBo Territoria (CBOT) \item  Aurea (AURE) \item EO2 (ALEO2) \item RAK Petroleum (RAKP)
  \end{enumerate}
\end{multicols}

\noindent \textbf{Group 4}
\begin{multicols}{2}
  \begin{enumerate}
    \setcounter{enumi}{176}
    \item DNO (DNO) \item Archer Limited (ARCH) \item Odfjell Drilling (ODL) \item BW Offshore (BWO) \item Panoro Energy (PEN) \item PGS (PGS) 
    \item TGS (TGS) \item Subsea 7 (SUBC) \item Equinor (EQNR) \item Aker BP (AKRBP) \item Aker (AKER) \item Aker Solutions (AKSO) 
    \item Akastor (AKAST) \item Etablissements Maurel et Prom (MAU)
    \item Vallourec (VK) \item CGG (CGG) \item TechnipFMC (FTI) \item Fugro (FUR) 
    \item SBM Offshore (SBMO) \item Galp Energia (GALP) \item TotalEnergies (TTE) \item Royal Dutch Shell B (RDSB) \item Schlumberger Limited (LSD) \item Alten (ATE) 
    \item Capgemini (CAP) \item Atos (ATO) \item STMicroelectronics (STM) \item ASML Holding (ASML) \item ASM International (ASM) \item BE Semiconductors (BESI) 
    \item Banco Comercial Portugues (BCP)
    \item Mota-Engil (EGL) \item Altri (ALTR) \item The Navigator Company (NVG) \item Semapa (SEM) \item Trigano (TRI) 
    \item Ipsos (IPS) \item Barco (BAR) \item TomTom (TOM2) \item PostNL (PNL) \item TF1 (TFI) \item Derichebourg (DBG)
    \item Heijmans (HEIJM) \item Koninklijke BAM Groep (BAMNB) \item Aegon (AGN) \item AXA (CS) \item Agaas (AGS) \item Bouygues (EN) 
    \item VINCI (DG) \item Eiffage (FGR) \item Faurecia (EO) \item Valeo (FR) \item Michelin (ML) \item Ackermans \& Van Haaren (ACKB) 
    \item Royal Boskalis Westminster (BOKA) \item Imerys (NK) \item Solvay (SOLB) \item Umicore (UMI) \item AkzoNobel (AKZA) \item Air Liquide (AI) 
    \item Koninklijke Philips (PHIA) \item Aperam (APAM) \item Eramet (ERA) \item Norsk Hydro (NHY) 
  \end{enumerate}
\end{multicols}

\end{document}